\documentclass[12pt,a4]{amsart}
 \usepackage[a4paper, left=28mm, right=28mm, top=26mm, bottom=31mm]{geometry}
 \usepackage{amsthm,amssymb}
 \usepackage[all]{xy}
 \usepackage{amsmath}
 \usepackage{textcomp}
 \usepackage{cases}
 \usepackage{here}
 \usepackage[dvipdfmx]{graphicx}
 \usepackage{amscd}
 \usepackage{mathrsfs}
 \usepackage{float}
 \usepackage{wrapfig}
 \usepackage{color}
 
\newtheorem{thm}{Theorem}[section]
 \newtheorem{dfn}[thm]{Definition}
 \newtheorem{lem}[thm]{Lemma}
 \newtheorem{prp}[thm]{Proposition}
 \newtheorem{rem}[thm]{Remark}
 
 \newtheorem{cor}[thm]{Corollary}

 \newtheorem{exam}[thm]{Example}
 \def\P{\mathbb{P}}
 \def\A{\mathbb{A}}
 \def\G{\mathbb{G}}
 \def\R{\mathbb{R}}
 \def\Z{\mathbb{Z}}
 \def\Q{\mathbb{Q}}
 \def\O{\mathcal{O}}
 
 \def\J{\mathcal{J}}
 \def\L{\mathcal{L}}
 \def\M{\mathscr{M}}
 \def\H{\mathscr{H}}
 \def\K{\mathscr{K}}
 \def\F{\mathscr{F}}
 \def\C{\mathscr{C}}
 \def\CC{\mathbb{C}}
 
 \def\T{\mathscr{T}}
 \def\X{\mathscr{X}}

 \def\Cent{\mathop{\mathrm{cent}}\nolimits}

 \def\Val{\mathop{\mathrm{Val}}\nolimits}
 \def\mpd{\mathop{\mathrm{mpd}}\nolimits}

 \def\cores{\mathop{\mathrm{cor}}\nolimits}
 \def\Tan{\mathop{\mathrm{Tan}}\nolimits}
 
 \def\Zar{\mathop{\mathrm{Zar}}\nolimits}

 \def\pr{\mathop{\mathrm{pr}}\nolimits}
 
 \def\CH{\mathop{\mathrm{CH}}\nolimits}
 
 \def\Gr{\mathop{\mathrm{Gr}}\nolimits}
 \def\trdeg{\mathop{\mathrm{tr.deg}}\nolimits}
 \def\res{\mathop{\mathrm{res}}\nolimits}
 \def\Im{\mathop{\mathrm{Im}}\nolimits}
 \def\Ker{\mathop{\mathrm{Ker}}\nolimits}
 \def\new{\mathop{\mathrm{new}}\nolimits}
 \def\sing{\mathop{\mathrm{sing}}\nolimits}
 
 \def\Sk{\mathop{\mathrm{Sk}}\nolimits}
 
 \def\ad{\mathop{\mathrm{ad}}\nolimits}
 \def\Trop{\mathop{\mathrm{trop}}\nolimits}
 \def\Ber{\mathop{\mathrm{Ber}}\nolimits}
 \def\relint{\mathop{\mathrm{rel.int}}\nolimits}
 \def\height{\mathop{\mathrm{ht}}\nolimits}
 
 \def\Span{\mathop{\mathrm{Span}}\nolimits}
 \def\rank{\mathop{\mathrm{rank}}\nolimits}
 
 \def\Frac{\mathop{\mathrm{Frac}}\nolimits}
 \def\supp{\mathop{\mathrm{supp}}\nolimits}
 \def\ZR{\mathop{\mathrm{ZR}}\nolimits}
 \def\Spv{\mathop{\mathrm{Spv}}\nolimits}
 \def\Spa{\mathop{\mathrm{Spa}}\nolimits}

 \def\Spec{\mathop{\mathrm{Spec}}\nolimits}

 \def\Hom{\mathop{\mathrm{Hom}}\nolimits}

 \def\id{\mathop{\mathrm{id}}\nolimits}
 \def\Trop{\mathop{\mathrm{Trop}}\nolimits}
 \def\ini{\mathop{\mathrm{in}}\nolimits}
 \def\alg{\mathop{\mathrm{alg}}\nolimits}
 \def\Gr{\mathop{\mathrm{Gr}}\nolimits}
 
\numberwithin{equation}{section}
 \numberwithin{figure}{section}

 \makeatletter
 \DeclareRobustCommand{\genericinterval}[2]{%
  \@ifstar{\genericinterval@star{#1}{#2}}{\genericinterval@nostar{#1}{#2}}}
   \newcommand{\genericinterval@star}[4]{\mathopen{}\mathclose{\left#1#3,#4\right#2}}
    \newcommand{\genericinterval@nostar}[4]{\mathopen{#1}#3,#4\mathclose{#2}}

         \makeatother
 
\begin{document}
 \title[On tropical cohomology]{On tropical cohomology of smooth algebraic varieties}
 \author{Ryota Mikami}
 \address{Institute of Mathematics, Academia Sinica, Astronomy-Mathematics Building, No.\ 1, Sec.\ 4, Roosevelt Road, Taipei 10617, Taiwan.}
 \email{ryotamikamimath467jhoiv9dhk3@gmail.com}
 \date{\today}
 
\begin{abstract}
 	In this paper, we give an explicit description of tropical cohomology of smooth algebraic varieties over trivially valued fields. 
 	We also construct ``monodromy weight'' spectral sequences for tropical cohomology of geometric strictly semi-stable reductions. 
 \end{abstract}
 
 \maketitle
 \setcounter{tocdepth}{1}
 \tableofcontents

\section{Introduction}
 Cohomology theories in tropical geometry were introduced by
 Gross-Siebert (\cite{GS10}) for 
 certain integral affine manifolds with singularities
 and 
 by 
 Itenberg-Katzarkov-Mikhalkin-Zharkov (\cite{IKMZ19}) 
 for tropical varieties. 
 The latter is called \emph{tropical cohomology}.
 Both are isomorphic to graded quotients of cohomology groups in  algebraic geometry in certain maximally degenerate cases.
 Hence they recover Hodge numbers, and are closely related to topological mirror symmetry.

 Tropical cohomology $H_{\Trop}^{p,q}$ was extended to algebaric varieties over complete valuation fields (\cite{Jel22}) 
 using tropical charts given by Chambert-Loir-Ducros \cite{CLD12}, Gubler \cite{Gub16}, and Jell \cite{Jel16}.
 The purpose of this paper is to study tropical cohomology of smooth algebraic varieties over trivially valued fields
 and of geometric strictly semi-stable reductions.
 
 Let $X$ be a smooth algebraic variety over a trivially valued field $K$.
 Liu defined a \textit{tropical cycle class map} 
 $$ \mathrm{CH}^p (X) \otimes_\Z \Q 
 \cong H^p (X_{\Zar}, \K_{M,X}^p \otimes \Q)
  \to H_{\Trop}^{p,p}(X;\Q)$$
 (\cite[Definition 3.8]{Liu20}),
 where $\K_{M,X}^p \otimes \Q$ is the Zariski sheaf of rational coefficents Milnor $K$-groups,
  and the isomorphism is Bloch's formula (\cite{Gab94}, \cite{Ros96}). 
 Inspired by his construction, 
 we introduce tropical analogs $K_T^p (-/K)$ (Definition \ref{tropical Milnor K}) of rational coefficents Milnor $K$-groups over $K$, \textit{tropical Milnor $K$-groups}.
 Similarly to Milnor $K$-groups (\cite{Gab94}, \cite{Ros96}), 
 the Zariski sheafification $\K_{T,X}^p$ have the \emph{Gersten resolution}  (Corollary \ref{corshftrK}) 
 \begin{align*}
 	0 \rightarrow \mathscr{K}_{T,X}^p 
 	& \rightarrow \bigoplus_{x \in X^{(0)}} i_{x *} K_T^p(k(x)/K)
 	\xrightarrow{d} \bigoplus_{x\in X^{(1)}} i_{x *}K_T^{p-1}(k(x)/K)   \\
 	  & 
 		 \xrightarrow{d} \bigoplus_{x\in X^{(2)}} i_{x *}K_T^{p-2}(k(x)/K) \xrightarrow{d}  \dots 
 		  \xrightarrow{d} \bigoplus_{x\in X^{(p)}} i_{x *}K_T^0(k(x)/K) \xrightarrow{d} 0 ,
 \end{align*}
 where
 $X^{(i)}$ is the set of points of $X$ of codimension $i$,
 the morphism
 $i_x \colon \Spec k(x) \rightarrow X$ ($x\in X$) is the natural one, 
 and $d$ are given by residue homomorphisms.
 In particular,  an analog of Bloch's formula holds:
 $$ H^p(X_{\Zar},\K_{T,X}^p) \cong \mathrm{CH}^p (X) \otimes_\Z \Q .$$

 The main theorem of this paper is the following.
 
 \begin{thm}[Theorem \ref{thm,trop,main}]\label{thm;trop;coho;Zariski;troK}
 For $p,q \geq 0$, we have
 $$H^{p,q}_{\Trop}(X;\Q) \cong H^q(X_{\Zar},\K_{T,X}^p).$$
 \end{thm}
 
 As a corollary, 
 we will give ``monodromy weight'' spectral sequences converging to
 tropical cohomology of geometric strictly semi-stable reductions.
 They and their constructions are similar to those for singular cohomology of complex algebraic varieties given by Steenbrink \cite{Ste76}.
 Let 
 $\pi \colon Y \to C$ be a flat, generically smooth, projective morphism from a smooth algebraic variety $Y$ to a smooth algebraic curve $C$ over  $K$.
 Let $c \in C$ be a closed point.
 We put $\hat{K}_c$ the discretely valued fraction field of the formal completion $\hat{\O_{C,c}}$ of the local ring $\O_{C,c}$,
 and 
 $Y_{\hat{K}_c} := Y \times_C \Spec \hat{K}_c$
 the base change.
 We assume that $Y_c:= \pi^{-1}(c)$ is a simple normal crossing divisor.
 (In equi-characteristic $0$, there are always strictly semi-stable reductions after base changes (\cite[Chapter II]{KKMSD73}).)
 We put $Y_{c,i} $ ($i \in I$) the irreducible components of $Y_c$, 
 and for $J \subset I$, we put $Y_{c,J} :=\bigcap_{i \in J} Y_{c,i}$.
 
 \begin{cor}\label{introduction spectral sequence for ss reduction}
 	For $r \geq 0$, we have a natural spectral sequence 
 	$$E_1^{p,q}= \bigoplus_{ \max \{0,p\} \leq u  }
   \bigoplus_{\substack{ J \subset I \\ \# J = -p + 2u+1 } }
 H^{r+p-u, p+q -u}_{\Trop} (Y_{c,J} ; \Q)  \Rightarrow H^{r,p+q}_{\Trop } (Y_{\hat{K}_c} ; \Q).$$
 \end{cor}
 
 \begin{rem}
 	We assume $K =\mathbb{C}$, 
 	and that for $J \subset I$,
 	we have 
 	\begin{align*}
 	H^{p, q }_{\Trop} (Y_{c,J} ; \Q ) & =0  \quad (p \neq q) ,\\
 	H_{\sing}^{2p+1}(Y_{c,J}(\mathbb{C} ); \Q  )& =0 \quad  (p \geq 0),  \\
 	 H_{\sing}^{2p}(Y_{c,J} (\mathbb{C} ); \Q )  & \cong \CH^p (Y_{c,J}) \otimes \Q  \quad 
 	(p \geq 0).
 	\end{align*}
 	(These hold when e.g., $Y_{c,J}$ is a smooth projective toric variety, or the wonderful compactification of the complement of a hyperplane arrangement (\cite{CP95}) (see Subsection \ref{subsec example})).
 	Then by Theorem \ref{thm;trop;coho;Zariski;troK} and Corollary \ref{introduction spectral sequence for ss reduction}, we have an isomorphism 
 	$$ H^{p, q }_{\Trop} (Y_{\hat{K}_c} ; \Q) \cong \Gr^W_{2p} H_{\sing}^{p+q}(Y_{\infty} (\mathbb{C} ); \Q ),$$
 	where the right-hand side is the weight-$2p$ graded quotient of the limit mixed Hodge structure (\cite[Theorem 11.22]{PS08}) of 
 	the $(p+q)$-th singular cohomology of fibers of $\pi$ at $c \in C$.   
 	Moreover, since, in this case, the right-hand side is of $(p,p)$-type, 
  by \cite[Corollary 11.25]{PS08}, 
  we also have 
 	$$ \dim H^{p, q }_{\Trop} (Y_{\hat{K}_c} ; \Q) \cong h^{p,q} (\pi^{-1}(c'))$$
 	for a closed point $c' \in C$ in general position. 

 	As mentioned in the beginning of introduction, 
 	this kind of results were previously 
 	proved 
 by Gross-Siebert (\cite{GS10}) for  certain (possibly non-semi-stable) toric degenerations of Calabi-Yau varieties using their cohomology of 
 certain tropical affine manifolds with singularities,
 and 
 by 
 Itenberg-Katzarkov-Mikhalkin-Zharkov (\cite{IKMZ19}) 
 for maximally degenerate smooth projective varieties 
 having smooth (i.e., locally matroidal) tropicalizations
 using tropical cohomology of the tropicalizations.
 Our result applies to any maximally degenerate strictly semi-stable reductions, 
 although it involves more complicated spaces,  Berkovich analytifications.
 \end{rem}
 
 There are some related works.
 	Theorem \ref{thm;trop;coho;Zariski;troK} and Corollary \ref{introduction spectral sequence for ss reduction} for curves are proved (in a more general form) by Jell-Wanner \cite{JW18} and Jell \cite{Jel19}. 
 	After the first version of the current paper, Amini-Piquerez \cite[Theorem 1.3]{AP21} proved a result similar to Theorem \ref{thm;trop;coho;Zariski;troK} for natural compactifications of unimodular tropical fans.
 
 Our proof of Theorem \ref{thm;trop;coho;Zariski;troK} is based on a theorem on coniveau spectral sequences of general cohomology theories, which has been developed by many mathematicians including Quillen \cite{Qui73}, Bloch-Ogus \cite{BO74}, Gabber \cite{Gab94}, Rost \cite{Ros96}, and Colliot-Th\'{e}l\`{e}ne-Hoobler-Kahn \cite{CTHK97}.
 This theorem was used to prove the Gersten resolutions and Bloch's formula  by Quillen \cite{Qui73} for algebraic $K$-groups, Gabber \cite{Gab94} and Rost \cite{Ros96} for Milnor $K$-groups. (Bloch's formula for $K_2$ was proved by Bloch \cite{Blo74}.)
 Using this theorem,
 we reduce Theorem \ref{thm;trop;coho;Zariski;troK} to $\A^1$-homotopy invariance of tropical cohomology over trivially valued fields.
 The main part of this paper is proof of $\A^1$-homotopy invariance.
 
 In this paper,
 we consider tropical  cohomology and tropical Milnor $K$-groups with $\Q$-coefficients.
 With $\Z$-coefficients, at least, we need to modify Section \ref{sec;trocoho;P^1;fiber}.
 
 The organization of this paper is as follows.
 In Section \ref{sec;not;term}, we fix several notations and terminologies.
 In Section \ref{sec;analytic;spaces}, we give a review on
 valuations and non-archimedean analytic spaces. 
 In Section \ref{sec polyhedra}, 
 we study homomorphisms to totally ordered abelian groups and limits of fans,
 which are used to define tropicalizations of valuations of higher heights.
 In Section \ref{sec tropicalizations}, we study tropicalizations of valuations of any heights.
 In Section \ref{sec;trocoho}, we recall tropical cohomology.
 In Section \ref{sec stalks, trop K}, 
 we study stalks of the sheaf $\F^p$, and introduce tropical Milnor $K$-groups.
 In Section \ref{sec;main theorem},
 we prove the main theorem (Theorem \ref{thm,trop,main}) of this paper.
 In Section \ref{sec;projective;line},
 we recall analytifications and tropicalizations of the affine line $\A^1$, explicitly.
 In Section \ref{sec;trocoho;P^1;fiber},
 we prove  $\A^1$-homotopy invariance of tropical cohomology over trivially valued fields.
 In Section \ref{section finite fields}, we prove the existence of corestriction map, which is needed for Theorem \ref{thm;trop;coho;Zariski;troK} over trivially valued finite fields.
 In Section \ref{sec ssreduction}, we prove Corollary \ref{introduction spectral sequence for ss reduction}.
 
\section{Notations and terminologies}\label{sec;not;term}
 For a $\Z$-module $G$ and  a commutative ring $R$, we put $G_R := G \otimes_\Z R$.
 
 A separated scheme $X$ of finite type over a field $K$ is called an algebraic variety.
 We denote the residue field of the structure sheaf at a point $x \in X $  by $k(x)$.
 We put $X_L$ the base change for a field extension $L/K$.
 When $K$ is equipped with the trivial valuation, 
 we put $X^{\circ}$ the subset of the Berkovich analytic space $X^{\Ber}$  consisting of valuations $v \in X^{\Ber}$ such that there exists a natural morphism $\Spec \O_v \to X$.
 
 We denote an algebraic closure of a field $L$ by $L^{\alg}$.
 We denote transcendental degree of a field extension $L/K$ by $\trdeg (L/K)$.
 For a valuation $v$ of a field $L$, we put $L_v$ the completion of $L$. 
 
 To simplify notation, we put $H_{\Trop}^{p,q} (-):= H_{\Trop}^{p,q}(-;\Q)$.

 In this paper except for Section \ref{sec;trocoho;P^1;fiber}, 
 let $M$ be a free $\Z$-module of finite rank $n$,
 and $N:= \Hom (M,\Z)$.
 
\section{Valuations and non-archimedean analytic spaces}\label{sec;analytic;spaces}
  In this section, we give a quick review on (non-archimedean) \textit{valuations} (Subsection \ref{subsec;val})
  and non-archimedean analytic spaces:
  \textit{Berkovich analytic spaces} (Subsection \ref{subsec;Berkovich}),
  \textit{Zariski-Riemann spaces} (Subsection \ref{subsec;ZR}), and
  \textit{Huber's adic spaces} (Subsection \ref{subsec;adic}).
  (See Section \ref{sec;projective;line} for analytifications of the affine line.)
  We refer to  \cite{HK94} and \cite[Chapter 6]{Bou72} for valuations,
  \cite{Hub93} and \cite{Hub94} for valuations and Huber's adic spaces,
  \cite{Ber90} and \cite{Tem15} for Berkovich analytic spaces,
  and
  \cite{Tem11} for Zariski-Riemann spaces.
  
 \subsection{Valuations}\label{subsec;val}
  In this subsection, rings are commutative rings with unit elements.
  \begin{dfn}
  We define a \emph{valuation} $v$ of a ring $R$ as a map $v \colon R \rightarrow \Gamma'_v\cup \{\infty\} $
  satisfying the following properties:
  \begin{itemize}
  \item  $\Gamma'_v$ is a totally ordered abelian group,
  \item $v(ab)=v(a)+v(b)  $ for $a,b \in R$, where we put  $ \gamma + \infty = \infty + \gamma = \infty$ ($\gamma \in \Gamma'_v$),
  \item $v(0)=\infty$ and $v(1)=0$,
  \item  $v(a + b)\geq \min \{v(a),v(b)\} $, where we extend the order of $\Gamma'_v$ to $\Gamma'_v\cup\{\infty\} $ by $ \infty \geq \gamma$ for $\gamma \in \Gamma'_v$.
  \end{itemize}
  \end{dfn}
  The set  $ \supp(v):=v^{-1}(\infty)$ is a prime ideal of $R$, which is called the \emph{support} of $v $.
  The subgroup  of $\Gamma'_v$ generated by $v(R)\setminus \{\infty\}$ is called the \emph{value group} of $v$.
  We denote it by $\Gamma_v$.
  The valuation $v$ gives a valuation on $\Frac (R / \supp (v) )$, which is also denoted by $v$.
  We put
  $$\O_v :=\{a \in \Frac(R/\supp(v)) \mid v(a) \geq 0 \} ,$$
  which is called  the \emph{valuation ring} of $v$, where $\Frac(R/\supp(v))$ is the fraction field.
  We put
  $$\kappa(v):=\Frac(\O_v / \{ a \in \O_v \mid v(a)>0\}),$$
  which is called  the \emph{residue field} of $v$.
  If $ R / \supp(v) \subset \O_v$,
  we call the image of  the maximal ideal under the canonical morphism $\Spec \O_v \to \Spec R$
  the \emph{center} of $v$.
  
  \begin{dfn}
  We call two valuations $v$ and $w$ of a ring $R$ are \emph{equivalent} if there exists an isomorphism $\varphi \colon \Gamma_v \overset{\sim}{\to} \Gamma_w$ of totally  ordered abelian groups satisfying $\varphi' \circ v =w$, where $\varphi' \colon \Gamma_v \cup \{\infty\} \rightarrow \Gamma_w\cup \{ \infty\}$ is the extension of $\varphi $ defined by $\varphi'(\infty) =\infty$.
  \end{dfn}
  
  We call the rank of a totally ordered abelian group $\Gamma$ as an abelian group the \emph{rational rank} of $\Gamma$.
  We denote it by $\rank \Gamma$.
  
  \begin{dfn}
  Let $\Gamma$ be a totally ordered abelian group.
  A subgroup $H$ of $\Gamma$ is called \emph{convex} if every element $\gamma \in \Gamma$ satisfying $h< \gamma<h' $ for some $h,h' \in H$ is contained in $H$.
  \end{dfn}
  
  When $H \subset \Gamma$ is a convex subgroup,  the quotient subgroup $\Gamma /H$  has a natural order, i.e., 
  $\overline{\gamma} \leq \overline{\gamma'} $
  if $\gamma \leq \gamma +h$ for some $h \in H$.
  
  The set of convex subgroups of $\Gamma$ are totally ordered by inclusions.
  \begin{dfn}
  We call the number of proper convex subgroups of a totally ordered abelian group $\Gamma$ the \emph{height} of $\Gamma$.
  We denote it by $\height \Gamma$.
  \end{dfn}
  
  The following well-known theorem is called the Harn embedding theorem.
  \begin{thm}[{Clifford \cite{Cli54}, Hausner-Wendel \cite{HW52}}]\label{thm:Harn:embed}
  Every totally ordered abelian group $\Gamma$ of finite height  $n$ has an embedding into the additive group $\R^n$ with the lexicographic order.
  \end{thm}
  
  \begin{rem}\label{convex subgroups of Rn}
  Let $G \subset \R^n$ be a totally ordered subgroup.  
  Then the convex subgroups of $G$ are  
  $$ G \cap (\{(0,\dots,0)\} \times \R^{n-r}) $$
   $ (0\leq r \leq n)$,
  where $(0,\dots,0) \in \R^r$.
  In particular, 
  a totally ordered abelian group $\Gamma$ of finite height  $n$ 
  can not be embedded in $\R^{n-1}$.
  \end{rem}

  We call the rational rank (resp.\ height) of  the value group of a valuation $v$ the rational rank (resp.\ height) of $v$.
  Rational ranks and heights of equivalence classes of valuations  
  are defined as those of representatives.
  
  \begin{dfn}\label{dfn trivial valuation}
  	A valuation $v$ of a field $L$ is said to be \emph{trivial} if  $\Gamma_v =\{0\}$.
  \end{dfn}
  
  Let $R$ be a ring.
  We call  the set of all equivalence classes of valuations of $R$ the \emph{valuation spectrum} of $R$.
  We denote it by $\Spv(R)$.
  We equip $\Spv(R)$ with the topology which is generated by 
  the sets
   $$ \{v\in \Spv(R) \mid v(a)\geq v(b)\neq \infty \} \qquad (a,b \in R) .$$
  In this paper, \emph{generalizations} and \emph{specializations} of a valuation in (subsets of) $\Spv(R)$ are in the topological sense.
  
  Let $v \colon R \rightarrow \Gamma_v \cup \{\infty\}$ be a valuation,
  $H \subset \Gamma_v$ a convex subgroup. We define a map
  \begin{align*}
  v/H\colon R\rightarrow (\Gamma_v/H)\cup \{\infty\}, \quad &  a \mapsto \begin{cases}
  v(a)\text{ mod  } H & \text{if } v(a) \neq \infty \\
   \infty  & \text{if } v(a)=\infty.
   	\end{cases}
  \end{align*}
  
  \begin{lem}[{Huber-Knebusch \cite[Lemma 1.2.1]{HK94}}]
  The map $v/H$ is a valuation of $R$, and it is a generalization of $v$ in $\Spv(R)$,
  called a vertical (or primary) generalization.
  \end{lem}
  
  For a field $K$, all specializations in $ \Spv (K)$ are vertical \cite[Proposition 1.2.4]{HK94}.
  
  \begin{rem}[{Bourbaki \cite[Chapter 6. Subsection 4.1, 4.2, 4.3]{Bou72}}]
  \label{rem;val;vert;special;resi;fld;val}
  For a valuation $v$ of a field $K$,
  there is a natural bijection between specializations $w$ of $v$ in $\Spv (K)$ and valuations $\overline{w}$ on the residue field $\kappa(v)$ of $v$
  given as follows.
  Let $w  \in \Spv (K)$ be a specialization of $v$.
  Then the image of the valuation ring $\O_w (\subset \O_v)$
  under the natural map $\O_v \twoheadrightarrow \kappa (v)$
  is the valuation ring $\O_{\overline{w}} $ 
  of a valuation $\overline{w} $ of $\kappa (v)$.
  The value group $\Gamma_{w}$ contains $\Gamma_{\overline{w}}$ 
  as a convex subgroup, 
  and we have $w / \Gamma_{\overline{w}} =v$.
  \end{rem}
   
 \subsection{Berkovich analytic spaces}\label{subsec;Berkovich}
  In \cite[Chapter 3]{Ber90}, Berkovich introduced the \textit{Berkovich analytic space} $X^{\mathrm{Ber}}$, which is a locally compact (\cite[Theorem 1.2.1]{Ber90}) Hausdorff  (\cite[Theorem 3.4.8 (i), Theorem 3.5.3 (i)]{Ber90}) topological space, associated to an algebraic variety $X$ over a complete valuation field $(L,v_L \colon L^{\times} \to \R)$ of height $\leq 1$.
  Berkovich analytic spaces are, as sets,
  the sets of valuations $v $ with target $\R \cup \{\infty \}$.
  (Strictly speaking, Berkovich uses  multiplicative seminorms instead of  valuations, but we identify them in a natural way by $ - \log \colon \R_{> 0} \cong \R $.)

  We denote the Berkovich analytic space associated to an $L$-\emph{affinoid algebra} $A$ (\cite[Definition 2.1.1]{Ber90}) by $\M(A)$ \cite[Section 1.2]{Ber90}.
  There exists a unique minimal closed subset of $ \mathscr{M}(A)$
  on which every element of $A$ has its minimum \cite[Corollary 2.4.5]{Ber90}, 
  called the \textit{Shilov boundary} of $\mathscr{M}(A)$.
  It is a finite set. 
   
 \subsection{Zariski-Riemann spaces}\label{subsec;ZR}
  For a finitely generated extension $L/K$ of fields,
  we put $\ZR(L/K) \subset \Spv (L)$ the subspace of equivalence classes of valuations of $L$ which are trivial on $K$.
  We call $\ZR (L/K)$ the \textit{Zariski-Riemann space}.
  
  There is another expression.
  For each $v \in \ZR(L/K)$ and proper integral algebraic variety $X$ over $K$ with function field $L$,
  by the valuative criterion of properness, there exists a unique morphism $\Spec \O_v \rightarrow X$ over $K$ extending the function field $\Spec L \to X$.
  This induces a map $\ZR(L/K) \twoheadrightarrow X$ by taking the image of the maximal ideal of $ \O_v$ (i.e., center of $v$).
  We have a map from $\ZR (L/K)$ to
  the inverse limit $\underleftarrow{\lim} X$ (as topological spaces) of such varieties $X$ under birational morphisms.
  
  The following  Propositon is well-known,
  	see e.g., \cite[Corollary 3.4.7]{Tem11}.
  \begin{prp}\label{ZRhomeo}
  The map  $\ZR(L/K) \rightarrow \underleftarrow{\lim} X$ is a homeomorphism.
  \end{prp}
  
  \begin{rem}[Abhyankar's inequality]\label{Abhyankar's inequality}
  For $v \in \ZR (L/K)$, 
  we have 
  $$\trdeg (\kappa (v) /K ) + \rank \Gamma_v \leq \trdeg (L/K).$$
  The equality holds for some explicit $v$.
  \end{rem}
  
  \begin{dfn} (This notation is not standard.)
  $$(\Spec L/K)^{\Ber}
  := \{ v\colon L \to \R \cup \{\infty\} \mid \text{a valuation trivial on } K \}.
  $$
  \end{dfn}
  For an integral algebraic variety $X$ over $K$ with function field $L$,
  the set $(\Spec L/K)^{\Ber}$ can be identified with the subspace of the analytification $X^{\Ber}$  
  consisting of points whose supports are the generic point of $X$. We introduce the induced topology on $(\Spec L/K)^{\Ber}$. 
   
 \subsection{Huber's adic spaces}\label{subsec;adic}
  For a separated scheme $X$ of finite type over a trivially valued field $K$,
  the adic space $X^{\ad}$ associated to $X$ is defined as follows.
  (See \cite{Hub93} and \cite{Hub94} for notations and the theory of his adic spaces.)
  For each affine open subvariety $U=\Spec R \subset X$,
  we put $U^{\ad}:= \Spa(R,R \cap K^{\alg})$ the space of equivalence classes of valuations on $R$ trivial on $K$
  (where $R$ is equipped with the discrete topology).
  We define $X^{\ad}$ by glueing $U^{\ad}_{\alpha}$ for an affine open covering $\{U_{\alpha}\}_\alpha$ of $X$.
  
  \begin{rem}
  Taking supports of valuations induces a surjective map
  $X^{\ad} \twoheadrightarrow X$
  whose fiber of $x \in X$ is homeomorphic to $\ZR(k(x) / K)$.
  \end{rem}
  
  \begin{rem}\label{rem;Ber;adic}
  Taking equivalence classes induces a map $X^{\Ber} \to X^{\ad}$.
  This induces a bijection
  $$X^{\Ber} / (\emph{the equivalence relation of valuations}) \cong X^{\ad,\height \leq 1}$$
  to the subset $X^{\ad,\height \leq 1}$ of $X^{\ad}$
  consisting of equivalence classes of valuations of height $\leq 1$.
  \end{rem}
  
  For an algebraic variety $Y$ over a complete valuation field $L$ of height $1$,
  let $Y^{\ad}$ be the adic space associated to $Y$
  in the sense of \cite{Hub94}.
  
  \begin{rem}
  Taking equivalence classes induces an injective map 
  $Y^{\Ber} \to Y^{\ad}$ 
  whose image is 
  the subset $Y^{\ad,\height = 1}$
  consisting of equivalence classes of valuations of height $\leq 1$.
  Note that this is not a continuous map.
  \end{rem}
  
\section{Targets of tropicalization maps}\label{sec polyhedra}
  In this section,
  we recall a partial compactification $\bigsqcup_{\sigma \in \Sigma } N_{\sigma ,\R}$ of $\R^n$ and fans and polyhedral complexes in it (Subsection \ref{subsec;fan}).
  We also study the set of equivalence classes of group homomorphisms
  from $M$ 
  to totally ordered abelian groups (Subsection \ref{subsection Homomorphisms to totally ordered abelian groups}).
  There is a natural bijection from it to the limit of fan structures of $\Hom (M,\R)= N_\R$ (Subsection \ref{subsection Limits of fan structures}).
  They are used as targets of tropicalization maps of valuations of heigher heights in Subsection \ref{subsection Tropicalizations of Zariski-Riemann spaces} and Subsection \ref{subsec;trop;adic}.
  
  Recall that 
  $M $ is a free $\Z$-module of finite rank $n$
  and $N:=\Hom(M,\Z)$.
  Let $\Sigma$ be a fan in $N_\R$, and $T_{\Sigma}$ the normal toric variety over a field $K$ corresponding to $\Sigma$. See \cite{CLS11} for toric varieties.
  In this paper, cones mean  strongly convex rational polyhedral cones.
  There is a natural bijection between cones $\sigma \in \Sigma$ and torus orbits $O(\sigma)$ in $T_{\Sigma}$.
  The torus orbit $O(\sigma)$ is isomorphic to a torus $ \Spec K[M  \cap \sigma^{\perp}].$
  We put $N_{\sigma}:= \Hom (M  \cap \sigma^{\perp} , \Z)$,
  \begin{align*}
  \sigma^{\perp}:= & \{m\in M_\R \mid n(m)=0 \ (n\in \sigma) \}, \\
  \sigma^{\vee}:= &\{m\in M_\R \mid n(m ) \geq 0 \ (n\in \sigma) \}.
  \end{align*}
  
 \subsection{Fans and polyhedral complexes}\label{subsec;fan}
 We recall tropical toric varieties, which are target spaces of Kajiwara-Payne's tropicalizations, see \cite{Kaj08}, \cite{Pay09}, \cite{FGP14}.
  A topology on the disjoint union  $\bigsqcup_{\sigma \in \Sigma } N_{\sigma,\R}$  is defined as follows.
  We extend the canonical topology on $\R$  to that on $\R \cup \{ \infty \}$ so that  $(a, \infty]$ ($a \in \R$) are a basis of neighborhoods of $\infty$.
   We consider the set of semigroup homomorphisms $\Hom (M \cap \sigma^{\vee}, \R \cup \{ \infty \}) $  as a topological subspace of $(\R \cup \{ \infty \})^{M \cap \sigma^{\vee}} $.
   For $\sigma \in \Sigma$,
  we define a topology on $\bigsqcup_{ \substack{\tau \in \Sigma \\ \tau \subseteq  \sigma  } } N_{\tau,\R}$ by a natural bijection
  \begin{align*}	
  \Hom (M \cap \sigma^{\vee}, \R \cup \{ \infty \}) & \cong \bigsqcup_{ \substack{\tau \in \Sigma \\ \tau \subseteq  \sigma  } } N_{\tau,\R}  \\
   n & \mapsto n|_{\langle n^{-1}(\R) \rangle },
  \end{align*}
  where $\langle n^{-1}(\R) \rangle  \subset M $ is the abelian subgroup generated by $n^{-1}(\R)$, which is of the form $M \cap \tau^{\perp}$ for some $\tau \in \Sigma$ with $\tau \subset \sigma$.
   Then we define a topology on  $\bigsqcup_{\sigma \in \Sigma } N_{\sigma,\R}$ by glueing the topological spaces $ \bigsqcup_{\substack{\tau \in \Sigma \\ \tau \subseteq \sigma  }} N_{\tau,\R} $.
  
  We shall recall fans and polyhedral complexes in $\bigsqcup_{\sigma \in \Sigma } N_{\sigma,\R}$.
  Let $\Gamma\subset \R$ be a subgroup.
  \begin{dfn}
  A subset of $\R^n$ is called a \emph{$\Gamma$-rational polyhedron}
  if  it is the intersection of finitely many sets of the form
  $$\{x \in \R^n \mid \langle x , a \rangle \leq b \} \ (a \in \Z^n, b \in  \Gamma),$$
  here $\langle x,a\rangle$ is the usual inner product of $\R^n$.
  \end{dfn}
  
  Strongly convex $\{0\}$-rational polyhedra are called \emph{cones},
  and we simply call $\R$-rational polyhedra \emph{polyhedra}.
  \begin{dfn}
  For a cone $\sigma\in \Sigma $ and a ($\Gamma$-rational) polyhedron (resp.\ a cone) $C\subset N_{\sigma,\R}$,
  we call its closure $P:=\overline{C}$ in  $ \bigsqcup_{\sigma \in \Sigma } N_{\sigma,\R}$ a \emph{($\Gamma$-rational) polyhedron} (resp.\ a \emph{cone}) in $\bigsqcup_{\sigma \in \Sigma } N_{\sigma,\R}$.
  In this case, we put $\relint(P):=\relint(C)$, and call it the \emph{relative interior} of $P$.
  We put $\dim(P):=\dim(C)$.
  We also put $\sigma_P \in \Sigma $ the unique cone such that $\relint(P)\subset N_{\sigma_P,\R}$.
  \end{dfn}

  A subset $Q$ of a polyhedron $P$ in $\bigsqcup_{\sigma \in \Sigma } N_{\sigma,\R}$ is called a \emph{face} of $P$ if it is
  the closure of the intersection $P^a \cap  N_{\tau,\R}$
  in $ \bigsqcup_{\sigma \in \Sigma } N_{\sigma,\R}$
  for some $ a \in M \cap \sigma_P^{\perp}$ and some cone $\tau \in \Sigma $,
  where $P^a$ is the closure of
  $$\{x \in P\cap  N_{\sigma_P,\R} \mid x(a) \leq  y(a) \text{ for any } y \in P \cap  N_{\sigma_P,\R} \}  $$
  in $ \bigsqcup_{\sigma \in \Sigma } N_{\sigma,\R}$.
  A finite collection $\Lambda$ of  ($\Gamma$-rational) polyhedra (resp.\ cones) in $ \bigsqcup_{\sigma \in \Sigma } N_{\sigma,\R}$ is called  a   \emph{($\Gamma$-rational) polyhedral complex} (resp.\ a \emph{fan}) if it satisfies the following two conditions:
  \begin{itemize}
  \item for $P \in \Lambda$, each face of $P$ is also in $\Lambda$, and 
  \item for $P,Q \in \Lambda$, the intersection $P \cap Q$ is a face of $P$ and $Q$.
  \end {itemize}
  We call the union
  $$\lvert \Lambda \rvert := \bigcup_{P\in\Lambda} P \subset \bigsqcup_{\sigma \in \Sigma } N_{\sigma,\R} $$
  the \emph{support} of $\Lambda$.
  We say that $\Lambda$ is a \emph{($\Gamma$-rational) polyhedral complex} (resp.\ a \emph{fan}) \emph{structure} of $\lvert \Lambda \rvert $.
  For a  subset $B \subset \bigsqcup_{\sigma \in \Sigma } N_{\sigma,\R}$,
  we put
  $$\Lambda \cap B := \{P \cap B\}_{P \in \Lambda }.$$
  
  A polyhedral complex $\Lambda'$ is called a refinement of a polyhedral complex $\Lambda$ (or we say that $\Lambda'$ is finer than $\Lambda$)
  if their supports are the same and 
  for any $P \in \Lambda'$,
  there exists a polyhedron $Q_{P} \in \Lambda$
  such that $\relint (P ) \subset \relint Q_P.$
  This induces a surjective map 
  $\Lambda' \ni P \mapsto Q_P \in \Lambda$.
  
  The following notions are used to study tropicalizations over valuation fields of height $1$ by Gubler \cite{Gub13}, e.g., for tropical compactifications.

  \begin{dfn}\label{admissible cones}
  	A subset of $\R^n \times \R_{\geq 0}$ is called a \emph{$\Gamma$-admissible cone} 
  if  it is the intersection of finitely many sets of the form
  $$\{(x,s) \in \R^n \times \R_{\geq 0} \mid \langle x , a \rangle  + s b \leq 0 \} \ (a \in \Z^n, b \in  \Gamma)$$
  and if it does not contains a line.
  \end{dfn}
  
  We define $\Gamma$-admissible fans similarly to fans.
  For a $\Gamma$-admissible cone $\sigma$, 
  we put $\sigma_1 := \{x \in \R^n \mid (x,1) \in \sigma\}$, a $\Gamma$-rational polyhedron. 
  For a $\Gamma$-admissible fan $\Sigma$, 
  we put $\Sigma_1 := \{ \sigma_1\}_{\sigma \in \Sigma}$, a $\Gamma$-rational polyhedral complex. 
   
 \subsection{Homomorphisms to totally ordered abelian groups}\label{subsection Homomorphisms to totally ordered abelian groups}
  In this subsection, we shall give a description of equivalence classes of group homomorphisms from $M$ to totally ordered abelian groups.
  Recall that 
  every totally ordered abelian group of height $r$ can be embedded in $\R^r$ (Theorem \ref{thm:Harn:embed}).

  For $l_i \in N_\R$, 
  we put $L_i$   
  the minimal $\Q$-linear subspace of $ N_\Q$ such that  $L_{i,\R}:=L_i\otimes_\Q \R \subset N_\R $ contains $l_i$.
  For $r\in \Z_{\geq 1}$,
  we put
  $$J_r:= J_r (N):= \bigg\{ (l_1,\dots,l_r) \in (N_\R)^r \bigg|   l_i
  \notin \sum_{j=1}^{i-1} L_{j,\R}
   \ (1 \leq i\leq r)
   \bigg\},$$
   where $ \sum_{j=1}^{0} L_{j,\R}:=\{0\}$.
  We put $J_0:=J_0 (N):=\{0\}$.
  We identify $(N_\R)^r$ with $\Hom (M,\R^r)$.
  
  \begin{lem}
  	Let $\L = (l_1,\dots,l_r) \in (N_\R)^r$.
  The abelian subgroup $\L(M) \subset \R^r$ is of height $r$ with respect to the lexicographic order
  if and only if 
  $\L = (l_1,\dots,l_r) \in  J_r$.
  \end{lem}
  \begin{proof}
  By Remark \ref{convex subgroups of Rn}, 
  the subgroup $\L(M) \subset \R^r$ is of height $r$ 
  if and only if 
  $$\L(M) \cap (\{0\}^{j-1} \times \R^{r-j+1})  \neq \L(M) \cap (\{0\}^{j} \times \R^{r-j})$$
  for any $j$.
  Since 
  $$ (l_1,\dots,l_s)(M)  \cong \L(M) / (\L(M) \cap (\{0\}^{s} \times \R^{r-s})),$$
  the latter is equivalent to that the natural surjection
  $$\R^j \supset  (l_1,\dots,l_{j}) (M) \twoheadrightarrow (l_1,\dots,l_{j-1})(M) \subset \R^{j-1}  $$
  is not injective for any $j$, i.e., $\L \in J_r$.
  (Note that $M\cap \Ker l_i =M\cap L_i^{\perp}$.)
  \end{proof}
  
  We say that $\L =(l_1, \dots, l_r) \in J_r$ and $ \L'=(l_1', \dots, l_r') \in J_r$
  are \emph{equivalent}
  if there is an isomorphism 
  $\varphi \colon \L(M) \cong \L'(M)$
  as totally ordered groups such that $\varphi \circ \L = \L'$.
  We put 
  $\L \sim_{I_r} \L'$
  if we have
  $$\R_{>0} \cdot l_i = \R_{>0} \cdot l'_i
  \in N_\R / \sum_{j=1}^{i-1} L_{j,\R} \ (1 \leq i\leq r),$$
   where we denote  the images of $l_i$ and $l_i'$ under the projection also by $l_i$ and $l_i'$.
  We also put $J_0/\sim_{I_r} :=J_0$.
  
  \begin{lem}
  $\L \in \J_r$ and $\L'\in \J_r$ are equivalent
  if and only if
  $\L  \sim_{I_r} \L'$.
  \end{lem}
  
  \begin{proof}
  	We give a proof by induction on $r \geq 1$.
  	The case of $r=1$ is trivial.
  	We assume that $r \geq 2$, and one of the two assertions (the equivalence  or $\L \sim_{I_r} \L'$) holds.
  	Then by the case of $r=1$, 
      the two assertions hold for 
  	$l_1, l_1' \in J_1$. 
  	In particular, we have $\Ker l_1 = \Ker l_1'$.
  	By the hypothesis of induction, the two assertions also hold for $M\cap \Ker l_1$ and 
  	$$(l_2,\dots,l_r),  (l_2',\dots,l_r') \in J_{r-1}(N /N \cap  L_1 ),$$
   where we denote  the images of $l_i$ and $l_i'$ ($2\leq i \leq r$) under the projection also by $l_i$ and $l_i'$.
      (Note that 
      $\Hom (M \cap \Ker l_1,\Z) \cong N / N \cap L_1$.)
  	It is easy to see that we have $\L \sim_{I_r} \L'$ since $l_1 \sim_{I_1} l_1'$ and 
      $(l_2,\dots,l_r) \sim_{I_{r-1}}  (l_2',\dots,l_r')$.
  	There are unique isomorphisms
  	$\varphi_1 \colon l_1(M) \cong l_1'(M)$
  	and 
  	$$\varphi_2 \colon (l_2,\dots,l_r)(M \cap \Ker l_1)  
  	\cong (l_2',\dots,l_r') (M \cap \Ker l_1)$$
  	as ordered groups which are compatible with $l_1,l_1', (l_2,\dots,l_r)$ and $(l_2',\dots,l_r')$.
    Hence
  	there exists a unique isomorphism
  	$\varphi \colon \L(M) \cong \L'(M)$
  	as ordered groups such that $\varphi \circ \L = \L'$.
  	Hence $\L $ and $\L'$ are eqiuvalent.
  \end{proof}
  
  For 
  $ \L \in J_r $, we put $\rank \L$ the rank of the free $\Z$-module $\L(M)$.
  
  We extend the notion of equivalence to group homomorphisms $  M \to \Gamma $ to general totally ordered abelian groups $\Gamma$ (from the case of $\R^r$) in the natural way.
  We put $\Val M$ the set of equivalence classes of group homomorphisms 
  from $M$ to totally ordred abelian groups,
  and 
  $\Val^{\height =r} M$ the subset of height $r$. 
  
  \begin{cor}\label{homomorphisms to totally ordered abelian groups}
  There is a natural bijection 
  between $\Val^{\height =r} M \cong J_r / \sim_{I_r}	$.
  \end{cor}
  
  Let $M'$ be a free $\Z$-module of finite rank, $N':=\Hom (M',\Z)$ its dual, and $\psi \colon M \to M'$ a group homomorphism. 
  Then $\psi $ induces maps $N'_\R \to N_\R$ and 
  $$\bigsqcup_{r\geq 0} J_r (N') /\sim_{I_r} \to \bigsqcup_{r\geq 0} J_r (N) /\sim_{I_r}.$$
  Under the bijection in Corollary \ref{homomorphisms to totally ordered abelian groups}, this map coincides with a map
  $$-\circ \psi \colon \Val M' \to \Val M.$$
   
 \subsection{Limits of fan structures}\label{subsection Limits of fan structures}
  In this subsection, we shall give a bijection between the limit of fan structures of $N_\R$ and the set of equivalence classes of group homomorphisms from $M$ to totally ordered abelian groups.
  
  Let 
  $ \underleftarrow{\lim} \Xi$ be the inverse limit of all fan structures $\Xi$ of $N_\R$ as sets. (Maps are given by refinements, see Subsection \ref{subsec;fan}.)
  For a subset $A\subset \R^r$, we write $A\geq 0$ (resp. $A>0$) if $a\geq 0$ (resp. $a>0$) for any $a\in A$ 
  with respect to the lexicographic order.
  \begin{lem}\label{lem;trop;adic;another;express;NR}
  There is a bijection
  \begin{align*}
  \bigsqcup_{r \geq 0}  J_r	 / \sim_{I_r}
   \cong  \underleftarrow{\lim} \Xi
  \end{align*}
  given by
  $\{0\} \mapsto (\{0\})_{\Xi}$ 
  and 
  $ \L= (l_1, \dots,l_r) \mapsto (P_{\L,\Xi})_{ \Xi}$,
  where for sufficiently fine fan structure $\Xi$ (which depends on $\L$),
  the cone $P_{\L,\Xi}$ is the unique cone in $\Xi$ satisfying
  \begin{align*}
  	& \L (M \cap P_{\L,\Xi}^{\vee }) \geq 0   ,   \\
  	& \L ((M \cap P_{\L,\Xi}^{\vee}) \setminus P_{\L,\Xi}^{\perp}) > 0 .
  \end{align*}
  \end{lem}
  \begin{proof}
  We prove the assertion (including the existence and the uniquness of such $P_{\L,\Xi}$) by induction on $\dim N$.
  When $\dim N=0$, it is trivial.
  We assume $\dim N \geq 1$.
  Note that $\bigcap_{\Xi} P_\Xi $ is a halfline 
  for any $(P_\Xi)_{\Xi} \in \underleftarrow{\lim} \Xi  \setminus (\{0\})_\Xi $.
  We fix $l \in N_\R \setminus \{0\}$.
  We shall show that there exists a bijection
  \begin{align}\label{eq limit fan structure}
  & \bigsqcup_{r \geq 1}\{ \L= (l_1,\dots, l_r) \in J_r |\R_{>0} \cdot l_1 = \R_{>0} \cdot l \} / \sim_{I_r}  \\
  \cong  &
   \{(P_\Xi)_{\Xi} \in \underleftarrow{\lim} \Xi  \mid l \in \bigcap_{\Xi} P_\Xi  \}   . \notag
  \end{align}
  We will see that the required bijection is given as disjoint union of bijections (\ref{eq limit fan structure}).
  We put $L$  
  the minimal subspace of $ N_\Q$ such that  $L_\R$ contains $l $.
  Let $\Xi_0$ be a sufficiently fine fan structure of $N_\R$
  such that
   the cone 
  $P_{l,\Xi_0}  \in \Xi_0$ 
    whose relative interior contains $l$ exactly
    spans $L_\R$.
  Then for $r \geq 1$,  we have a bijection
  \begin{align*}
  & \{ L= (l_1,\dots, l_r) \in J_r | \R_{>0} \cdot l_1 =\R_{>0} \cdot  l \} / \sim_{I_r} \\
  \cong  &
    J_{r-1} (N_{P_{l,\Xi_0}})  / \sim_{I_{r-1}} 
  \end{align*}
  by $(l_1 ,l_2,\dots,l_r) \mapsto (l_2, \dots,l_r)$.
  By the hypothesis of the induction, 
  we have a bijection
  \begin{align*}
    J_{r-1} (N_{P_{l,\Xi_0}})  / \sim_{I_{r-1}} 
    \cong \underleftarrow{\lim} \Lambda,
  \end{align*}
  where $\Lambda$ runs through all fan structures of $N_{P_{l,\Xi_0},\R}$.
  By Remark \ref{induction limits of fans}, 
  we have 
  $$
   \{(P_\Xi)_{\Xi} \in \underleftarrow{\lim} \Xi  \mid l \in \bigcap_{\Xi} P_\Xi  \}  
   \cong \underleftarrow{\lim}  \Lambda.$$
   Hence 
  we get bijection (\ref{eq limit fan structure}). 
  It remains to show the existence and the uniqueness of $P_{\L,\Xi}$ and that bijection (\ref{eq limit fan structure}) maps $\L \mapsto (P_{\L,\Xi})_\Xi$.
  We fix
   $\L = (l_1,l_2,\dots,l_r) \in J_r$.
  When $r=1$, it is easy, i.e.,
  the cone
  $P_{\L,\Xi} \in \Xi$ in the assertion is the cone whose relative interior contains $l_1$,
  and $\L =(l_1) \mapsto (P_{\L,\Xi})_{\Xi}$ under bijection (\ref{eq limit fan structure}).
  We assume 
  $r \geq 2$. 
  For a cone $P \subset N_\R$,
  we have 
  \begin{itemize}
  	\item $\L (M \cap P^{\vee}) \geq 0$
  	if and only if 
  	$$ l_1 (M \cap P^{\vee}) \geq 0 \ \text{and} \  (l_2,\dots,l_r) (M \cap l_1^{\perp} \cap P^{\vee}) \geq 0,$$
  	and in this case, 
  	\item $\L ((M \cap P^{\vee})\setminus P^{\perp}) > 0$
  	if and only if  
  	$$(l_2,\dots,l_r) ((M \cap l_1^{\perp} \cap P^{\vee}) \setminus P^{\perp}) > 0.$$
  \end{itemize}
  Note that $l_1 (M\cap P^{\vee}) \geq 0$ if and only  if $l_1 \in P$.
  Hence by induction, $P_{\L,\Xi}$ exists, and  is unique. 
  $(P_{\L,\Xi})_{\Xi}$ is the image of $\L$
  under bijection (\ref{eq limit fan structure}).
  \end{proof}

  \begin{rem}\label{induction limits of fans}
  Let $l \in N_\R $ be a non-zero element.
  Let $L$ and  $\Xi_0$ be as in proof of Lemma \ref{lem;trop;adic;another;express;NR}. (In particular, we have $N_{P_{l,\Xi_0}} \cong N / N \cap L$.) 
  For each fan structure $\Xi $ of $N_\R$, 
  we put 
  $P_{l,\Xi} \in \Xi $ the cone whose relative interior contains $l$.
  When $\Xi$ is finer than $\Xi_0$,
  the natural surjection
  $N_{P_{l,\Xi},\R} \twoheadrightarrow N_{P_{l,\Xi_0},\R}$
  is a bijection.  We identify them.
  The set
  $$ \{ \pr_{P_{l,\Xi}}(P)  \mid P \in \Xi , \ P_{l,\Xi} \subset P \} $$
  is a fan structure of $N_{P_{l,\Xi},\R}$,
   where $\pr_{P_{l,\Xi}} \colon N_\R \twoheadrightarrow N_{P_{l,\Xi},\R}$ is the projection.
  Since
  $$
   \{(P_\Xi)_{\Xi} \in \underleftarrow{\lim} \Xi \mid P_{l,\Xi} \subset P_\Xi  \ \text{for any } \Xi \}
   =
   \{(P_\Xi)_{\Xi} \in \underleftarrow{\lim} \Xi \mid l \in \bigcap_\Xi P_\Xi  \},
   $$
   the projection $ \pr_{P_{l,\Xi}}(P)$
   induces a bijection
  $$
   \{(P_\Xi)_{\Xi} \in \underleftarrow{\lim} \Xi  \mid l \in \bigcap_{\Xi} P_\Xi  \}  
  \cong \underleftarrow{\lim} \Lambda,$$
  where $\Lambda$ runs through all fan structures of $N_{P_{l,\Xi_0},\R}$.
  \end{rem}
  
  By  Corollary \ref{homomorphisms to totally ordered abelian groups} and Lemma \ref{lem;trop;adic;another;express;NR}, we have the following.
  \begin{cor}\label{limits of fan and homomophirms to totally ordered groups}
  	We have  a natural bijection 
  	$$\Val M \cong 
   \underleftarrow{\lim}\Xi,$$ 
   where $\Xi$ runs through all fan structures of $N_\R$.
  \end{cor}
  
  By Corollary \ref{limits of fan and homomophirms to totally ordered groups}, 
  we define 
  ranks and heights of elements of 
   $\underleftarrow{\lim}\Xi$
   as those of their images in $\Val M$.
  
  \begin{rem}\label{rem;trop;adic;l_r;P;span;same}
  	\begin{itemize}
  		\item 
  Let
  $ \L= (l_1, \dots,l_r) \in J_r $ map to 
  $ (P_{\L,\Xi})_{ \Xi} \in \underleftarrow{\lim} \Xi $.
  By  proof of Lemma \ref{lem;trop;adic;another;express;NR},
  for a sufficiently fine fan structure $\Xi$  of $N_\R$,
  we have
  $$\Span_\R (P_{\L, \Xi})  = \sum_{i=1}^r L_{i,\R},$$
  where $L_{i} $ is the minimal $\Q$-linear subspace of $N_\Q$ 
  such that $L_{i,\R}$ contains $l_i$.
  (In particular, we have $\rank \L =  \dim P_{\L, \Xi}.$)
  
  \item Conversely, for a fan structure $\Xi$ of $N_\R$ and 
        a cone $P \in \Xi$ of dimesion $r$, 
  	  there exists 
  	  $$\L = (l_1,\dots,l_{\dim P}) \in J_{\dim P} \cap (N_\Q)^r$$
  	  such that 
  	    $P_{\L,\Xi} =P$ 
  	   and 
  	    $$\Span_\R P = \sum_{i=1}^r L_{i,\R} = \sum_{i=1}^r \R \cdot l_{i}.$$
  	  Namely, by taking a refinement of $\Xi$, 
  	  we may assume that $P= \R_{\geq 0}^{\dim P} \times \{0\}^{n-\dim P}$
  	  for some identification $N \cong \Z^n$,
  	  then 
  	  $\L =(e_1,\dots,e_{\dim P})$ is a required one, 
  	  where $e_i = (0,\dots,0,1,0, \dots, 0)$ is the $i$-th coordinate. 
  	\end{itemize}
  \end{rem}

  Let $M'$ be a free $\Z$-module of finite rank, $N':=\Hom (M',\Z)$ its dual, and $\psi \colon M \to M'$ a morphism. 
  Then $\psi $ induces a morphism $\psi \colon N'_\R \to N_\R$.
  Let  $\Xi'$ (resp. $\Xi$) be a fan structure of $N'_\R $ (resp. $N_\R$)
  such that 
  for a cone $P'\in \Xi'$, 
  the cone $\psi (P')$ is contained in some cone $P\in \Xi$.
  Let $P_{P'} \in \Xi$ be the minimal one among them.
  Then we get a map $\Xi' \ni P' \mapsto P_{P'} \in \Xi$. 
  It induces a map 
  $$\underleftarrow{\lim}_{\Xi'} \Xi'
  \to 
  \underleftarrow{\lim}_{\Xi} \Xi,$$
  where $\Xi'$ (resp. $\Xi$) runs through all fan structures of $N'_\R$ (resp. $N_\R$).
  Under the bijection in Lemma \ref{lem;trop;adic;another;express;NR}, 
  this map coincides with a map 
  $$\bigsqcup_{r\geq 0} J_r (N') /\sim_{I_r} \to \bigsqcup_{r\geq 0} J_r (N) /\sim_{I_r}$$
  given by $\psi$.
  Hence under the bijection in Corollary \ref{homomorphisms to totally ordered abelian groups}, it also coincides with a map
  $$-\circ \psi \colon \Val M' \to \Val M.$$
   
\section{Tropicalizations}\label{sec tropicalizations}
   In this section, we 
   shall recall 
   \textit{tropicalizations} of  Berkovich analytifications of algebraic varieties (Subsection \ref{subsec;trop;Ber})  
   and tropical compactifications (Subsection \ref{subsec;trop;triv;val}). 
   See also Section \ref{sec;projective;line}
   for tropicalizations of the affine line. 
   We also introduce and study 
   tropicalizations of Zariski-Riemann spaces (Subsection \ref{subsection Tropicalizations of Zariski-Riemann spaces}) and Huber's adic spaces (Subsection \ref{subsec;trop;adic}).
   (Our tropicalizations of adic spaces are  different from Foster-Payne's adic tropicalizations, see \cite{Fos16}.)

 \subsection{Tropicalizations of Berkovich analytic spaces}\label{subsec;trop;Ber}
  We recall basics of \textit{tropicalizations} of Berkovich analytic spaces, see \cite{Gub13}, \cite{GRW16}, \cite{GRW17}, and \cite{Pay09}.
  Let $(L,v_L \colon L^{\times} \to \R)$ be a complete valuation field of height $\leq 1$. We put $\Gamma_L$ its value group.
  In this subsection, every algebraic variety is defined over $L$.
  Let $\Sigma$ be a fan in $N_\R$, and $T_{\Sigma}$ the normal toric variety over $L$ corresponding to $\Sigma$.
  
  For a cone $\sigma \in \Sigma$, 
  the \textit{tropicalization map} $$\Trop \colon O(\sigma)^{\mathrm{Ber}} \rightarrow N_{\sigma,\R} = \Hom(M \cap \sigma^{\perp} ,\R)$$
  is the proper surjective continuous map given by the restriction
  $$ \Trop(v_x):= v_x|_{M \cap \sigma^{\perp}}  \colon M \cap \sigma^{\perp}  \rightarrow  \R $$
  ($ v_x \in O(\sigma)^{\mathrm{Ber}}$).
  We define the tropicalization map $$\Trop \colon T_{\Sigma}^{\mathrm{Ber}}=\bigsqcup_{\sigma \in \Sigma }  O(\sigma)^{\mathrm{Ber}} 
  \rightarrow 
  \bigsqcup_{\sigma \in \Sigma } N_{\sigma,\R}$$
  as their direct sum, which
   is  proper, surjective, and  continuous.
  
   For a morphism $\varphi \colon X \rightarrow T_{\Sigma} $ from an algebraic variety $X$ over $L$ to $T_{\Sigma}$, the image $\Trop(\varphi(X^{\mathrm{Ber}}))$ of $X^{\mathrm{Ber}}$  is called a \textit{tropicalization} of $X^{\mathrm{Ber}}$ (or $X$).
  For simplicity, we often write $\Trop(\varphi(X))$ instead of $\Trop(\varphi(X^{\Ber}))$.
  When $\varphi \colon X \rightarrow T_{\Sigma}$ is a closed immersion,
  the tropicalization $\Trop(\varphi (X))$ is a finite union of $(\dim X)$-dimensional $ \Gamma_v$-rational polyhedra.

  For a toric morphism $\psi \colon T_{\Sigma'} \to T_{\Sigma}$,
  there exists a continous map $\Trop(T_{\Sigma'}) \to \Trop(T_{\Sigma})$
  inducing a commutative diagram
  $$\xymatrix{
  	T_{\Sigma'}^{\Ber} \ar@{>>}[d]^-{\Trop} \ar[r]^-{\psi} & T_{\Sigma}^{\Ber}\ar@{>>}[d]^-{\Trop}\\
  	\Trop(T_{\Sigma'}) \ar[r]& \Trop(T_{\Sigma}).
  }$$
  We also denote it by $\psi $.
  
  Tropicalizations do not change under base extensions, i.e.,
  for an extension $L'/L$ of complete valuation fields of height $\leq 1$, we have a commutative diagram
  $$\xymatrix{
  	X^{\Ber}_{L'}\ar@{>>}[d] \ar[r]^-{\varphi_{L'}} &T^{\Ber}_{\Sigma,L'} \ar@{>>}[d] \ar@{>>}[dr]^-{\Trop} & & \\
  	X^{\Ber} \ar[r]^-{\varphi} & T^{\Ber}_{\Sigma} \ar@{>>}[r]^-{\Trop} & \bigsqcup_{\sigma \in \Sigma } N_{\sigma,\R},
  }$$
  where $(-)_{L'}$ means the base change to $L'$.
  In particular, we have
  $$\Trop(\varphi_{L'}(X_{L'})) = \Trop(\varphi(X)) \subset \bigsqcup_{\sigma \in \Sigma } N_{\sigma,\R}.$$
  
  When $\varphi \colon X \to T_{\Sigma}$ is a closed immersion,
  we put $\Sk_{\varphi}(X) \subset X^{\Ber}$ 
  the union of the Shilov boundaries of fibers $(\Trop \circ \varphi)^{-1} (a)$ $(a \in \Trop (\varphi (X)) )$,
  a \textit{tropical skeleton} of $X$.
  When there is no confusion, we simply denote it  by $\Sk(X)$.
  
  Let 
  $\sigma \in \Sigma$ be a cone,
  $a \in \Trop(\varphi(X)) \cap N_{\sigma,\R}$ a point,
  and $L'/L$ be  an extension of complete valued field 
  such that $a \in  N_\sigma \otimes_\Z \Gamma_{L'}$.
  Then the \emph{initial degeneration} $\ini_{a} X$ (\cite[Subsection 3.5]{GRW17}) is the special fiber of a natural admissible formal model 
  of a strictly $L'$-affinoid domain $\Trop^{-1}(a) \cap X_{L'}^{\Ber}$.
  The reduction map 
  $\Trop^{-1}(a) \cap X_{L'}^{\Ber} \to \ini_a X$ 
  is a surjective, and functorial (\cite[Proposition 2.17]{GRW17}). The inverse image of generic points is the Shilov boundary (\cite[Subsection 2.13]{GRW17}). 
  
  \begin{lem}[{\cite[Lemma 4.4]{GRW17}}]\label{skeleton base change}
  For a complete valued field extension $L'/L$,
    we have a surjection 
      $ \Sk_{\varphi_{L'}} (X_{L'}) 
      \to \Sk_{\varphi} (X_{L})$.
  \end{lem}

  \begin{lem}\label{skeleton functorial}
  	Let $ X \subset \G_m^n$ be an irreducible algebraic subvariety over $L$, 
  	and $\psi \colon \G_m^n \to \G_m^r$ a morphism given by monic monomials 
  	such that the closure $X':= \overline{\psi (X)} \subset \G_m^r$ is of the same dimension as $X$.
  	Then the natural morphism $\psi \colon X^{\Ber} \to X'^{\Ber}$
  	satisfies $\psi (\Sk X ) \supset \Sk X'$.
  \end{lem}
  \begin{proof}
    By Lemma \ref{skeleton base change},
  	  we may assume that $\Gamma_L =\R$. 
  	  Since Shilov boundaries are not contained in lower dimensional subvarieties, 
  	  we have $\psi (X^{\Ber}) \supset \Sk X'$.
  	  Since reduction maps are functorial (\cite[Proposition 2.17]{GRW17}) and 
  	  $$\dim \ini_a X = \dim X = \dim X' =\dim  \ini_{\psi (a)} X'$$ 
  	  ($a \in \Trop(X)$), 
  	  we have $\Sk X   \supset \psi^{-1} (\Sk X') $.
  	  Hence the assertion holds.
  \end{proof}
  
  Let $Y$ be an algebaric variety over $L$.
  In the rest of this subsection, we assume that $L$ is trivially valued.
  \begin{dfn}\label{lem;X;circ}
  We put $Y^{\circ}  $ the subset of $Y^{\Ber}$
  consisting of
  valuations $v$ such that there exists a natural morphism $\Spec \O_v \to Y$.
  \end{dfn}
  
  When $Y$ is proper, by the valuative criterion of properness, we have $Y^{\circ} =Y^{\Ber}$.
  
  \begin{rem}\label{rem;X;circ;cpt}
  When there is a closed immersion $\varphi \colon Y \to T_{\Sigma}$, 
  we have
  $$Y^{\circ}= 
  \bigcup_{  P \in \Lambda : \text{compact}} (\Trop \circ \varphi)^{-1}(P),$$
  where $\Lambda$ is a (fixed) fan structure of $\Trop(\varphi (Y))$.
  This can be seen by the theory of toric geometry.
  In particular, since $\Trop \circ \varphi$ is proper, the subset $Y^{\circ} $ is compact. 
  \end{rem}
   
 \subsection{Tropical compactifications}\label{subsec;trop;triv;val}
  In this subsection, we briefly recall tropical compactifications introduced by Tevelev \cite{Tev07} in the trivially valued (algebraically closed) case and by Gubler \cite{Gub13} in general. 
  Let $X \subset \G_m^n=\Spec L[M]$ be a pure-dimensional closed subvariety over a complete valued field $(L,v\colon L^{\times} \to \R)$.

  In the trivially valued and reduced case, we shall use Tevelev's definition.

  \begin{dfn}\label{dfn;trop;cpt}
  	When $L$ is trivially valued and $X$ is reduced, 
  	a \emph{tropical fan} $\Sigma$ for $X$ is a fan in $N_\R$ 
  such that  the multiplication map
  $$ \G_m^n \times \overline{X} \to T_{\Sigma} $$
  is faithfully flat and $\overline{X} $ is proper, 
  where 
  the closure $\overline{X}$ is taken in the toric variety $T_{\Sigma}$ over $L$ corresponding to the fan $\Sigma$. 
  In this case, $\overline{X}$ is called a \emph{tropical compactification} of $X$.
  \end{dfn}
  
  When $v(L^{\times}) \neq \{0\}$,
  we put $\T := \Spec \O_v [M]$ a torus over the valuation ring $\O_v$, 
  and  for a  $v(L^{\times})$-admissible fan $\Sigma$ in $N_\R \times \R_{\geq 0}$, 
  put $\T_{\Sigma}$ the corresponding normal $\T$-toric scheme over $\O_v$ (\cite[7.7]{Gub13}).
  \begin{dfn}\label{2 dfn;trop;cpt}
  When $v(L^{\times}) \neq \{0\}$,
  	a \emph{$v(L^{\times})$-admissible tropical fan} for $X$ 
  	is a  $v(L^{\times})$-admissible fan $\Sigma$ in $N_\R \times \R_{\geq 0}$
  	such that 
  	$\Sigma_1 := \{ \sigma \cap (N_\R \times \{ 1 \}) \}_{\sigma \in \Sigma}$ 
  	is a polyhedral complex structure of $\Trop(X) \times \{1\}$, 
  	and 
  	there is a closed subscheme 
  	$\F \subset \T \times_{L^{\circ}} \T_{\Sigma}$ satisfying the following properties: 
  	\begin{itemize}
  		\item the second projection induces a faithfully flat morphism $f\colon \F \to \T_{\Sigma}$, and 
  		\item the morphism 
  		     $$ \Phi \colon \T \times_{L^{\circ}} \T_{\Sigma} \ni (t,x) \mapsto 
  			 (t^{-1},tx) \in 
  			 \T \times_{L^{\circ}} \T_{\Sigma}$$
  			induces an isomorphism 
  			$\G_m^n \times X \cong f^{-1}(\G_m^n)$.
  	\end{itemize}
  	In this case, we call the clousre $\X$ of $X$ in $\T_{\Sigma}$
  	a \emph{tropical compactification}.
  \end{dfn}
  
  Definition \ref{2 dfn;trop;cpt} also works in trivially valued case, 
  and when $X$ is reduced, it coincides with Definition \ref{dfn;trop;cpt} (\cite[Remark 12.2]{Gub13}).

  \begin{thm}[{\cite[Theorem 12.3]{Gub13}}]
  	When $L$ is trivially valued and $X$ is reduced (resp. $L$ is non-trivially valued), 
  there exists a tropical 
  (resp. a $v(L^{\times})$-admissible tropical) fan $\Sigma$ 
  for $X$.
  \end{thm}
  
  \begin{rem}\label{rem;trop;cpt;property}
  \begin{enumerate}
  \item (\cite[Proposition 12.5]{Gub13}) The support of a tropical fan for $X$ is  $\Trop(X)$.
  \item (\cite[Proposition 12.4]{Gub13}) Any refinement of  a tropical 
  (resp. a $v(L^{\times})$-admissible tropical) fan 
  for $X$  is also 
   a tropical 
  (resp. a $v(L^{\times})$-admissible tropical) fan for $X$. 
  \end{enumerate}
  \end{rem}

  \begin{prp}[{\cite[Proposition 2.3]{Tev07}}]\label{rem,trop,cpt,prp}
  	When $L$ is trvially valued, 
  for a fan $\Sigma$ in $N_{\R}$,
  the closure $\overline{X}$ of $X$ in $T_{\Sigma}$ is proper
  if and only if $\Trop(X)$ is contained in the support of $\Sigma$.
  \end{prp}
  
  \begin{rem}[{\cite[Proposition 12.6]{Gub13}}]
  \label{trop cpt dimensions}
  For a tropical 
  (resp. a $v(L^{\times})$-admissible tropical) 
  fan $\Sigma$ for $X$ and $\sigma \in \Sigma$, the intersection $\overline{X} \cap O(\sigma)$ is non-empty and of pure dimension $ (\dim X - \dim \sigma)$
  (resp. $ (\dim X - \dim \sigma +1 )$).
  \end{rem}
  
  \begin{rem}\label{tropical compactifications and initial degeneration}
  We assume $v(L^{\times}) \neq \{0\}$.
  	Let $\Sigma$ be a $v(L^{\times})$-admissible tropical fan for $X$,  
    $\sigma \in \Sigma$ a cone, 
  	$\sigma_1 =\sigma \cap (N_{\R }\times \{ 1\} ) \in \Sigma_1$,
  	and $x \in  \relint \sigma_1  \subset \Trop (X)$.
  	The map  
  	$$\Trop^{-1}(x) \cap X^{\Ber} \to \X \cap O(\sigma)$$
  	of taking centers is surjective. 
  	In fact, by extending the base field $L$,  this map factors through 
  	$$\Trop^{-1}(x) \cap X^{\Ber} \to \ini_x X \to \X \cap O(\sigma),$$
  	where 
  	the first map is a reduction map, which is surjective by \cite[Proposition 2.17]{GRW17}, 
  	and 
  	the second map is, 
  	by \cite[Remark 12.7]{Gub13},  
    a morphism of algebraic varieties over the residue field 
  	of the form 
  	$$\ini_x X \cong (\X \cap O(\sigma)) \times \G_m^s 
  	\twoheadrightarrow \X \cap O(\sigma)$$
  	($s \in \Z_{\geq 0} $), where the second morphism is the first projection.
  \end{rem}
   
 \subsection{Tropicalizations of Zariski-Riemann spaces}\label{subsection Tropicalizations of Zariski-Riemann spaces}
  In this subsection, we shall introduce tropicalizations of Zariski-Riemann spaces.
  Let $K$ be a trivially valued field.
  Let $x \in \G_m^n = \Spec K[M]$,
  and $\overline{\{x\}}\subset \G_m^n$ the closure.
  We put $k(x)$ the residue field. 
  For a fan structure $\Lambda$ of $\Trop ( \overline{\{x\}})$, 
  we put 
  $\overline{\{x\}}^{\Lambda} \subset T_{\Lambda}$
  the closure in the toric variety $T_{\Lambda}$ corresponding to $\Lambda$.
  By Proposition \ref{rem,trop,cpt,prp},
  the algebraic variety 
  $\overline{\{x\}}^{\Lambda} $
  is proper and intersects with any orbit $O(\lambda)$ ($\lambda \in \Lambda$).
  
  \begin{dfn}\label{dfn:trop'n:ZRspace}
  For $v \in \ZR (\Spec k(x)/K)$, we put 
  $\Trop_{\Lambda}^{\ad}(v) \in \Lambda$
  the cone such that 
  the image of the maximal ideal under 
  the natural morphism
  $$\Spec \O_v \to 
  \overline{\{x\}}^{\Lambda} \subset T_{\Lambda}$$
  is contained in the orbit $O( \Trop_{\Lambda}^{\ad} (v) )$.
  This induces a surjective map
  $$ \Trop_{\Lambda}^{\ad} \colon  \ZR (k(x)/K) \twoheadrightarrow \Lambda. $$
  By Proposition \ref{ZRhomeo},
  we have a surjective map
  $$\Trop^{\ad}  \colon \ZR(k(x)/K) \twoheadrightarrow \underleftarrow{\lim} \Lambda,$$
  called a \emph{tropicalization map} of the Zariski-Riemann space $\ZR(k(x)/K)$,
  where $\underleftarrow{\lim} \Lambda$ is the inverse limit of all
  fan structures $\Lambda$ of $\Trop(\overline{\{x\}}) $ as sets. 
  (The maps $\Trop^{\ad}_{\Lambda}$ and $\Trop^{\ad}$ are continuous with respect to some natural topology of $\Lambda$ and $\underleftarrow{\lim} \Lambda$, but we do not use them in this paper.)
  \end{dfn}
  
  Any fan structure of $\Trop(\overline{\{x\}}) $ has a refinement which is a subfan of a given fan structure $\Xi$ of $N_\R$.
  Hence we have a natural injective map
  $$\underleftarrow{\lim} \Lambda \hookrightarrow \underleftarrow{\lim} \Xi,$$
  where 
  $\Xi$ runs through all fan structures of $N_\R$.
  We identify $\underleftarrow{\lim} \Lambda $ and its image.
  
  For a valuation $v \in \ZR (k(x)/K)$,
  the composition 
  $$  M \to k(x)^{\times} \xrightarrow{v} \Gamma_v$$
  is a group homomorphism to a totally ordered group $\Gamma_v$.
  (Here for simplicity, we identify $v \in \ZR (L/K)$ and a representative.)
  By Corollary \ref{homomorphisms to totally ordered abelian groups},
  the group homomorphism $M \to \Gamma_v$ can be considered as an element in $J_r / \sim_{I_r}$ for some $r$.
  
  \begin{lem}\label{rem;trop;adic;l_1}
  	For $v \in \ZR (L/K)$, 
  	the element
  $\Trop_{\Lambda}^{\ad} (v) \in \underleftarrow{\lim} \Lambda (\subset \underleftarrow{\lim} \Xi)$
  is the image of 
  $  M \to \Gamma_v$
  under the bijection in 
  Lemma \ref{lem;trop;adic;another;express;NR}.
  \end{lem}
  \begin{proof}
  This easily follows from the definition of 
  $\Trop_{\Lambda}^{\ad}$.
  \end{proof}
   
 \subsection{Tropicalizations of adic spaces over trivially valued fields}\label{subsec;trop;adic}
  In this subsection,
  we study tropicalizations and tropical skeletons of adic spaces associated with algebraic varieties over a trivially valued field $K$.
  
  We define tropicalizations as the direct sum of tropicalizations of Zariski-Riemann spaces.
  Let $X $ be a closed subvariety of a torus $\G_m^n =\Spec K[M]$ over $K$,
  and
  $\Lambda $ a fan structure of  $\Trop(X) \subset N_{\R}$.
  For $x \in X$, 
  we define 
  $$\Trop_{\Lambda}^{\ad} \colon \ZR (k(x)/K) \to \Lambda$$
  in the same way as Definition \ref{dfn:trop'n:ZRspace}.
  
  \begin{dfn}
  We define 
  $$ \Trop_{\Lambda}^{\ad} \colon X^{\ad} 
  = \bigsqcup_{x \in X} \ZR (k(x) / K) \twoheadrightarrow \Lambda$$
  the disjoint union of 
  $\Trop_{\Lambda}^{\ad}$ on $\ZR (k(x)/K)$.
  We have  a surjective map
  $$\Trop^{\ad} \colon X^{\ad} \twoheadrightarrow \underleftarrow{\lim } \Lambda $$
  called a \emph{tropicalization map} of $X^{\ad}$,
  where $\underleftarrow{\lim} \Lambda$ is the inverse limit of all
  fan structures $\Lambda$ of $\Trop(X) $ as sets. 
  \end{dfn}
  
  Let $Y$ be a closed subvariety  of a toric variety $T_\Sigma$ over $K$,
  and
  $\Xi $ a  fan structure of  $\Trop(Y) $.
  For a cone $\sigma \in \Sigma$,
  we put 
  $$\Xi \cap \Trop (O(\sigma)) 
  := \{ \xi \cap \Trop (O(\sigma)) \mid \xi \in \Xi , \ 
  \relint \xi \subset \Trop O(\sigma) \},$$
  a fan in $\Trop (O(\sigma))$.
  We identify $\Xi \cap \Trop (O ( \sigma))$
  with 
  $$\{\xi \in \Xi \mid \relint \xi \subset \Trop O(\sigma) \}$$
  by taking closures in $\Trop (T_{\Sigma})$.
  
  \begin{dfn}
  We define 
  $$\Trop^{\ad}_{\Xi} \colon Y^{\ad} 
  = \bigsqcup_{\sigma \in \Sigma}  (Y \cap O(\sigma))^{\ad}
  \twoheadrightarrow
   \bigsqcup_{\sigma \in \Sigma}  \Xi \cap \Trop (O (\sigma))
  = \Xi$$
  the disjoint union of $\Trop^{\ad}_{\Xi \cap \Trop(O(\sigma))} $.
  We have a surjective map
  $$\Trop^{\ad} \colon Y^{\ad} \twoheadrightarrow \underleftarrow{\lim } \Xi $$
  called a \emph{tropicalization map} of $Y^{\ad}$,
  where $\underleftarrow{\lim} \Xi$ is the inverse limit of all
  fan structures $\Xi$ of $\Trop(Y) $ as sets. 
  \end{dfn}
  
  We put 
  $$\Trop^{\ad}(Y^{\ad}) \cap \Trop (O(\sigma)) := 
  \underleftarrow{\lim }\Xi \cap \Trop (O(\sigma)) \cong \Trop^{\ad} ((Y \cap O(\sigma))^{\ad}) .
  $$
  We have 
  $$\Trop^{\ad}(Y^{\ad}) 
  =\bigsqcup_{\sigma \in \Sigma } \Trop^{\ad}(Y^{\ad}) \cap \Trop (O(\sigma)). $$
  
  By Lemma \ref{lem;trop;adic;another;express;NR}, 
  there is a natural map 
  $$\Trop(Y) \to \Trop^{\ad}(Y^{\ad})$$
  whose image consists of points of height $0$ and $1$.
  By Lemma \ref{rem;trop;adic;l_1},
  we have a commutative diagram
  $$\xymatrix{
  Y^{\Ber} \ar[r] \ar[d]^-{\Trop} &
  Y^{\ad} \ar[d]^-{\Trop^{\ad}} \\
  \Trop(Y) \ar[r] &
  \Trop (Y^{\ad}).
  }$$
   
\section{Tropical cohomology}\label{sec;trocoho}
   In this section, we recall \textit{tropical cohomology}.
   Recall that 
   $M$ is a free $\Z$-module of finite rank and $N:=\Hom(M,\Z)$.
   
 \subsection{Tropical cohomology of polyhedral complexes}\label{subsec;trop;cohomology;fan}
  We recall tropical cohomology introduced by Itenberg-Katzarkov-Mikhalkin-Zharkov \cite{IKMZ19}.
  See also \cite{JSS19}.
  Let $T_{\Sigma}$ be the toric variety corresponding to a fan $\Sigma$ in $N_\R$.
  Let $\Lambda $ be a polyhedral complex in $\Trop(T_{\Sigma})$.
  Recall that for a  subset $B \subset \Trop(T_{\Sigma})$,
  we put
  $\Lambda \cap B := \{P \cap B\}_{P \in \Lambda }.$
  For $P \in \Lambda$,
  we put $\sigma_P \in \Sigma$ the cone such that $\relint(P) \subset N_{\sigma_P,\R}$.
  
  Let $p \geq 0$ be a non-negative integer.
  For $P\in \Lambda$, we put  
  $$\Tan_\Q P := \Tan_\Q (P \cap N_{\sigma_P,\Q})
     := \sum_{x,y\in P\cap  N_{\sigma_P,\Q}} \Q ( x-y ) \subset N_{\sigma_P,\Q}$$
  a $\Q$-linear subspace of $N_{\sigma_P,\Q}$,
  $$F_p(P,\Lambda):= \sum_{\substack{P' \in \Lambda\cap N_{\sigma_P,\R}\\ \relint(P) \subset P'}} \bigwedge^p \Tan_\Q (P') \subset \bigwedge^p N_{\sigma_P,\Q},$$ 
  and  
  $$F^p(P,\Lambda):= \bigwedge^p  (M \cap \sigma_P^{\perp} )_\Q 
  \big/ 
  \big\{f \in \bigwedge^p (M \cap \sigma_P^{\perp})_\Q 
  \big| \alpha(f)=0 \ (\alpha \in F_p(P,\Lambda))
  \big\} ,$$
  where we identify $$\bigwedge^p N_{\sigma_P,\Q} \cong \Hom \big(\bigwedge^p  (M \cap \sigma_P^{\perp} )_\Q, \Q \big).$$
  We have
  $$F^p(P,\Lambda) \cong \Hom (F_p(P,\Lambda),\Q).$$
  Since $F_p(P,\Lambda)$ (resp.  $F^p(P,\Lambda)$) depends only on the support $\lvert \Lambda \rvert$,
  we sometimes write $F_p(P,\lvert \Lambda \rvert)$ (resp.\ $F^p(P,\lvert \Lambda \rvert)$) instead of $F_p(P, \Lambda )$ (resp.\ $F_p(P,\Lambda)$).
  When there is no confusion, we simply write $F_p(P)$ (resp.\ $F^p(P)$).
  
  \begin{rem}\label{remdefofF_Pmap}
  Let $P_1,P_2 \in \Lambda$ with $P_2 \subset P_1$.
  Then we have $\sigma_{P_1} \subset \sigma_{P_2}$.
  \begin{itemize}
  \item When $\sigma_{P_1} =\sigma_{P_2}$, there exists a natural injection
  $$ i_{P_2 \subset P_1} \colon F_p(P_1) \hookrightarrow F_p(P_2).$$
  \item When $P_2 = P_1 \cap \overline{ N_{\sigma_{P_2},\R }},$
  the natural projection $ N_{\sigma_{P_1},\R } \twoheadrightarrow N_{\sigma_{P_2},\R }$ induces
  a morphism $$ i_{P_2 \subset P_1} \colon F_p(P_1) \twoheadrightarrow F_p(P_2).$$
  \item  In general, we put
  $$ i_{P_2 \subset P_1} := i_{P_2 \subset Q} \circ i_{Q \subset P_1} \colon F_p(P_1) \twoheadrightarrow F_p(Q) \hookrightarrow F_p(P_2),$$
  where $Q:= P_1 \cap \overline{N_{\sigma_{P_2},\R }}$.
  \end{itemize}
  \end{rem}
  
  Let $B \subset \lvert \Lambda \rvert$ be a locally closed subset.
  \begin{dfn}\label{deftrocoho}
  \begin{enumerate}
  \item For every cone $P \in \Lambda$,
  we put  $C_q(B\cap P)$ the free $\Q$-vector space generated by continuous maps $\gamma \colon \Delta^q \rightarrow B \cap P$ from the standard $q$-simplex $\Delta^q$.
  We put
  $$C_{p,q}(B,\Lambda) :=\bigoplus_{P \in \Lambda} F_p(P,\Lambda) \otimes_\Q C_q(B\cap P) / \text{(equivalence relation)},$$
  where the equivalence relation is generated by 
  $$ \alpha_{P_1} \otimes \gamma - i_{P_2 \subset P_1}(\alpha_{P_1}) \otimes \gamma$$
  for $P_1,P_2 \in \Lambda $ with $P_2 \subset P_1$, 
  $\alpha_{P_1} \in F_p(P_1,\Lambda)$, 
  and $\gamma \colon  \Delta^q \to B \cap P_2 \subset B \cap P_1$.
  We call its elements \emph{tropical $(p,q)$-chains}. \label{enu:def:tro:coho:fan}
  \item  For $\gamma \in C_q(B \cap P)$, we denote the usual boundary by $\partial(\gamma):=\sum_{i=0}^q (-1)^i \gamma^i$.
  For each $v \otimes \gamma \in F_p(P,\Lambda) \otimes C_q(B \cap P)$, we put
  $$\partial(v \otimes \gamma):= (-1)^p \sum_{i=0}^q (-1)^i v \otimes \gamma^i \in C_{p,q-1}(B,\Lambda).$$
  We obtain complexes $(C_{p,*}(B,\Lambda),\partial)$.
  \item We define the \emph{tropical homology groups} to be
  $$H_{p,q}^{\Trop}(B,\Lambda):= H_q(C_{p,*}(B,\Lambda),\partial).$$
  We put $(C^{p,*} (B, \Lambda), \delta)$ the dual complex of $(C_{p,*}(B,\Lambda),\partial)$.
  We call its cohomology groups
  $$H_{\Trop}^{p,q}(B,\Lambda) := H^q(C^{p,*}(B,\Lambda),\delta)$$
  the \emph{tropical cohomology groups} of $(B,\Lambda)$.
  \end{enumerate}
  \end{dfn}
  
  \begin{rem}
  For a refinement $\Lambda'$ of $\Lambda$ and any $p,q$, the natural map
  $$H_{\Trop}^{p,q}(B,\Lambda) \to H_{\Trop}^{p,q}(B,\Lambda')$$
  is an isomorphism. This follows from \cite[Proposition 2.8]{MZ13} (see \cite[Section 3]{JSS19} for a proof).
  We write $H_{\Trop}^{p,q}(\lvert \Lambda \rvert) :=H_{\Trop}^{p,q}(\lvert \Lambda \rvert,\Lambda)$ for short.
  \end{rem}
  
  \begin{rem}\label{remark stalk tropical variety}
  By \cite[Proposition 3.15]{JSS19},
  tropical cohomology group of $\lvert \Lambda \rvert$ is isomorphic to 
  sheaf cohomology of a sheaf 
  $\F_{\lvert \Lambda \rvert}^p:=\Ker(\C_{\lvert \Lambda \rvert}^{p,0}\to \C_{\lvert \Lambda \rvert}^{p,1})$ on ${\lvert \Lambda \rvert}$, 
  where $\C_{\lvert \Lambda \rvert}^{p,*}$ is the complex of 
  the sheafifications of the presheaves of tropical $(p,*)$-cochains.
  By \cite[Proposition 3.11]{JSS19},
  for $P\in \Lambda$ and $x\in \relint P$, 
  we have $\F_{{\lvert \Lambda \rvert},x}^p \cong F^p(P)$.
  \end{rem}
  
  For a subset $D \subset B$,
  we put
  $$C^{p,q}_{D}(B,\Lambda):= \Ker(C^{p,q}(B,\Lambda) \to C^{p,q}(B\setminus D,\Lambda)).$$
  We put $H_{\Trop,D}^{p,q}(B,\Lambda)$ its cohomology group.
   
 \subsection{Tropical cohomology of algebraic varieties}\label{subsec:trop:coh:hom}
  In this subsection, we recall tropical cohomology of an algebraic variety $X$ 
  over a complete valuation field $(L,v_L \colon L^{\times} \to \R)$.
  This was introduced by Jell \cite[Section 8]{Jel22}, 
  and is based on tropical charts, 
  which were given by Chambert-Loir-Ducros \cite{CLD12}, Gubler \cite{Gub16}, and Jell \cite{Jel16}.
  We put $\Gamma_v$ the value group of $(L,v_L)$.
  We fix a toric structure of each affine space.
  
  We define a sheaf $\mathscr{C}_X^{p,q}$  on $X^{\Ber}$ by, for each open subset $V \subset X^{\Ber}$, putting
  $\mathscr{C}_X^{p,q}(V)$ the set of equivalence classes of $(U_i,V_i,\varphi_i,\Lambda_i,\alpha_i)_i$ consisting of
  \begin{itemize}
  \item a Zariski open covering $  \{U_i\}_i $ of $X$ and an open covering $\{V_i \}_i $ of $V$,
  \item closed immersions $ \varphi_i \colon U_i \rightarrow \A^{n_i}$   such that $V_i =(\Trop \circ \varphi_i)^{-1}(\Omega_i) \subset U_i^{\Ber}$ for some open subsets $\Omega_i \subset \Trop(\varphi_i(U_i))$,
  \item $ \Gamma_v$-rational polyhedral complex structures  $  \Lambda_i $ of $\Trop(\varphi_i(U_i))$, and
  \item $\alpha_i \in C^{p,q}(\Trop(\varphi_i(V_i)),\Lambda_i)$
  \end{itemize}
  satisfying the following:
  for any $i,j$,
  there exists
  \begin{itemize}
  \item a Zariski open covering $\{U_{i,j,k}\}_k$ of $U_{i}\cap U_{j}$,
  \item closed immersions $\varphi_{i,j,k} \colon U_{i,j,k} \rightarrow \A^{n_{i,j,k}}$,
  \item toric morphisms $ \Psi_{(i,j,k),l} \colon \A^{n_{i,j,k}} \to \A^{n_l}$ ($l\in \{i,j\}$), and
  \item $ \Gamma_v$-rational polyhedral complex  structures $\Lambda_{i,j,k}$ of $\Trop(\varphi_{i,j,k}(U_{i,j,k}))$
  \end{itemize}
  such that
  \begin{itemize}
  \item for each $P \in \Lambda_{i,j,k}$ and $l\in \{i,j\}$,
  there exists $Q \in \Lambda_l $ containing $\Psi_{(i,j,k),l}(P )$,
  \item for $i,j,k$ and $l\in \{i,j\}$, the diagram 
  $$\xymatrix{
  & U_{i,j,k} \ar[d]^-{\varphi_{i,j,k}} \ar[dr]^-{\varphi_l} \\
   &\A^{n_{i,j,k}} \ar[r]_-{\Psi_{(i,j,k),l}} & \A^{n_l} \\
  }$$
  is commutative, and
  \item for $i,j,k$, we have
  \begin{align*}
  	\Psi_{(i,j,k),i}^* \alpha_i |_{\Trop(\varphi_{i,j,k}(V_i \cap V_j \cap U_{i,j,k}^{\Ber}))} &= \Psi_{(i,j,k),j}^* \alpha_j |_{\Trop(\varphi_{i,j,k}(V_i \cap V_j \cap U_{i,j,k}^{\Ber}))}\\
  												   &\in C^{p,q}(\Trop(\varphi_{i,j,k}(V_i \cap V_j \cap U_{i,j,k}^{\Ber})), \Lambda_{i,j,k}).
  \end{align*}
  \end{itemize}
  The equivalence relation is generated by
  $$(U_i,V_i,\varphi_i,\Lambda_i,\alpha_i)_i \sim (U'_j,V'_j,\varphi'_j,\Lambda'_j,\alpha'_j)_j $$
  satisfying the following:
  for each $j$, there exist $i(j)$ and a toric morphism $\psi_j \colon \A^{n'_j} \to \A^{n_{i(j)}}$
  such that
  \begin{itemize}
  \item $U_{i(j)}$ (resp.\ $V_{i(j)}$) contains $U'_j$ (resp.\ $V'_j$),
  \item the diagram
  $$\xymatrix{
  	U_{i(j)} \ar[r]^-{\varphi_{i(j)}} & \A^{n_{i(j)}}\\
  	U'_j \ar[u] \ar[r]^-{\varphi'_j} & \A^{n'_j} \ar[u]_-{\psi_j}
  }$$
  is commutative,
  \item for $P \in \Lambda'_j$, there exists $Q \in \Lambda_{i(j)}$ containing $\psi_j(P)$, and
  \item  $\alpha'_j = \psi_j^* \alpha_i|_{\Trop(\varphi'_j(V'_j))} $.
  \end{itemize}
  The coboundary map $\delta$ in Definition \ref{deftrocoho} induces a complex
  $$ \C_X^{p,0} \rightarrow \C_X^{p,1} \rightarrow \C_X^{p,2} \rightarrow \dots$$
  of sheaves on $X^{\Ber}$, which is exact by \cite[Proposition 3.15]{JSS19}.
  The cohomology groups
  $$H_{\Trop}^{p,q}(X):= \Ker(\C_X^{p,q}(X^{\Ber})\rightarrow \C_X^{p,q+1}(X^{\Ber}))/\Im(\C_X^{p,q-1}(X^{\Ber})\rightarrow \C_X^{p,q}(X^{\Ber}))$$
  are  called  the \textit{tropical cohomology groups} of $X$.
  
  We also use another expression (\cite{Jel16}, \cite{Jel22}) of tropical cohomology by embeddings of $X$ to toric varieties. 
  When $X$ has a closed immersion to a toric variety, 
  there are many closed immersions of $X$ to toric varieties \cite[Theorem 1.2]{FGP14}.
  In this case, 
  we define a sheaf $\C_{T,X}^{p,q}$ on $X^{\Ber}$ in a similar way to $\C_X^{p,q}$
  but using closed immersions $X \to T_{\Sigma_i}$ to a toric variety $T_{\Sigma_i}$ instead of pairs of open subvarieties $U_i \subset X$ and closed immersions $U_i \to \A^{n_i}$.
  There exists a natural isomorphism $\C_X^{p,q}\cong \C_{T,X}^{p,q}$ (\cite[Remark 8.2]{Jel22}).
  In particular, tropical cohomology $H_{\Trop}^{p,q}(X)$ is isomorphic to
  $$ H_{T,\Trop}^{p,q}(X)
  := \Ker(\C_{T,X}^{p,q}(X^{\Ber})\rightarrow \C_{T,X}^{p,q+1}(X^{\Ber}))/\Im(\C_{T,X}^{p,q-1}(X^{\Ber})\rightarrow \C_{T,X}^{p,q}(X^{\Ber})).$$ 
  When there is no confusion, we identify $\C_X^{p,q}\cong \C_{T,X}^{p,q}$ and $H_{\Trop}^{p,q}(X) \cong H_{T,\Trop}^{p,q}(X)$.
  
  The sheaf $\C^{p,q}_X $ is not flabby. 
  This is because
  for an open subset $\iota \colon U \hookrightarrow X$, a point $x \in X^{\Ber} \setminus U^{\Ber}$, 
  and $\alpha \in \C_U^{p,q}(U^{\Ber})$, 
  the element
  $\iota_* \alpha_x  $ 
  in the stalk $ (\iota_* \C_U^{p,q})_x$
  does not necessarily 
  come from finitely many tropical charts, and hence $\alpha$ does not necessarily extends to $x\in X^{\Ber}$.
  When $\alpha$ is given by finitely many tropical charts, this problem does not happen.
  In particular, we have the following.
  
  \begin{lem}\label{tropical p,q cochains sheaf c-soft}
  	The sheaf $\C^{p,q}_X $ is c-soft, i.e., for any compact subset $K\subset X^{\Ber}$, 
  	the natural map $$\Gamma (X^{\Ber}; \C^{p,q}_X ) \to \Gamma (K; \C^{p,q}_X )$$
  	is surjective.
  \end{lem}
  \begin{proof}
  	This is \cite[Lemma 8.10]{Jel22} when $X$ has an embedding to a toric variety, or $\C^{p,q}_X$ is with $\R$-coefficients. We shall see that a similar proof works in general. 
  	By \cite[Proposition 2.5.1]{KS90}, since $X^{\Ber}$ is Hausdorff, every element of $\Gamma (K; \C^{p,q}_X )$ is the restriction of an section on an open neighborhood of $K$.
  	Hence it suffices to show that 
  	for an open subset $V\subset X^{\Ber}$,
  	a section 
  	$g:=(U_i,V_i,\varphi_i,\Lambda_i,\alpha_i)_{i\in I} \in \C^{p,q}_X (V)$ with $\# I <\infty$ 
  	extends to $\C^{p,q}_X (X^{\Ber})$.
  	We consider extension by $0$, i.e., extensions by taking $0$ for tropical $(p,q)$-chains whose supports are not contained in tropicalizations of any subsets of $V$.
  	It suffices to show that 
    for 
     an affine open subvariety $U\subset X$,
     the section 
  	$g|_{U^{\Ber} \cap V}$ can extend by $0$ to $U^{\Ber}$.
  	(Then glueing is obvious.)
  	Since 
  	$g|_{U^{\Ber}\cap V} \in \C^{p,q}_{U}(U^{\Ber}\cap V)$ is given by finitely many tropical charts, it is given by a single closed immersion $ \varphi \colon U\to \A^r$, 
  	hence it can extend to $U^{\Ber}$ by $0$.
  \end{proof}
  
  By \cite[Proposition 3.15]{JSS19} and Lemma \ref{tropical p,q cochains sheaf c-soft}, tropical cohomology is the sheaf cohomology groups of a sheaf $\F_X^p:=\Ker(\C_X^{p,0}\to \C_X^{p,1})$ on $X^{\Ber}$. 
  For a closed subscheme $Z \subset X$,
  we put $$\C_{Z \subset X}^{p,q}
  :=\Ker(\C_X^{p,q}\rightarrow
  \pi_* \pi^* \C_X^{p,q} ),$$
  where $\pi \colon X^{\Ber}\setminus Z^{\Ber} \to X^{\Ber}$ is the inclusion,
  and
  $$H_{\Trop,Z}^{p,q}(X):= \Ker(\C_{Z \subset X }^{p,q}(X)\rightarrow \C_{Z \subset X}^{p,q+1}(X))/\Im(\C_{Z \subset X}^{p,q-1}(X)\rightarrow \C_{Z \subset X}^{p,q}(X)).$$

  In the rest of this subsection, we assume that $L$ is trivially valued.
  Recall that 
  $X^{\circ}  \subset X^{\Ber}$ is the subset
  consisting of
  valuations $v$ having natural morphisms $\Spec \O_v \to X$.
  We put $\C_{X^{\circ}}^{p,q}:= \C_{X}^{p,q}|_{X^{\circ}}$.
  For a closed subscheme $Z \subset X$,
  we also define sheaves 
  $$\C_{Z^{\circ} \subset X^{\circ}}^{p,q} 
  := \Ker (
  \C_{X^{\circ}}^{p,q} \to \pi^{\circ}_{*} \pi^{\circ,*} \C_{X^{\circ}}^{p,q} 
   ), $$
  where $\pi^{\circ} \colon X^{\circ} \setminus Z^{\circ} \to X^{\circ}$ is the inclusion.
  We define
  $ H_{\Trop}^{p,q}(X^{\circ})$ as the $q$-th cohomology group 
  of $\C_{X^{\circ}}^{p,*} (X^{\circ})$, which is the $q$-th cohomology group of 
  $\F_{X^\circ}^p := \F_X^p |_{X^{\circ}}.$
  We also define $ H_{\Trop,Z^{\circ}}^{p,q}(X^{\circ})$ similarly.
  
  When $X$ has a closed immersion to a toric variety, 
  since $X^{\circ}$ is compact (Remark \ref{rem;X;circ;cpt}),
  every sections of $\C_{X^{\circ}}^{p,*} (X^{\circ})$ comes from a single closed immersion to a toric variety. 
  In particular,
  since tropical cohomology of a suitable subset of a tropical variety 
  is 
  isomorphic to cohomology of sections of the sheafifications of the presheaves of tropical cochains (see \cite[Lemma 3.14]{JSS19}),
   we have 
  $$ H^{p,q}_{\Trop,T}(X^{\circ}) \cong \underrightarrow{\lim}_{\varphi} H^{p,q}_{\Trop}(\Trop(\varphi(X^{\circ}))) ,$$
  where the limit is indexed by 
  the category whose objects are
   closed immersions $\varphi \colon X \to T_{\Sigma}$ to toric varieties  
   and morphisms from $\varphi \colon X \to T_{\Sigma}$ to $\varphi' \colon X \to T_{\Sigma'}$
   are toric morphisms $\psi \colon T_{\Sigma' } \to T_{\Sigma} $ such that 
   $\psi \circ \varphi' = \varphi $.
  For a closed subscheme $Z \subset X$,
  we also have
  $$ H^{p,q}_{\Trop,T,Z^{\circ}}(X^{\circ})
  \cong \underrightarrow{\lim}_{\varphi}
  H^{p,q}_{\Trop,\Trop(\varphi(Z^{\circ}))}(\Trop(\varphi(X^{\circ}))).$$
  
  \begin{lem}\label{lem;X;circ;trocoho}
  We assume that $X$ has a closed immersion to a toric variety.
  Then we have a natural isomorphism
  $$ H^{p,q}_{\Trop,T}(X)
  \cong H^{p,q}_{\Trop,T}(X^{\circ}).$$
  \end{lem}
  \begin{proof}
  Note that 
  for a closed immersion $\varphi$ of $X$ into a toric variety, a fan structure $\Lambda$ of $\Trop(\varphi(X))$, and a subset $B \subset \Trop(\varphi(X))$ 
  such that 
  $\Trop(\varphi(X^{\circ})) \subset B$ 
  and there is 
  a strong deformation retraction
  $ \psi \colon B\times [0,1] \to B $ of $B $ onto $\Trop(\varphi(X^{\circ}))$ preserving the fan structure,
  i.e.,
  $$\psi((P \cap B,[0,1])) \subset P \cap B  \quad (P \in \Lambda),$$
  we have
  \begin{align}\label{equation retraction}
   H^{p,q}_{\Trop}(\Trop(\varphi(X^{\circ})), \Lambda ) 
  \cong  H^{p,q}_{\Trop}(B,\Lambda).
  \end{align}
  By Remark \ref{rem;X;circ;cpt},
  surjectivity follows from (\ref{equation retraction}) for $B =\Trop(\varphi(X))$.
  There is a sequence of compact subsets $X_1 := X^{\circ} \subset X_2 \subset \dots $ of $X^{\Ber}$
  such that $X^{\Ber} = \bigcup_i X_i$
  and $B= \Trop (\varphi (X_i))$ satisfies the above condition. 
  Let 
  $\alpha \in \C_{T,X}^{p,q}(X^{\Ber})$  be a cocycle.
  Then the restriction $\alpha|_{X_j}$ comes from a single closed immersion $\varphi_{\alpha,j} \colon X \to T_{\Sigma_{\alpha,j}}$ to a toric variety $T_{\Sigma_{\alpha,j}}$.  
  We may assume that there are toric morphisms $\psi_{i,i+1} \colon T_{\Sigma_{\alpha,i+1}} \to T_{\Sigma_{\alpha, i}}$ such that $\psi_{i,i+1} \circ \varphi_{\alpha,i+1} = \varphi_{\alpha,i}$.
  Then 
  since by (\ref{equation retraction}), 
  we have 
  $$
   H^{p,q}_{\Trop, \Trop (\varphi_{\alpha,j} (X_{j}) \setminus \Trop (\varphi_{\alpha,j} (X_{j-1}))}(\Trop(\varphi_{\alpha,j} (X_j)), \Lambda ) =0,$$
  when the restriction of $\alpha$ to $X^{\circ}$
  is the coboundary of some cochain, 
  the cycle $\alpha$
  is a coboundary of some cochain on $X^{\Ber}$.
  Hence injectivity holds.
  \end{proof}

  \begin{cor}\label{cor;X;circ;trocoho;T;Z}
  For a closed subscheme $Z \subset X$,
  we have
  $$ H^{p,q}_{\Trop,T,Z}(X) \cong H^{p,q}_{\Trop,T,Z^{\circ}}(X^{\circ}).$$
  \end{cor}
  \begin{proof}
  In a similar way to Lemma \ref{lem;X;circ;trocoho},
  the pull back map
  $$H^{q}(\Gamma(X^{\circ} \setminus Z^{\circ}, \C^{p,*}_{T,X^{\circ}})) \to
  H^{p,q}_{\Trop,T}((X \setminus Z)^{\circ})$$
  is an isomorphism. Hence the assertion holds.
  \end{proof}
  
  \begin{dfn}\label{stalks of F^p at higher valuations}
  	For $v \in X^{\ad}$, 
  we put
  $$\F^p_{X,v}
      := \underrightarrow{\lim}_{
  		\varphi \colon U \to \A^{n_\varphi} 
  		\ }
          F^p (\Trop^{\ad}_{\Lambda_\varphi}(\varphi(v)), \Trop (\varphi (U))),
  $$
  where 
  the limit is indexed by 
  the category 
  whose objects are 
  closed immersions $\varphi \colon U \to \A^{n_{\varphi}}$ 
  from affine open subvarieties $U$ of $X$ with $v \in U^{\ad}$ 
  to affine spaces endowed with toric structures,
  and morphisms are toric morphisms (in the opposito direction) of affine spaces compatible with closed immersions,
  and 
  $\Lambda_\varphi$ is a polyhedral complex structures of $\Trop(\varphi(U))$.
  \end{dfn}

  Obviously,
  for $x \in X^{\Ber}$, this ``stalk'' at the image under  the natural map $X^{\Ber} \to X^{\ad}$ coincides with the usual stalk $\F_{X,x}^p$ 
  of the sheaf $\F_X^p$.
   
\section{Stalks and tropical Milnor K-groups}\label{sec stalks, trop K}
   In this section,
   we shall study ``stalks'' of the sheaf $\F_X^p$.
   We shall also define and study a tropical analog of rational Milnor $K$-groups.
   Recall that $M$ is a free $\Z$-module of finite rank $n$.
   Let $K$ be a trivially valued field.
   Let $X$ be an algebraic variety over $K$. 

 \subsection{Stalks}
  We shall study the ``stalk'' $ \F^p_{X,v}$ at  $v \in X^{\ad}$ (Definition \ref{stalks of F^p at higher valuations}).
  We first consider the stalk when $v$ is a trivial valuation, i.e., $\Gamma_v =\{0\}$.
  
  Let $L/K$ be an extension of fields,
  and $\varphi \colon \Spec L \to \G_m^n = \Spec K[M]$ 
  be a morphism over $K$.
  We put $\overline{\varphi(\Spec L)} \subset \G_m^n$ the closure.
  The morphism $\varphi$ gives   a group homomorphism $\varphi \colon M \to L^{\times}$.
  We denote the wedge product of $\varphi \otimes_\Z \Q$ by 
  $\varphi \colon \wedge^p M_\Q \to \wedge^p (L^{\times})_\Q.$
  A valuation $v \in \ZR (L/K)$ is also a group homomorphism 
  $v \colon L^{\times} \to \Gamma_v$.
  We denote 
  the wedge product of 
  $v \otimes_\Z \Q $ by 
  $ \wedge^p v \colon \wedge^p (L^{\times})_\Q \to \wedge^p \Gamma_{v,\Q} .$
  
  \begin{lem}\label{lemFpexpressvalua}
  We have
  $$F^p(0 , \Trop ( \overline{\varphi(\Spec L)})) = \wedge^p  M_\Q/J_M,$$
  where $J_M $ is the $\Q$-vector subspace generated by $f \in \wedge^p M_\Q$
  such that $\wedge^p v (\varphi(f))=0$ for $v \in \ZR(L/K)$.
  
  Moreover, when $L = k(\varphi(\Spec L))$, 
  where $k(\varphi(\Spec L))$ is the residue field of the structure sheaf at $\varphi(\Spec L) \in \G_m^n$,
  we have 
  $$F^p(0 , \Trop ( \overline{\varphi(\Spec L)})) = \wedge^p  M_\Q/J'_M,$$
  where  $J_M'$ is the $\Q$-vector subspace generated by $f \in \wedge^p M_\Q$
  such that $\wedge^p v (\varphi(f))=0$ for $v \in \ZR(L/K)$ with
  $\Gamma_v \cong \Z^p ,$ 
  where $\Z^p$ is equipped with the lexicographic order.
  \end{lem}
  \begin{proof}
    The first assertion follows from Remark \ref{rem;trop;adic;l_r;P;span;same}
    and Lemma \ref{rem;trop;adic;l_1}.
    The second assertion follows from vertical generalizations (Subsection \ref{subsec;val})
    and the  fact that 
    for $w \in \ZR(L/K)$ with $\height (w) = \trdeg (L/K)$, 
    by \cite[Chapter 6. Section 10.3. Corollary 3 of Theorem 1]{Bou72},
    we have $\Gamma_w = \Z^{\trdeg (L/K)}$.
  \end{proof}
  
  Of course, the projection
  $$\wedge^p M_\Q \twoheadrightarrow F^p(0 , \Trop ( \overline{\varphi(\Spec L)})) $$
  factors through
  $$\wedge^p M_\Q \twoheadrightarrow \wedge^p \varphi(M)_\Q \twoheadrightarrow F^p(0 , \Trop ( \overline{\varphi(\Spec L)})) .$$
  
  \begin{rem}\label{remark on different tropicalizations}
    Let $L'/L$ be an extension of fields, and 
     a commutative diagram  
  $$\xymatrix{
  & \Spec L \ar[r]^-{\varphi_1} & \G_m^r \\
  & \Spec L' \ar[r]_-{\varphi_2} \ar[u]  & \G_m^l \ar[u]_-{\psi}
  }$$
  over $K$ with
  $\psi \colon \G_m^l \to \G_m^r$ given by monic monomials.
  This diagram induces a surjective map
  $$\Trop ( \overline{\varphi_2(\Spec L')})\twoheadrightarrow \Trop ( \overline{\varphi_1(\Spec L)})$$ 
  and hence induces an injective pull-back map
  \begin{align}
  F^p(0, \Trop ( \overline{\varphi_1(\Spec L)})) \hookrightarrow F^p(0, \Trop ( \overline{\varphi_2(\Spec L')})). \label{map:Fp:pull-back}
  \end{align}
  \end{rem}
  
  \begin{dfn}\label{tropical Milnor K}
  Let $p \geq 0$ be a non-negative integer.
  We put 
  $$  K_T^p(L/K):=\lim_{\substack{\rightarrow \\ \varphi\colon \Spec L \rightarrow \G_m^r }} F^p(0, \Trop( \overline{\varphi(\Spec L)} ) ) , $$
  where the limit is indexed by 
  the  category whose objects are 
  $K$-morphisms  $\varphi \colon \Spec L \to \G_m^r$  to tori  of arbitrary dimensions 
  and morphisms are $K$-morphisms (in the opposite direction) of tori given by monic monomials compatible with morphisms from $\Spec L$.
  We call it the $p$-th \emph{tropical Milnor $K$-group}.
  When there is no confusion, we put $K_T^p(L):=K_T^p(L/K)$.
  \end{dfn}
  
  \begin{rem}
    By \cite[Proposition 3.11]{JSS19} (Remark \ref{remark stalk tropical variety}), for a valuation $v\in X^{\Ber}$ with $\Gamma_v=\{0\}$, 
    we have a natural isomorphism
    $$\F^p_{X,v}\cong K_T^p(k(\supp(v))/K),$$
    where $k(\supp(v))$ is the residue field at $\supp(v)\in X$.
  \end{rem}
  
  There is a natural surjective map 
  $\wedge^p (L^{\times})_\Q \twoheadrightarrow K_T^p(L/K).$
  Moreover, by Lemma \ref{lemFpexpressvalua}, we have the following.
  \begin{cor}\label{lemK_Tval}
  We have
  $$K_T^p(L/K) \cong  \wedge^p ( L^{\times})_\Q / J \cong  \wedge^p ( L^{\times})_\Q / J',$$
  where $J$  (resp.\ $J'$) is the $\Q$-vector subspace generated by $f \in \wedge^p ( L^{\times})_\Q$
  such that $\wedge^p v (f)=0 $ for $v \in \ZR(L/K)$
  (resp.\ for $v \in \ZR(L/K)$ with $\Gamma_v \cong \Z^p$, where $\Z^p$ is equipped with the lexicographic order).
  \end{cor}

  \begin{exam}
  \begin{itemize}
    \item For any $L /K$, we have $K^0_T (L/K)=\Q$ by definition.
    \item For any $L/K$, by Corollary  \ref{lemK_Tval}, we have $K^1_T (L/K) = (L^{\times } /(L \cap  K^{\alg})^{\times})_\Q.$
    \item For any $L/K$ and any $p \geq \trdeg(L/K) +1$, we have $K_T^p (L/K)=0$.
  \end{itemize}
  \end{exam}
  
  The wedge product induces a multiplication
  $$K_T^p (L/K) \times K_T^q (L/K) \to K_T^{p+q} (L/K).$$

  Next, we consider the ``stalk'' $ \F^p_{X,v}$ at a general valuation $v \in X^{\ad}$ (Definition \ref{stalks of F^p at higher valuations}).
  Let 
  $a_1,\dots, a_{ \rank v} \in k(\supp(v))^{\times}$ 
  be elements such that $v(a_1), \dots, v(a_{\rank v})$ form a basis of a $\Q$-vector space $\Q \cdot \Gamma_v$.
  Then we have a decomposition 
  \begin{align}
  (k(\supp(v))^{\times})_\Q =  \Q\langle a_j\rangle_{j=1}^{\rank v} \oplus \Ker (v\otimes \Q \colon (k(\supp(v))^{\times})_\Q \to \Q \cdot \Gamma_v)\label{decomp for stalk of F^p},
  \end{align}
    where $\Q\langle a_j\rangle_{j=1}^{\rank v}$ is a $\Q$-vector space with basis $a_1,\dots, a_{ \rank v} .$
  
  \begin{lem}\label{stalk of F^p}
    We have a natural isomorphism 
    $$\F^p_{X,v}\cong \bigwedge^p  (k(\supp(v))^{\times})_{\Q} /J_v, $$
    where 
    $J_v$ is the $\Q$-vector subspace generated by $f \in  \bigwedge^p  (k(\supp(v))^{\times})_{\Q}$
  such that $\wedge^p w (f)=0 $ for any specialization $w \in \ZR(k(\supp(v))/K)$ of $v$.
  
  Moreover, 
  decomposition (\ref{decomp for stalk of F^p}) and the reduction map 
  $$\Ker ((k(\supp(v))^{\times})_\Q \to \Q \cdot \Gamma_v) \to (\kappa(v)^{\times})_\Q$$
  (extended $\Q$-linearly) 
  induce an isomorphism
    $$\F^p_{X,v}\cong \bigoplus_{i=0}^p \bigwedge^i \Q\langle a_j\rangle_{j=1}^{\rank v} \otimes K^{p-i}_T (\kappa (v) /K). $$
  \end{lem}
  
  \begin{proof}
    To simplify notation, by \cite[Proposition 3.11]{JSS19} (Remark \ref{remark stalk tropical variety}), 
    we may assume that $\supp (v)$ is a unique generic point of $X$. 
    By Remark \ref{rem;val;vert;special;resi;fld;val}, the first and the second assertions are equivalent. We discuss the second one.

    Let
    $$\varphi \colon \Spec K(X) \to \G_m^r =\Spec K[M'] $$ 
    be a morphism over $K$ to a torus $\G_m^r$
    such that $\varphi(M') $ contains $a_1,\dots, a_{ \rank v}$.
    We put $\overline{\varphi(\Spec K(X))}$ the closure in $\G_m^r$.
    Let $\Lambda $ be  a fan structure of $\Trop(\overline{\varphi(\Spec K(X))})$  such that 
     the cone $\sigma_\varphi := \Trop^{\ad}_{\Lambda} (\varphi(v)) \in \Lambda$ is of dimension $=\rank v$.
    By definition, we have
    $$\F^p_{X,v} = \underrightarrow{\lim}_{\varphi} F^p(\sigma_\varphi, \Trop(\overline{\varphi(\Spec K(X))})).$$
    We put $\overline{\varphi(\Spec K(X))}^{T_{\sigma_\varphi}}$ the closure in the affine toric variety $T_{\sigma_\varphi}$ corresponding to the cone $\sigma_\varphi$. 
    Then we have a natural morphism 
    $$\Spec \O_v \to \overline{\varphi(\Spec K(X))}^{T_{\sigma_\varphi}},$$
    and we put 
    $y \in \overline{\varphi(\Spec K(X))}^{T_{\sigma_\varphi}} \cap O(\sigma_\varphi)$ 
    the image of the maximal ideal of the valuation ring $\O_v$.
    We fix $\tilde{a_i} \in M'$ with $a_i =\varphi(\tilde{a_i})$.
    Then a decomposition 
    $$M'_\Q =\Q\langle \tilde{a_i'}\rangle_{j=1}^{\rank v} \oplus (M'\cap \sigma_\varphi^{\perp})_\Q$$
    induces an isomorphism 
    $$ F^p(\sigma_\varphi, \Trop(\overline{\varphi(\Spec K(X))})) 
    \cong \bigoplus_{i=0}^p \bigwedge^i \Q\langle \tilde{a_j} \rangle_{j=1}^{\rank v} 
    \otimes F^{p-i} (0_{N'_{\sigma_\varphi,\R}}, \Trop(\overline{\varphi(\Spec K(X))}^{T_{\sigma_\varphi}})) ,$$
    where $0_{N'_{\sigma_\varphi,\R}}$ is the zero in 
    $N'_{\sigma_\varphi,\R}:= \Hom (M' \cap \sigma_{\varphi}^{\perp}, \R) $.
    Consequently, it suffices to show that 
    the natural surjective morphism 
      $$F^{p-i} (0_{N'_{\sigma_\varphi,\R}}, \Trop(\overline{\varphi(\Spec K(X))}^{T_{\sigma_\varphi}}))  \twoheadrightarrow 
    F^{p-i} (0_{N'_{\sigma_\varphi,\R}}, \Trop(\overline{\{y\}}^{O(\sigma_\varphi)}))$$ 
    induces an isomorphism
    \begin{align} 
     \underrightarrow{\lim}_{\varphi, \sigma_\varphi} F^{p-i}(0_{N'_{\sigma_\varphi,\R}}, \Trop(\overline{\varphi(\Spec K(X))}^{T_{\sigma_\varphi}}))
     & \cong \underrightarrow{\lim}_{\varphi, \sigma_\varphi} F^{p-i} (0_{N'_{\sigma_\varphi,\R}}, \Trop(\overline{\{y\}}^{O(\sigma_\varphi)})) \label{eq stalk} \\
     (&\cong K^{p-i}_T (\kappa (v) /K)), \notag
    \end{align}
    where $\overline{\{y\}}^{O(\sigma_\varphi)}$ is the closure in $O(\sigma_\varphi)$.
  
    Let $f_1,\dots,f_s \in K[M' \cap \sigma_{\varphi}^{\perp}] \subset K[M']$
    be
    such that 
    $$\overline{\{y\}}^{O(\sigma_\varphi)} =V(f_1,\dots,f_s) \cap \overline{\varphi(\Spec K(X))}^{O(\sigma_\varphi)}$$
    as sets.
    Let 
    $$ \psi :=(\varphi,(f_1,\dots,f_s)\circ \varphi ) 
    \colon \Spec K(X) \to \G_m^r \times \G_m^s .$$
    Then we have 
    $$\overline{\psi(\Spec (K(X)))}^{T_{\sigma_\varphi}\times \A^s} = 
       \overline{\varphi(\Spec K(X))}^{T_{\sigma_\varphi}}$$
     and 
    $$\overline{\psi(\Spec (K(X)))}^{T_{\sigma_\varphi}\times \A^s}  
    \cap O(\sigma_\varphi \times \R_{\geq 0}^s) 
    = \overline{\{y\}}^{O(\sigma_\varphi)}, $$
    where  $\A^s$ is endowed with the natural toric structure corresponding to a cone  $\R_{\geq 0}^s$.
    For a sufficiently fine fan structure $\Lambda'$ of  $  \Trop(\overline{\psi(\Spec (K(X)))})$ (where the closure $\overline{\psi(\Spec (K(X)))}$ is taken in $\G_m^r \times \G_m^s$) 
    and $\sigma_\psi := \Trop^{\ad}_{\Lambda'}(\psi(v))$,
      the natural morphism 
    $$
     \overline{\psi(\Spec (K(X)))}^{T_{\sigma_\psi}  } \cap O(\sigma_\psi)   
    \to
    \overline{\varphi(\Spec K(X))}^{T_{\sigma_\varphi}} \cap O(\sigma_\varphi)
     $$
    factors through
     $$
     \overline{\psi(\Spec (K(X)))}^{T_{\sigma_\varphi}\times \A^s}  
    \cap O(\sigma_\varphi \times \R_{\geq 0}^s). $$
    Hence 
  $$F^{p-i}(0_{N'_{\sigma_\varphi,\R}}, \Trop(\overline{\varphi(\Spec K(X))}^{T_{\sigma_\varphi}}))
  \to 
  F^{p-i}(0_{N'_{\sigma_\psi,\R}}, \Trop(\overline{\psi(\Spec K(X))}^{T_{\sigma_\psi}}))$$
  factors through 
  \begin{align*}
   & F^{p-i}(0_{N'_{\sigma_\varphi \times \R_{\geq 0}^s,\R}},  
   \Trop 
   (\overline{\psi(\Spec (K(X)))}^{T_{\sigma_\varphi}\times \A^s} ))  \\
  \cong &
  F^{p-i} (0_{N'_{\sigma_\varphi,\R}}, \Trop(\overline{\{y\}}^{O(\sigma_\varphi)})).
  \end{align*}
  Hence 
  (\ref{eq stalk}) is injective. (Surjectivity immediately follows.)
  \end{proof}
  
 \subsection{Tropical $K$-groups form a cycle module}\label{Subsection tropical K group cycle module}
  We shall show that tropical Milnor $K$-groups satisfy good properties,
  i.e., they define a cycle module (\cite[Definition 2.1]{Ros96}), which is a ``module'' over Milnor $K$-groups satisfying nice properties like Milnor $K$-groups.
  
  We recall definition and several maps of rational Milnor $K$-groups.
  For a field $E$, its $p$-th rational Milnor $K$-group is defined by 
  $$K_{M,\Q}^p (E) := \bigwedge^p (E^{\times})_\Q / I_M , $$
  where  $I_M$ is the $\Q$-linear subspace of $\bigwedge^p (E^{\times})_\Q$ generated by
  $$\{a_1 \wedge \dots \wedge a_p \mid a_i = 1-a_j  \text{ for some } i \neq j \}.$$
  In particular, we have $K_{M,\Q}^0(E) = \Q$ and $K_{M,\Q}^1(E) = (E^\times)_\Q$. The image of $a_1 \wedge \dots \wedge a_p$ in $K_{M,\Q}^p(E)$  is denoted by $(a_1,\dots,a_p)$.
  \begin{itemize}
  \item A  morphism $\varphi \colon F \to E$ of fields induces  a map
  $$ \varphi_* \colon K_{M,\Q}^p(F) \to K_{M,\Q}^p(E)$$
  by $$\varphi_*((a_1,\dots,a_p)) = (\varphi(a_1), \dots, \varphi(a_p)).$$
  \item For a finite morphism $\varphi \colon F \to E$,
  there is a natural map 
  $$\varphi^* \colon K_{M,\Q}^p(E) \to K_{M,\Q}^p(F)$$
  called the norm homomorphism.
  It is a generalization of
  the multiplication
  $$ \times [E:F] \colon K_{M,\Q}^0 (E) =\Q \to \Q = K_{M,\Q}^0(F)$$
  and
  the usual norm map $E^{\times} \to F^{\times}$.
  This is defined by Bass and Tate \cite{BT72} with respect to a choice of generators of $E$ over $F$, and the independence of the choice was proved by Kato \cite{Kat80}.
  \item For a normalized discrete valuation $v \colon F^{\times} \to \Z$,
  the residue homomorphism (Milnor \cite{Mil70})
  $$\partial_v \colon K_{M,\Q}^p(F) \to K_{M,\Q}^{p-1}(\kappa(v)) $$
  is  characterized by
  \begin{align*}
   \partial_v ((\pi,u_1,\dots , u_{p-1})) & = (\overline{u}_1,\dots,\overline{u}_{p-1}) \\
   \partial_v ((u_1,\dots , u_p)) & = 0 
   \end{align*}
  for a uniformizer $\pi$ of $v$ and $u_i \in F$
  with $v(u_i) =0 \ (1 \leq i \leq p)$, where $\overline{u}_i$ is the reduction.
  (Recall that $\kappa(v)$ is the residue field of $v$.)
  
  We also put 
  $$ s_{v}^{\pi} \colon 
  K_{M,\Q}^p(F) \ni 
  (u_1 , \dots, u_p)
  \mapsto 
  \partial_{v} ((\pi, u_1,\dots, u_p))  
  \in K_{M,\Q}^{p}(\kappa(v)) $$ 
  ($u_i \in F^{\times}$).
  \end{itemize}
  
  \begin{lem}\label{lemtroKfactorMilK}
    For a finitely generated extension $L/K$ of fields,
  the canonical surjective $\Q$-linear map
  $$ \bigwedge^p (L^{\times})_\Q \twoheadrightarrow K_T^p (L/K)$$
  factors through
  $$ \bigwedge^p (L^{\times})_\Q \twoheadrightarrow  K_{M,\Q}^p (L) \twoheadrightarrow K_T^p (L/K).$$
  \end{lem}
  \begin{proof}
  For any $a \in L$, there are no $2$-rational-rank valuations of $K(a)$ which are trivial on $K$.
  Hence the assertion follows from  Corollary \ref{lemK_Tval}.
  \end{proof}
  
  \begin{lem}\label{lemtrKinducedhom}
    Let $E$ and $F$ be finitely generated fields over $K$.
  \begin{itemize}
  \item For a  morphism $\varphi \colon F \to E$ of fields over $K$,
   the map
  $$ \varphi_* \colon K_{M,\Q}^p(F) \to K_{M,\Q}^p(E)$$
  induces a map 
   $$\varphi_*  \colon K_T^p(F/K) \to K_T^p(E/K)$$
   of tropical Milnor $K$-groups.
  \item For a finite morphism $\varphi \colon F \to E$ over $K$,
  the norm homomorphism
  $$\varphi^* \colon K_{M,\Q}^p(E) \to K_{M,\Q}^p(F)$$
  induces a map 
  $$\varphi^* \colon K_T^p(E/K) \to K_T^p(F/K),$$
  also called the norm homomorphism.
  \item For a normalized discrete valuation $v \colon F^{\times} \to \Z$  which is trivial on $K$,
  the residue homomorphism 
  $$\partial_v \colon K_{M,\Q}^p(F) \to K_{M,\Q}^{p-1}(\kappa(v)) $$
  induces a map
  $$\partial_v \colon K_T^p(F/K) \to K_T^{p-1}(\kappa(v)/K), $$
  also called the residue homomorphism.
  \end{itemize}
  \end{lem}
  \begin{proof}
   The assertion on $\varphi_*$ follows from Corollary \ref{lemK_Tval} and the existence of extensions of valuations for extensions of fields (\cite[Chapter 6. Section 3.4. Proposition 5]{Bou72}).
  The assertion 
  on $\partial_v$
  follows from 
    Corollary \ref{lemK_Tval} and 
  Remark \ref{rem;val;vert;special;resi;fld;val}.
  
  We shall prove the assertion on $\varphi^*$.
  By Corollary \ref{lemK_Tval}, it suffices to show that
  for  $\alpha \in K_{M,\Q}^p(E)$ whose image in $K_T^p (E/K)$ is $0$, we have
  $$\wedge^p v (\varphi^* \alpha)=0 $$
  for any $v \in \ZR(F/K)$ with $\Gamma_v \cong \Z^p$,
  where $\Z^p$ is equipped with the lexicographic order.
  We shall show this assertion by induction on $p$.
  When $p=0$, the assertion is trivial.
  We assume $p\geq 1$.
  Let $v_1 \in \ZR(F/K)$ be  the vertical generalization of $v$ of height $1$.
  Then we have $\Gamma_{v_1} \cong \Z$.
  For  an extension  $w_i \in \ZR (E/K)$ of $v_1$ to $E$,
  by the assertion on $\partial_{w_i}$, 
  the image of $\partial_{w_i} (\alpha) \in K_{M,\Q}^{p-1} (\kappa(w_i))$
  in $K_T^{p-1} (\kappa (w_i)/K)$ is $0$.
  Hence by applying the hypothesis of induction to $\varphi_i \colon \kappa(v_1) \to \kappa(w_i)$ and $\partial_{w_i} (\alpha) \in K_{M,\Q}^{p-1}(\kappa(w_i))$, 
  the image of $\varphi_i^{*} (\partial_{w_i} (\alpha)) $ in $K_T^{p-1}(\kappa(v_1)/K)$ is $0$.
  Let $ \overline{v} \in \ZR(\kappa(v_1)/K)$ be the valuation of height $\height (v) -1$ corresponding to $v$ in the sense of Remark \ref{rem;val;vert;special;resi;fld;val}.
  Then we have
  \begin{align*}
     \wedge^{p-1} \overline{v}(\partial_{v_1} \circ \varphi^* \alpha) 
             = \wedge^{p-1} \overline{v}\bigg(\sum_{w_i} \varphi_{i}^* \circ \partial_{w_i} (\alpha)\bigg)
             =0 ,
   \end{align*}
  where $w_i \in \ZR (E/K)$ runs through all extensions of $v_1 \in \ZR(F/K)$ to $E$,
  and  the first equality follows from a basic property
  $$\partial_{v_1}\circ \varphi^* = \sum_{w_i} \varphi_{i}^* \circ \partial_{w_i}$$
  of Milnor $K$-groups.
  Hence $\wedge^p v (\varphi^* \alpha) =0$.
  \end{proof}
  
  We also denote the induced maps of tropical Milnor $K$-groups by $\varphi_*,\varphi^*,\partial_v$.
  
  \begin{thm}\label{trKcymod}
  The functor
  \begin{align*}
    (\emph{finitely generated fields over } K) & \to ( \Z_{\geq 0}\emph{-graded abelian group} ) \\
    L & \mapsto \bigoplus_p K_T^p(L/K).
   \end{align*}
  with $\varphi_*,\varphi^* , \partial_v$ and 
  the natural multiplication 
  $$K_{M,\Q}^p (L) \times K_T^q (L/K) \to K_T^{p+q} (L/K)$$
  is a cycle module in the sense of  Rost \cite[Definition 2.1]{Ros96}. 
  \end{thm}
  \begin{proof}
  This easily follows from Lemma \ref{lemtroKfactorMilK}, Lemma \ref{lemtrKinducedhom}, and
  the fact that Milnor $K$-groups form a cycle module \cite[Theorem 1.4 and Remark 2.4]{Ros96}.
  \end{proof}
  
  We give an explicit resolution of the Zariski sheaf of tropical Milnor $K$-groups on 
  a smooth algebraic variety  $X$ over $K$.
  Let $X^{(i)}$ be the set of points of the scheme $X$ of codimension $i$.  
  For  $i$, a point $x \in X^{(i)}$, and $y \in X^{(i +1)} $,
  Rost defined a map 
  $$\partial_x^y \colon K_T^p(k(x)/K) \to K_T^{p-1}(k(y)/K)$$ 
  (for cycle modules)
   \cite[Section 2]{Ros96} as follows.
  When $y \notin \overline{\{x\}}$, we put $\partial_x^y=0$.
  When $ y \in \overline{\{x\}}$,
  we put
  $$ \partial_x^y :=\sum_v  \varphi_v^* \circ \partial_v\colon K_T^p(k(x)/K) \to K_T^{p-1}(k(y)/K),$$
  where $v \in \ZR(k(x)/K)$ runs through all normalized discrete valuations of $k(x)$ whose
  center in $\overline{\{x\}}$ is $y$,
  and $\varphi_v \colon k(y) \to \kappa(v)$ is the induced morphism.
  We put the Zariski sheaf of $p$-th tropical Milnor $K$-groups $\mathscr{K}_{T,X}^p$ the sheaf on $X_{Zar}$ defined by
  $$\mathscr{K}_{T,X}^p(U):= \Ker ( \bigoplus_{x\in U^{(0)}}K_T^p(k(x)/K)
  \xrightarrow{d} \bigoplus_{y\in U^{(1)}}K_T^{p-1}(k(y)/K) )$$
  for open subsets $U \subset X$,
  where $d:= (\partial_x^y)_{ x \in U^{(0)},  y \in U^{(1)}}$.
  
  \begin{cor}\label{corshftrK}
  For any $p \geq 0$, the sheaf $\mathscr{K}_{T,X}^p$ has the \emph{Gersten resolution}, i.e.,
  an exact sequence
  \begin{align*}
    0 \rightarrow \mathscr{K}_{T,X}^p 
    & \rightarrow \bigoplus_{x \in X^{(0)}} i_{x *} K_T^p(k(x)/K)
    \xrightarrow{d} \bigoplus_{x\in X^{(1)}} i_{x *}K_T^{p-1}(k(x)/K)   \\
      & 
       \xrightarrow{d} \bigoplus_{x\in X^{(2)}} i_{x *}K_T^{p-2}(k(x)/K) \xrightarrow{d}  \dots 
        \xrightarrow{d} \bigoplus_{x\in X^{(p)}} i_{x *}K_T^0(k(x)/K) \xrightarrow{d} 0 ,
   \end{align*}
  where
  $i_x \colon \Spec k(x) \rightarrow X$ are the natural morphisms,
  we identify the groups $K_T^* (k(x)/K)$ and the constant Zariski sheaf on $\Spec k(x)$ given by them,
  and $d:= (\partial_x^y)_{\{x \in X^{(i)} , \ y \in X^{(i+1)}\}}$. 
  In particular, we have
  $$H_{Zar}^p(X,\mathscr{K}_{T,X}^p) \cong CH^p (X)_\Q.$$
  \end{cor}
  \begin{proof}
  The first assertion follows from  Theorem \ref{trKcymod} and \cite[Theorem 6.1]{Ros96}.
  The second assertion follows from the first one. 
  \end{proof}
  
\section{Proof of the main theorem}\label{sec;main theorem}
   The aim of this section is to prove the main theorem (Theorem \ref{thm,trop,main}).
   It follows from
   (easy lemmas in Subsection \ref{subsec:tro:coh:sev:comput}
   and)
    a theorem on coniveau spectral sequences of general cohomology theories \cite[Corollary 5.1.11]{CTHK97},
   which is developed by many mathematicians including 
   Quillen \cite{Qui73},
   Bloch-Ogus \cite{BO74}, Gabber \cite{Gab94}, Rost \cite{Ros96}, and Colliot-Th\'{e}l\`{e}ne-Hoobler-Kahn \cite{CTHK97}.
   To apply this theorem to tropical cohomology, 
   we will prove 
   Proposition \ref{Keyprp etale}, Proposition \ref{Keyprp}, and Corollary \ref{corestriction}.

   Let $K$ be a trivially valued field.
   
 \subsection{Proof of the main theorem}\label{subsection proof of main theorem}
  \begin{prp}\label{Keyprp etale}
  Let $ \Phi \colon X' \rightarrow X$ be an \'{e}tale morphism of smooth algebraic varieties over $K$.
  Let $Z \subset X$ be a closed subscheme.
  We assume $Z':=\Phi^{-1}(Z) \rightarrow Z$ is an isomorphism.
  Then we have
  $$ \Phi^* \colon H_{\Trop,Z}^{p,q}(X)\cong H_{\Trop,Z'}^{p,q}(X').$$
  \end{prp}
  
  \begin{proof}
  Let $v'\in Z'^{\Ber}$, and  $v:=\Phi(v') \in Z^{\Ber}$.
  We shall show 
  the natural morphism
  $ \C_{Z \subset X,v}^{p,q}
  \to  \C_{Z' \subset X',v'}^{p,q} $
  of stalks is an isomorphism.
  We have a commutative diagram 
  $$\xymatrix{
  0 \ar[r] & \C_{Z \subset X,v}^{p,q} \ar[r] \ar[d]
  &  \C_{ X,v}^{p,q}  \ar[r] \ar[d]
  & \pi_* \pi^* \C_{ X,v}^{p,q } \ar[d]
  \\
  0 \ar[r] & \C_{Z' \subset X',v'}^{p,q} \ar[r] 
  &  \C_{ X',v'}^{p,q}  \ar[r]
  & \pi_* \pi^*  \C_{ X',v'}^{p,q } ,
  }$$ 
  whose rows are exact.
  Since an \'{e}tale morphism is open, 
  the image $\Phi (X'^{\Ber})$ contains an open neighborhood of $v \in X^{\Ber}$.
  Hence the second and the third vertical arrows are injective. 
  Hence it is enough to prove 
  $ \C_{ X,v}^{p,q}
  \to  \C_{ X',v'}^{p,q} $
  is surjective. 
  More precisely, it is enough to show that 
  for $f \in \O_{X',\supp (v')}$, 
  the tropicalization map $\Trop \circ f$ on an open neighborhood of $v'$ factors through $\Phi$ and a tropicalization map on an open neighborhood of $v$.
  
  Since $\Phi $ is etale, 
  there exist  $f' \in \O_{X',\supp (v')}^{\times}$ and $g\in \O_{X,\supp (v)}$ such that $f=f'\cdot g$.
  (Here, to simplify notation, for an element $h \in \O_{X,\supp (v)}$, let $h$ also donote the image of $h$ under $ \O_{X,\supp (v)} \hookrightarrow \O_{X',\supp(v')}$.)
  Since $Z' \cong Z$, 
  there exists $g' \in \O_{X,\supp(v)}^{\times}$ such that 
  $\overline{g'} =\overline{f'} \in k(\supp(v'))^{\times}$.
  Let 
  $$ W':=\{w'\in U'^{\Ber} \mid w'(1-\frac{g'}{f'}) >0 \}$$
  be an open neighborhood of $v'$, 
  where $U'\subset X'$ is an open subscheme 
  such that $v'\in U'^{\Ber}$, 
  functions $f,g,g'$ are in $ O(U')$,
  and $f'$ is in $\O^{\times} (U')$
  Then by ultrametric inequality, 
  for $w' \in W'$, we have
  $ w'(g') = w'(f' ) . $
  Hence
  we have 
  $$\Trop \circ (f',g)|_{W'}= \Trop \circ (g',g) \circ \Phi|_{W'} ,$$
  where $\Trop \colon \A^{2,\Ber} \to \Trop(\A^2)$.
  Hence 
  $\Trop \circ f|_{W'} $ factors through $ \Trop \circ (g',g) \circ \Phi|_{W'}$. 
  \end{proof}

  \begin{prp}\label{Keyprp}
  Let $X $ be a smooth quasi-projective variety over $K$,
  and $Z \subset  X$ a closed subscheme.
  We put $\pi \colon X \times \A^1 \to X$ the first projection.
  Then
  the pullback map
  $$ \pi^* \colon H_{\Trop,Z}^{p,q}(X ) \to H_{\Trop,Z \times \A^1}^{p,q}(X \times \A^1)$$
  is an isomorphism.
  \end{prp}
  Proof of Proposition \ref{Keyprp} will be given in Section  \ref{sec;trocoho;P^1;fiber}.
  
  \begin{prp}\label{corestriction}
  We assume that $K$ is a finite field.
  Let $L/K$ be an extension of trivially valued finite fields,
  and $X$ be a smooth irreducible algebraic variety over $K$.
    Then there is a morphism
  $$\cores \colon H_{\Trop}^{p,q}(X_L) \to  H_{\Trop}^{p,q}(X)$$
  such that 
  $\cores \circ \res = [L:K]$, 
  where 
  $ \res \colon H_{\Trop}^{p,q}(X) \to H_{\Trop}^{p,q}(X_L)$
  is the natural morphism.
  \end{prp}

   Proof of Proposition \ref{corestriction} 
   will be given in Section \ref{section finite fields}.

  We put  $\mathscr{H}^{p,q} $ the Zariski sheaf on $X$ associated to the presheaf $U \mapsto H^{p,q}_{\Trop}(U)$.
  
  \begin{thm}\label{thm,trop,main}
  Let $X$ be a smooth algebraic variety over $K$.
  Then there exist  natural isomorphisms
  \begin{align*}
  	 H_{\Trop}^{p,q} (X) & \cong H_{\Zar}^q(X,\H^{p,0}) ,   \\
     \H^{p,0} & \cong \K_{T,X}^p  .
  \end{align*}
  In particular, we have
  $$ H_{\Trop}^{p,p}(X) \cong CH(X)^p_\Q.$$
  \end{thm}
  \begin{proof}
  The last assertion follows from Corollary \ref{corshftrK}.
  
  For each $r$, by Proposition \ref{Keyprp etale}, 
  Proposition \ref{Keyprp},
  Proposition \ref{corestriction}, and 
  \cite[Remarks 5.1.3, Corollary 5.1.11, and Proposition 5.3.2 (a)]{CTHK97},
  there exists a spectral sequence
  $$ E_1^{p,q}=\coprod_{x \in X^{(p)}} H_{\Trop,x}^{r,p+q}(X) \Rightarrow H_{\Trop}^{r,p+q}(X)$$
  whose $E_2$-terms are
  $E_2^{p,q}\cong H_{\Zar}^p(X,\mathscr{H}^{r,q})$,
  and we have
  $$ \mathscr{H}^{r,0}(V)
  \cong \Ker \big( d_1 \colon 
  \bigoplus_{x \in V^{(0)}}  H^{r,0}_{\Trop,x} (X) 
  \rightarrow 
  \bigoplus_{x \in V^{(1)}}  H^{r,1}_{\Trop, x}(X) \big) $$ 
  for open subvariety $V \subset X$,
  where
  \begin{itemize}
  \item $H_{\Trop,x}^{r,p+q}(X):= \underrightarrow{\lim}_{x\in U} H_{\Trop,\overline{\{x\}}\cap U}^{r,p+q}(U)$, where $U \subset X$ runs through all open neighborhoods of $x$,
  \item $V^{(i)} \subset V$ is the subset of points of codimension $i$, and
  \item $d_1$ is the differential map of the spectral sequence.
  \end{itemize}
  By Lemma \ref{lemtrcohoxtrivial}, we have
  $E_1^{p,q}=0$ for $q\geq 1$.
  We have $E_2^{p,q}\cong H_{\Zar}^p(X,\mathscr{H}^{r,q})=0$ for $q \leq -1$.
  Hence 
  $$H_{\Zar}^p(X,\H^{r,0}) \cong E_2^{p,0}=E_{\infty}^{p} \cong H_{\Trop}^{r,p} (X).$$
  By Lemma  \ref{lemH^r,1trK} and Corollary \ref{corshftrK}, we also have
  \begin{align*}
  	 \mathscr{H}^{r,0}(V) &
  \cong \Ker \big(
  \bigoplus_{x \in V^{(0)}}  H^{r,0}_{\Trop,x} (X) 
   \xrightarrow{d_1} \bigoplus_{x \in V^{(1)}}  H^{r,1}_{\Trop, x}(X) \big)\\
  				 &  \cong \Ker \big(
             \bigoplus_{x \in V^{(0)}}   K_T^{r} (k(x) /K)
            \xrightarrow{d=(\partial_{\eta}^x)_x} \bigoplus_{x \in V^{(1)}}  K_T^{r-1} (k(x) /K) \big)  \\
  				 & = \mathscr{K}_{T,X}^r(V).
   \end{align*}
  \end{proof}
  
  \begin{rem}
  	By construction, the isomorphism $H_{\Trop}^{p,p}(X) \cong CH(X)^p_\Q$ coincides with Liu's tropical cycle class map (\cite[Definition 3.8]{Liu20}), 
  	which uses the Zariski sheaf of Milnor $K$-groups $\K_M^p$ and Bloch's formula.
  \end{rem}
   
 \subsection{Easy lemmas}\label{subsec;trop;main;thm;basic;lem}
  \label{subsec:tro:coh:sev:comput}
  In this subsection, we shall show Lemma \ref{lemtrcohoxtrivial} and \ref{lemH^r,1trK},
  which are used to prove Theorem \ref{thm,trop,main}.
  Let $X$ be a smooth algebraic variety over $K$.
  Let $M$ be a free $\Z$-module of rank $n$ and $N:= \Hom (M,\Z)$.
  
  \begin{lem}\label{lem:trop:cone:dim:exchang}
  	We assume that $X$ is affine.
  Let $\varphi \colon X \to T_{\sigma}$ be a closed immersion to the affine toric variety $T_{\sigma}$ corresponding to a cone $\sigma \subset N_\R$.
  Let $p \geq 0$, and  $x \in X$ a point of codimension $p$.
  We assume that
  $X \cap \varphi^{-1}(O(\sigma)) = \overline{\{x\}}$ 
  and $X\cap \varphi^{-1}(\G_m^n)\neq \emptyset $, 
  where $\G_m^n=\Spec K[M]$.
  Let $p_x \subset \O(X)$ the prime ideal corresponding to $x$.
  
  Then
  there exist $f_1, \dots ,f_r \in \O(X) \setminus p_x$
  such that
  \begin{enumerate}
  \item $$\Trop((\varphi,(f_i)_i)(X)) \cap (\sigma \times \{0\}^r)
  \subset \Trop(T_{\sigma}) \times (\R \cup \{\infty\})^r $$ 
  is a finite union of cones of dimension $\leq p$,
  and \label{enu:1:lem:trop;cone:dim:exchang}
  \item $$\Trop((\varphi,(f_i)_i)(X)) \cap( \sigma \times \{0\}^r \setminus \relint (\sigma \times \{0\}^r))
  \subset \Trop(T_{\sigma}) \times (\R \cup \{\infty\})^r $$ 
  	is a finite union of cones of dimension $\leq p-1$, \label{enu:2;lem:trop:cone:dim:exchang}
  \end{enumerate}
  where 
  $(\varphi ,(f_i)_i) \colon X \to T_{\sigma} \times \A^r$
  is a closed immersion, and 
  $\A^r$ is equipped with a natural toric structure.
  \end{lem}
  \begin{proof}
  Let $\Lambda $ be a fan structure of $\Trop(\varphi(X)) \cap \sigma$.
  We put $X'$ the closure of $\varphi(X) \cap \G_m^n$ in $T_{\Lambda}$.
  We put $B \subset X $ the union of the image of the generic points of the irreducible components of $X' \cap O(\tau)$ for  cones $\tau \in \Lambda$ of dimension $\geq p$ under the natural morphism $X' \to X$.
  Then  any $y \in B $ is of codimension $\geq p$ in $X$,
  in particular,  every $y \in B \setminus \{x\}$ is not a generalization of $x$.
  Hence there exist $f_1,\dots,f_r \in \O(X) \setminus p_x$
  such that for any $y \in B \setminus \{x\}$, we have $f_i\in p_y$ for some $1\leq i\leq r$, where $p_y$ is the prime ideal corresponding to $y$.
  
  Let $\Xi$ be a fan structure of 
  $$\Trop((\varphi,(f_i)_i)(X) \cap \G_m^{n+r}) \cap  (\sigma \times \{0\}^r)$$
  such that for any $\xi \in \Xi$, there exists a cone $\tau_\xi\in \Lambda$ with $\relint \pr(\xi)\subset \relint \tau_\xi$, where $\pr \colon N_\R \times \R^r \to N_\R$ is the projection.
  Let $\xi \in \Xi$ be a cone of dimension $\geq p$. 
  Suppose that $\dim \xi \geq p+1$ or $\xi$ does not intersect with $\relint ( \sigma \times \{0\}^r)$.
  Then since $X' \cap O(\tau_\xi)\subset X'$ is of codimension $\dim \tau_\xi (\geq \dim \xi )$ and $x\in  \varphi^{-1}(O(\sigma))$, the generic point of any irreducible component of $X' \cap O(\tau_\xi)$  
  does not map to $x$, i.e., maps to $B\setminus \{x\}$.
  Hence the product $f_1 \cdot \cdot \cdot f_r$ is $0$ on $X' \cap O(\tau_\xi)$.
  Hence $f_1 \cdot \cdot \cdot f_r$ is  $0$ on $X'' \cap O(\xi)$,
  where $X''$ is the closure of $((\varphi,(f_i)_i)(X) \cap \G_m^{n+r})$ in $T_{\Xi}$.
  This contradicts to $\xi \subset \sigma \times \{0\}^r$.
  Hence there are no such $\xi$.
  \end{proof}
  
  \begin{lem}\label{lemtrcohoxtrivial}
  Let $p \geq 0$,  $x \in X^{(p)}$, and $q \geq 1$.
  Then $$H_{\Trop,x}^{r,p+q}(X) =0.$$
  \end{lem}
  \begin{proof}
  	By Remark \ref{rem;X;circ;cpt}, Corollary \ref{cor;X;circ;trocoho;T;Z} and Lemma \ref{lem:trop:cone:dim:exchang},
   each element of $H_{\Trop,x}^{r,p+q}(X)$ is given by a cocycle of 
   $$\Trop((\varphi,(f_i)_i)(X)) \cap (\overline{\sigma} \times \{0\}^r)$$ 
   as in Lemma \ref{lem:trop:cone:dim:exchang} 
   for some $\varphi$ defined on an affine neighborhood of $x$,
   where $\overline{\sigma}$ is the closure in $\Trop (T_{\sigma})$.
  	Then since 
    $$\Trop((\varphi,(f_i)_i)(X)) \cap (\overline{\sigma} \times \{0\}^r) 
    = \overline{\Trop((\varphi,(f_i)_i)(X)) \cap (\sigma \times \{0\}^r)},$$
    the assertion follows from the long exact sequences of relative tropical cohomology.
  \end{proof}
  
  \begin{lem}\label{lemH^r,1trK}
  There are natural isomorphisms
  $$ H^{r,0}_{\Trop,x} (X) \cong K_T^r (k(x) /K) \quad (x \in X^{(0)})$$
  and
  $$H^{r,1}_{\Trop,y}(X) \cong K_T^{r-1}(k(y) /K )  \quad (y \in X^{(1)})$$
  such that
  $$\xymatrix{
  	H^{r,0}_{\Trop,x}(X)  \ar[d] \ar[r]^-{d_1} &  H^{r,1}_{\Trop, y}(X) \ar[d] \\
  	K_T^r ( k(x)/K ) \ar[r]^-{\partial_x^y}  & K_T^{r-1} (k(y)/K)
  }$$ 
  is commutative for $y \in \overline{ \{x\} }$.
  \end{lem}
  \begin{proof}
  The first isomorphism is given by Corollary  \ref{cor;X;circ;trocoho;T;Z}.
  The second one is given as follows. (By construction, the diagram is commutative.)
  We fix $y \in X^{(1)}$ and $x \in X^{(0)}$ such that $y \in \overline{ \{x\} }$.
  By Corollary \ref{cor;X;circ;trocoho;T;Z} and Lemma \ref{lem:trop:cone:dim:exchang},
  the tropical cohomology $H^{r,1}_{\Trop,y}(X)$ is isomorphic to
  $$\underrightarrow{\lim}_{\varphi}
  H^{r,1}_{\Trop,\overline{l_\varphi} \setminus l_\varphi} (\overline{l_\varphi}, \Lambda_{\varphi}),$$ 
  where $\varphi \colon U_\varphi \to T_{l_\varphi}$
  run through
  all closed immersions
  from open neighorhoods $U_\varphi \subset X$ of $y$
  to the affine toric varieties $T_{l_{\varphi}}$ corresponding to  $1$-dimensional cones $l_\varphi $
  such that
  $\varphi^{-1}(O(l_{\varphi})) = U_\varphi \cap \overline{\{y\}}$,
  the closure $\overline{l_\varphi} $ is taken in $\Trop(T_{l_\varphi})$, 
  and $\Lambda_\varphi $ is a fan structure of $\Trop(\varphi(U_\varphi))$.
  (Note that for $a \in \relint l_\varphi$, the Shilov boundary of $(\Trop \circ \varphi)^{-1}(a)$ is a (non-normalized) discrete valuation of $k(x)$ corresponding to $y$.)
  By Remark \ref{rem;val;vert;special;resi;fld;val},
  a map
  $$\underrightarrow{\lim}_{\varphi}
  H^{r,1}_{\Trop,\overline{l_\varphi} \setminus l_\varphi} (\overline{l_\varphi}, \Lambda_{\varphi})
  \to K_T^{r-1}(k(x))$$ 
  given by $$H^{r,1}_{\Trop,\overline{l_\varphi} \setminus l_\varphi} (\overline{l_\varphi}, \Lambda_{\varphi})
  \ni \alpha_\varphi \mapsto (-1)^r (\alpha_\varphi)_{\gamma_\varphi}( d_\varphi \wedge \cdot) $$ 
  is the required isomorphism,
  where $d_\varphi \in l_\varphi$ is the primitive element,
  and $\gamma_\varphi \colon [0,1] \to \overline{l_\varphi}$ is a fixed homeomorphism
  such that $\gamma_\varphi(0)$ is $0 \in l_\varphi $.
  \end{proof}
   
 \subsection{Example}\label{subsec example}
  In this subsection, we shall give examples of smooth algebraic varieties 
  over $\mathbb{C}$ with the trivial valuation 
  for which 
  tropical cohomology is isomorphic to singular cohomology.
  
  Let $X$ be a smooth algebraic variety over  $\mathbb{C}$.
  By \cite[Proposition 5.5]{M21},
  there exists a natural morphism 
  $$\K_T^p \ni f_1 \wedge \dots \wedge f_p 
  \mapsto  \frac{ - d (\log f_1)  }{2 \pi i  }  \wedge \dots \wedge \frac{ - d  (\log f_p) }{2 \pi i  }  
  \in  \Omega_X^p ,$$ 
  where $\Omega_X^p$ is the Zariski  sheaf of algebraic $p$-forms.
  It induces a morphism 
  \begin{align}
  H_{\Trop}^{p,q}(X) \to H_{\sing}^{p+q} (X(\CC),\CC) 
  \label{tro coh to sing coh CC}
   \end{align}
  which is compatible with 
  the cycle class map $\CH^p (X) \to  H_{\sing}^{2p} (X(\CC),\CC)$ and the natural isomorphism 
  $\CH^p(X) \otimes \Q \cong H_{\Trop}^{p,p}(X)$ 
  (\cite[Remark 5.6]{M21}).
  
  The morphism (\ref{tro coh to sing coh CC}) induces 
  an isomorphism
  \begin{align}
  H_{\Trop}^{p,q}(X) \cong H_{\sing}^{p+q} (X(\CC),\Q)
   \label{tro coh = sing coh}
   \end{align}
  in the following cases.
  
  \begin{enumerate}
  	\item When $X \cong \A^n$ ($n\geq 0$), the isomorphism (\ref{tro coh = sing coh}) follows from $\A^1$-homotopy invariance and compatibility of  (\ref{tro coh to sing coh CC}) with the morphisms from $\CH^p(X)\otimes \Q$.
  	\item When $X $ is $\G_m^n$ ($n\geq 1$) or the complements of hyperplane arrangements of $\A^n$, the isomorphism (\ref{tro coh = sing coh}) follows from the case of affine spaces, induction, and long exact sequences of (tropical and singular) cohomology for pairs.
  	\item When there is a simple normal crossing divisor $D=\bigcup_{i \in I} D_i \subset X$ 
  	such that 
      the isomorphism (\ref{tro coh = sing coh}) 
  	holds for each strata 
  	$$\bigcap_{j \in J} D_j \setminus (\bigcup_{i \in I \setminus J} D_i )$$
  	($J \subset I$), 
      the isomorphism (\ref{tro coh = sing coh}) 
  	holds for $X$ by induction and long exact sequences of (tropical and singular) cohomology for pairs.
  \end{enumerate}
  	Examples of (3) contains  smooth toric varieties and the wonderful compactifications of the complements of  hyperplane arrangements in the sense of \cite{CP95}.
   
\section{Analytifications and tropicalizations of the affine line}\label{sec;projective;line}
   Let $(L,v_L \colon L^{\times} \to \R)$ be a complete valuation field of height $1$.
   In this section, we recall Berkovich's and Huber's analytifications of the affine line $\A^1_{L}=\Spec L [T]$ 
   (Subsection \ref{subsec:struc:P1})
   and their  tropicalizations (Subsection \ref{subsec;trop;p1}) in details.
   In Subsection \ref{subsec;A1;fiber},
   we recall some facts on tropicalizations of $\A^1$-fibers.
   
   Let $v_L$ also denote the extension to an algebraic closure $L^{\alg}$ and its extension to the $v_L$-adic completion $\hat{L^{\alg}}$ of $L^{\alg}$. 
   We denote the residue field of $(L,v_L)$ by $\kappa(L)$.
   
 \subsection{Analytifications of the affine line}\label{subsec:struc:P1}
  We describe the structure of the analytifications of the affine line.
  See \cite[1.4.4]{Ber90}, \cite[Example 2.20]{Sch11}, \cite{Hub93}, \cite{Hub94}.
  
  As a set,  Huber's analytification $\A^{1,\ad}_{L}$ is the set of equivalence classes of continuous valuations of $5$ types.
  There is a canonical inclusion $\A^{1,\Ber}_{L} \subset \A^{1,\ad}_{L}$ as sets.
  The Berkovich analytic space $\A^{1,\Ber}_{L}$ is identified with the set of valuations of heights $1$, and is also identified with the set of points of type $1-4$.
  (Points of type $5$ are of height $2$.)

  \begin{dfn}\label{dfn:P1:type:pt}
  \begin{itemize}
  \item $x \in \A_{L}^{1,\ad}$ is said to be of \emph{type $1$}
  when there exists $a \in \hat{L^{\alg}}$ 
  such that $x$ is defined by the composition
  $$L[T] \xrightarrow{T = a} \hat{L^{\alg}}  \xrightarrow{v_L} v_L(\hat{L^{\alg}}).$$
  In this case, we say that $x$ corresponds to $a$.
  
  \item $x \in \A^{1,\ad}_{L}$ is said to be of \emph{type $2$} (resp.\ of \emph{type $3$}) when
  there exists $a \in \hat{L^{\alg}}$
  and $r \in v_L((\hat{L^{\alg}})^{\times})$ (resp.\ $r \in \R \setminus v_L((\hat{L^{\alg}})^{\times})$)
  such that  $x$ is defined by the restriction of
  $$  \hat{L^{\alg}}[T] \ni \sum_i a_i (T-a)^{i} \mapsto \min_i \{v_L(a_i) + ir\}$$
  ($a_i \in \hat{L^{\alg}}$).
  In this case, we say that $x$ \emph{corresponds} to $(a,r)$.
  
  \item $x \in \A^{1,\ad}_{L}$ is said to be of \emph{type $5$} when
  there exists $a \in \hat{L^{\alg}}$, a real number $r \in v_L((\hat{L^{\alg}})^{\times})$,  and $\epsilon \in \{ 1 , - 1\}$
  such that  $x$ is defined by the restriction of
  $$  \hat{L^{\alg}}[T]  \ni \sum_i a_i (T-a)^{i} \mapsto \min_i \{ (v_L(a_i) +ir, \epsilon i ) \in \R \times \Z\}$$
  ($a_i \in \hat{L^{\alg}}$),
  where  $\R \times \Z$ is equipped with the rexicographic order.
  In this case, we say that $x$ \emph{corresponds} to $(a,r,\epsilon)$.
  \end{itemize}
  \end{dfn}

  We do not recall points of type $4$, which are not important in this paper.

  Only a point $u$ of type $5$ has a (unique) non-trivial generalization $w$ in $\A^{1,\ad}_L$ in the topological sense. When $u$ corresponds to $(a,r,\epsilon)$ ($a \in \hat{L^{\alg}}$,
   $r \in v_L((\hat{L^{\alg}})^{\times})$,  $\epsilon \in \{ 1 , - 1\}$),  the valuation $w$ is the point of type $2$ corresponding to $(a,r)$.

  \begin{rem}\label{same valuation pair}
  \begin{itemize}
  \item For a point $x$ of type $2$ or $3$ corresponding to $(a,r)$,
  the number $r$ is unique, but $a$ is not unique.
  In fact, for $a' \in \hat{L^{\alg}}$ with $v_L(a-a') \geq r$, the valuation $x$ also corresponds to $(a',r)$.
  \item Let $u$ be a point of type $5$ corresponding to $(a,r,\epsilon)$.
  When $\epsilon = 1$, for $a' \in \hat{L^{\alg}}$ with $v_L(a-a') > r$,
  the valuation $u$ also corresponds to $(a',r,1)$.
  When $\epsilon = -1$, for $a' \in \hat{L^{\alg}}$ with $v_L(a-a') \geq r$,
  the valuation $u$ also corresponds to $(a',r,-1)$.
  \end{itemize}	
  \end{rem}
  
  \begin{rem}
  Let $x \in \A^{1,\Ber}_{L} $ be a point  of type $1$, $2$, or $3$ corresponding to $a \in \hat{L^{\alg}}$ or $(a,r)$ ($r\in \R$).
  We put $(x,\infty) \subset \A^{1,\ad}_{L}$
  the set of valuations
  corresponding to $(a,r')$ ($ r' <r$)
  and $(a,r,\epsilon)$ ($ (r',\epsilon) < (r,0)$), 
  where we put $r:=\infty$ and $(r,0):= \infty$ when $x$ is of type $1$.
  The map
  $$ (x,\infty) \cap \A_L^{1,\Ber} \ni (\emph{the valuation corresponding to }(a,r')) \mapsto r' \in \R_{<r}$$
  is  a homeomorphism.
  We put $[x,\infty):= \{x\} \cup (x,\infty)$.
  \end{rem}

  For $\mu \in \A^{1,\ad}_L$ of type $2$ or $5$,
  we put $\mpd (\mu)$ the minimum of $[L(a'),L]$ over all $ a' \in L^{\alg}$ such that $\mu$ corresponds to $(a',r)$ or $(a',r,\epsilon)$ for some $r$ or $(r,\epsilon)$.

  \begin{lem}{\cite[Theorem 3.9 b)]{APZ90}}\label{mpd near type 1 point}
  For any $a \in L^{\alg}$, 
  there exists $r \in v_L((\hat{L^{\alg}})^{\times})$
  such that $\mpd (w) = [L(a):L]$,
  where $w$ is the valuation corresponding to $(a,r)$. 
  \end{lem}

  Let $w \in \A^{1,\ad}_{L}$ be a point of type $2$  corresponding to $(a,r)$.
  We put $u_{w,\infty} \in \A^{1,\ad}_{L}$ be the specialization of $w$ corresponding to $(a,r,-1)$, which does not depend on the choice of $(a,r)$.
  We also put $u_{w,0} \in \A^{1,\ad}_{L}$ the specialization of $w$ corresponding to $(0,r,1)$ when $w$ corresponds to $(0,r)$ ($r \in \R$), 
  and put $u_{w,0}:= u_{w,\infty} \in \A^{1,\ad}_{L}$ otherwise.

  We put $\kappa(w)\cap \kappa(L)^{\alg} $ the subfield of the residue field $\kappa (w) $ consisting of elements algebraic over $\kappa(L)$.
  By Remark \ref{rem;val;vert;special;resi;fld;val}, 
  for a specialization $u$ of type $5$  of  $w $, we have
  a natural injection
  $\kappa(w)\cap \kappa(L)^{\alg} \hookrightarrow \kappa(u)$.
  
  \begin{lem}{\cite[Theorem 2.1 and Corollary 2.1]{APZ88}}\label{rem;P1;adic;formal;model}
  Let $w \in \A^{1,\ad}_{L}$ be a point of type $2$  corresponding to $(a,r)$ ($a\in L^{\alg}$, $r \in \R$),
  and $u \in \A^{1,\ad}_{L}$ a specialization corresponding to $(a,r,\epsilon)$ ($\epsilon \in \{ \pm 1\}$).
  \begin{enumerate}
  \item   The residue field $\kappa(w)$ is isomorphic to the one variable rational function field over $\kappa(w) \cap \kappa(L)^{\alg}$. \label{enu:1:rem:P1:adic:formal:model}
  	\item Let $\mu := w$ or $u$. 
  	Then the reductions of polynomials $g \in L[T]$ of degree $< \mpd (\mu)$ 
  	with $\mu(g) =0$
  	generate the multiplicative group $(\kappa(\mu) \cap \kappa(L)^{\alg})^{\times}$, 
  	and we have $v_L(g(a))=0$.
  	 This gives a canonical inclusion
  	 $$ \kappa (\mu) \cap \kappa (L)^{\alg} \hookrightarrow \kappa(L(a)).$$ \label{enu:2new:rem:P1:adic:formal:model}
  	  When $\mpd (\mu)=[L(a),L]$,  this inculsion is an equality. \label{enu:3new:rem:P1:adic:formal:model}
       \label{enu:3:rem:P1:adic:formal:model}
  \end{enumerate}
  \end{lem}
  \begin{proof}
  (\ref{enu:1:rem:P1:adic:formal:model}) is \cite[Corollary 2.1]{APZ88}.
  
  (\ref{enu:2new:rem:P1:adic:formal:model})
  $\mu (g)= v_L(g(a))$ 
  can be proved in the same way as \cite[Theorem 2.1 a)]{APZ88}.
  The other parts of 
  (\ref{enu:3new:rem:P1:adic:formal:model}) can be proved in the same way as  \cite[Theorem 2.1 d)]{APZ88}.
  \end{proof}

  \begin{cor}\label{enu:2:rem:P1:adic:formal:model}
  We have 
     $$ \kappa (w) \cap \kappa (L)^{\alg} = \kappa (u_{w,\infty}) = \kappa (u_{w,0}).$$
  \end{cor}
  \begin{proof}
  We have $\mpd (w)=\mpd (u_{w,\infty})$.
  Moreover, when $w$  corresponds to $(0,r)$, we obviously have
  $$\kappa (u_{w,0}) \cap \kappa (L)^{\alg} = \kappa(L(0))=\kappa (L).$$
  Hence the assertion
  follows from Lemma \ref{rem;P1;adic;formal;model} (\ref{enu:3new:rem:P1:adic:formal:model}).
  \end{proof}
  
  By Remark \ref{rem;val;vert;special;resi;fld;val} and Lemma \ref{rem;P1;adic;formal;model} (\ref{enu:1:rem:P1:adic:formal:model}),  we have bijections
  \begin{align*}
  & \{ \text{non-trivial specializations of } w \in  \A^{1,\ad}_{L}\} \\
   \cong & \{ \text{equivalence classes of non-trivial valuations of } \kappa(w) \text{ which are trivial on } \kappa(L) \} \\
   \cong & \{ \text{closed points of }  \P^{1}_{\kappa(w) \cap \kappa(L)^{\alg}} \}.
   \end{align*}
  For a specialization $u$ of $w$, we put $\overline{u}$ the corresponding valuation of $\kappa(w)$.
   
 \subsection{Tropicalizations and skeletons of the affine line}\label{subsec;trop;p1}
  We shall give explicit descriptions of tropicalizations and tropical skeletons of the affine line.
  We fix toric structures of affine spaces.

  \begin{rem}\label{rem[x,inf][a,infty]}
  Let $a \in L$, 
  and $(T-a) \colon \A^{1}_{L} \to \A^{1}_{L}$ the morphism given by $T\mapsto T-a$.
  We identify $\Trop (\A^1) = \R \cup \{ \infty \}$.
  Then for $x\in \R \subset \Trop (\A^1)$, we have 
  $$(\Trop \circ (T-a))^{-1}(x)= \M (L\{(e^{x} (T-a))^{\pm 1}\}) $$
  (see \cite{Ber90} for the affinoid domain $\M (L\{(e^{x} (T-a))^{\pm 1}\})$).
  Its Shilov bounary constists of one point,
  and we have
  $\Sk_{(T-a)} \A^{1,\Ber}_{L} = [a,\infty) \cap \A^{1,\Ber}_L$.
  There is a retraction
  $ \A^{1,\Ber}_{L} \to [a,\infty) \cap \A^{1,\Ber}_L$ 
  which  gives 
  a commutative diagram
  $$\xymatrix{
  \A^{1,\Ber}_{L} \ar[d] \ar[dr]^-{\Trop \circ (T-a)} 
     & \\
   [a,\infty) \cap \A^{1,\Ber}_L  \ar[r]^-{\cong} & \R \cup \{\infty\}. \\
  }$$ 
  
  We have a generalization. 
   Let $T-a_i \in L[T]$ ($i\in I$).
   We have a morphism 
  $(T-a_i)_i \colon \A^{1}_{L} \to (\A^{1}_{L})^{\# I}$.
  For 
  $$x=(x_i)_{i\in I} \in \Trop((T-a_i)_i(\A_{L}^{1,\Ber})) \cap \R^{\# I}
  \subset  \Trop (\A^{\# I}_L)  
  ,$$ 
  we have 
  $$(\Trop \circ (T-a_i)_i)^{-1}(x) 
  \cong \M (L\{(e^{x_i } (T-a_i))^{\pm 1}\}_{i\in I}),$$ 
  which is 
  isomorphic to $\M (L \{ S^{\pm 1} , (S-b_l)^{-1}\}_l )$
  for an indeterminant $S$ and some $b_l \in L$ with $v(b_l)=0$ when $x\in v_L(L^\times)^n $,  
  and is a single point of type $3$ when $x\notin v_L(L^\times)^n $.
  In any case, its Shilov boundary constists of one point, 
  and we have
  $$\Sk_{(T-a_i)_i} \A^{1,\Ber}_{L} = \bigcup_i [a_i,\infty) \cap \A^{1,\Ber}_L.$$
  There is a retraction
  $ \A^{1,\Ber}_{L} \to  \bigcup_i [a_i,\infty) \cap \A^{1,\Ber}_L$ 
  which  gives a commutative diagram
  $$\xymatrix{
  \A^{1,\Ber}_{L} \ar[d] \ar[dr]^-{\Trop \circ (T-a_i)_i} 
     & \\
   \bigcup_i [a_i,\infty) \cap \A^{1,\Ber}_L  \ar[r]^-{\cong} & \Trop ((T-a_i)_i (\A^{1,\Ber}_{L} )) .\\
  }$$ 
  \end{rem}
   
  \begin{lem}\label{skeletons of A1:label}
   Let $f =(f_i)_i \colon \A^1_L  \rightarrow \A^r_L$ be a closed immersion.
   We put $a_{i,k}$ the zeros of $f_i$. 
   Then we have 
   $$\Sk_{f} \A^1_L = \bigcup_{i,k}[a_{i,k},\infty)\cap \A^{1,\Ber}_L.$$
  \end{lem}
  \begin{proof}
  By Lemma \ref{skeleton base change},  we may assume that $L$ is algebraically closed.
  Let $a = (a_{i,k})_{i,k} \colon \A^1_L  \rightarrow \A^s_L$.
  Then there is a natural finite-to-one map 
  $$ \Trop (a (\A^1_L)) \twoheadrightarrow \Trop (f (\A^1_L)).$$
  Hence fibers of $  \Trop \circ f  $ are disjoint unions of finitely many fibers of $\Trop \circ a$. 
  In particular, $\Sk_f \A^1_L = \Sk_a \A^1_L$.
  Hence the assertion follows from Remark \ref{rem[x,inf][a,infty]}.
  \end{proof}
  
  We put 
   $$\Sk_{f} \A^{1,\ad}_L := \bigcup_{i,k}[a_{i,k},\infty).$$

  \begin{lem}\label{tropicalization A1 as retraction}
   Let $f $ and $a_{i,k}$ be as in Lemma \ref{skeletons of A1:label},
  and 
  $$a \in \A^{1,\ad}_L  \setminus \bigcup_{i,k}[a_{i,k},\infty)$$ 
  a point of type $1$. 
  We put $w \in \A^{1,\ad}_L $  the point (of type $2$) such that 
  $$ [a,\infty) \cap  \bigcup_{i,k}[a_{i,k},\infty)
   = [w,\infty) .$$
  Then we have 
  $$ \Trop (f (([a,\infty) \setminus [w,\infty) )\cap \A^{1,\Ber}_L) )
    = \Trop ( f (w) ).$$
  \end{lem}
  \begin{proof}
  This follows from Lemma \ref{skeleton base change} and Remark \ref{rem[x,inf][a,infty]}.
  \end{proof}

  \begin{lem}\label{existence of nice tropicalization of A1}
   Let $f $ and $a_{i,k}$ be as in Lemma \ref{skeletons of A1:label},
   and $L' \subset L$ a dense subfield.
   Then there exist a morphism 
   $g =(g_{i,k})_{i,k} \colon \A^1_L \to \A^s_L$ given by polynomials $g_{i,k} \in L'[T]$ irreducible in $L[T]$
   such that $\Trop \circ (f,g)$ induces a homeomorphism
   $$
    \big( \bigcup_{i,k}[a_{i,k},\infty)
    \cup  [b_{i,k}, \infty)  \big) 
    \cap \A^{1,\Ber}_L
   \cong \Trop ((f,g) (\A^1_L)),$$
   where $b_{i,k}$ is the zero of $g_{i,k}$.
  \end{lem}
  \begin{proof}
  By Lemma \ref{mpd near type 1 point}, 
  for each irreducible factor $f_{i,k} \in L[T]$ of $f_i$, 
  a polynomial $g_{i,k} \in L'[T]$ close to $f_{i,k}$ is irreducible over $L$.
  For $g_{i,k}$ sufficiently close to $f_{i,k}$,
  we get the isomorphism in the assertion using Lemma \ref{tropicalization A1 as retraction} for both $(f, a:=b_{i,k})$ and $(g,a:=a_{i,k})$ ($i,k$). 
  \end{proof}

 \subsection{Tropicalizations of $\A^1$-fibers}\label{subsec;A1;fiber}
  In this subsection, we recall some facts about tropicalizations of fibers of the projection 
  $ \pi \colon X \times \A^1 \to X$ for algebraic variety $X$ over a trivially valued field $K$.
  Recall that we have
  $$X^{\Ber}/ (\text{the equivalence relation of valuations})
  \cong X^{\ad,\height \leq 1} \subset X^{\ad}.$$
  Let $v \in X^{\Ber}$ be a valuation of height $1$. We put $[v] \in X^{\ad}$ its equivalence class.
  
  \begin{rem}\label{rem;pi-1v;Ber;ad;identify}
  There is a natural inclusion
  $$\A^{1,\ad}_{k(\supp(v))_v} \hookrightarrow \pi^{-1}([v])$$
  whose image is the subset consisting of (possibly trivial) specializations of valuations in $\pi^{-1}([v])$ of height $1$.
  We have the following commutative diagram:
  $$\xymatrix{
  &  \A^{1,\Ber}_{k(\supp(v))_v} \ar[r]^-{\cong}\ar@{^{(}->}[d] & \pi^{-1}(v) \ar@{^{(}->}[d] \ar@{^{(}->}[r] 
  & (X \times \A^1)^{\Ber}  \ar[d] 
  \\
  &  \A^{1,\ad}_{k(\supp(v))_v}  \ar@{^{(}->}[r] &\pi^{-1}([v])   \ar@{^{(}->}[r]
  & (X \times \A^1)^{\ad}   ,
  }$$ 
  see \cite[Section 3.1]{Ber90} and \cite[Definition/Exercise 4.1.7.1]{Tem15}.
  We identify elements of the sources and the images of the above injections.
  \end{rem}
  
  Let $\varphi \colon X \times \A^1 \to \A^r$ be a closed immersion over $K$. (We fix a toric structure of $\A^r$.)
  By abuse of notation, we put 
  $$\varphi \colon \A^1_{k(\supp(v))_v} \to \A^r_{k(\supp(v))_v}$$
  the morphism induced by a natural morphism
  $$\iota \colon \A^1_{k(\supp (v))} \cong \pi^{-1}(\supp(v)) \to X \times \A.$$
  We have a commutative diagram 
  $$\xymatrix{
  \A^{1,\Ber}_{k(\supp(v))_v}   
    \ar[d]^-{\Trop \circ \varphi}   
   \ar@{^{(}->}[r]  
   \ar@{^{(}->}[ddr] &
  (X \times \A^1)^{\Ber} \ar[d]_-{\Trop \circ \varphi} 
    \ar[ddr] &
   \\
   \Trop(\varphi(\A^{1,\Ber}_{k(\supp(v))_v}  )) \ar@{^{(}->}[r]\ar@{^{(}->}[ddr] &
   \Trop(\varphi((X \times \A^1)^{\Ber})) \ar[ddr]  &
    \\
    &
   \A^{1,\ad}_{k(\supp(v))_v} 
    \ar[d]^-{\Trop \circ \varphi}   \ar@{^{(}->}[r]   &
  (X \times \A^1)^{\ad} \ar[d]_-{\Trop \circ \varphi} \\
  &
   \Trop^{\ad} (\varphi(\A^{1,\ad}_{k(\supp(v))_v}  )) \ar@{^{(}->}[r] &
   \Trop^{\ad} (\varphi((X \times \A^1)^{\ad}))  , 
  }$$ 
  where 
  $\Trop^{\ad} (\varphi(\A^{1,\ad}_{k(\supp(v))_v}  )) $ is defined as the image of the composition 
  $$
   \A^{1,\ad}_{k(\supp(v))_v} 
   \hookrightarrow
  (X \times \A^1)^{\ad} 
  \twoheadrightarrow 
   \Trop^{\ad} (\varphi((X \times \A^1)^{\ad})) . $$ 
  The image of 
  $$ \Trop(\varphi(\A^{1,\Ber}_{k(\supp(v))_v}  )) 
  \hookrightarrow
  \Trop^{\ad} (\varphi(\A^{1,\ad}_{k(\supp(v))_v}  )) 
  $$  
  is the subset of height $1$ points.

  We assume that 
  $$\varphi = (f_i)_{1 \leq i \leq r} 
  \colon \A^1_{k(\supp(v))_v} \to \A^r_{k(\supp(v))_v}$$
  induces a bijection
   $$\Sk_{\varphi} \A^1_{k(\supp(v))_v} 
   = \bigcup_{i,k}[a_{i,k},\infty)\cap \A^{1,\Ber}_{k(\supp(v))_v}
   \cong \Trop (\varphi (\A^1_{k(\supp(v))_v})),$$
   where $a_{i,k}$ runs through all zeros of $f_i$,
   and the equality is Lemma \ref{skeletons of A1:label}.
  Then we have 
  $$\Sk_\varphi \A^{1,\ad}_{k(\supp (v))_v} 
  \cong 
  \Trop^{\ad} (\varphi(\A^{1,\ad}_{k(\supp(v))_v}  )) 
  . $$

  \begin{rem}
  In a similar way to the trivially valued case, 
  we can directly define tropicalizations and skeletons of adic spaces associated to algebraic varieties over complete non-trivial valuation fields.
  However, this is excessive for our purpose. We instead adopt the above ad hoc approach.
  \end{rem}

  The projection
  $\pi \colon X \times \A^1 \rightarrow X $
  induces a surjective map
  $\pi^{\circ} \colon (X \times \A^1)^{\circ} \rightarrow X^{\circ}$.
  We will study 
  a subset $(\pi^{\circ})^{-1}(v)$ of the fiber $\pi^{-1}(v)$ for $v \in X^{\circ}$.
  The canonical homeomorphism $\pi^{-1}(v)\cong \A^{1,\Ber}_{k(\supp(v))_v}$
  induces a homeomorphism
  $$(\pi^{\circ})^{-1}(v) \cong 
  \{w \in \A^{1,\Ber}_{k(\supp(v))_v} \mid w(T) \geq 0 \} .$$ 
  By abuse of notation, we put 
  $$(\pi^{\circ})^{-1}([v])  := 
  \{w \in \A^{1,\ad}_{k(\supp(v))_v} \mid w(T) \geq 0 \}  
  \subset (X\times \A^1)^{\ad}.   $$ 
  We put
  \begin{align*}  
  \Sk_{\varphi}((\pi^{\circ})^{-1}(v))
  :=  & \Sk_{\varphi}\A^{1,\Ber}_{k(\supp(v))_v} \cap (\pi^{\circ})^{-1}(v), \\
  \Sk_{\varphi}((\pi^{\circ})^{-1}([v]))
  :=  & \Sk_{\varphi}\A^{1,\ad}_{k(\supp(v))_v} \cap (\pi^{\circ})^{-1}([v]) .
  \end{align*}

\section{$\A^1$-homotopy invariance}\label{sec;trocoho;P^1;fiber}
   In this section, we shall prove Proposition \ref{Keyprp},
   i.e., $\A^1$-homotopy invariance of tropical cohomology.
   Let $X  $ be a smooth quasi-projective variety over a trivially valued field $K$.
   We put
   $\pi \colon X \times \A^1 \rightarrow X $
   the first projection.  
   
   By five lemma, to prove Proposition \ref{Keyprp}, we may assume that $Z= X$.
   By Corollary \ref{cor;X;circ;trocoho;T;Z}, 
   Proposition \ref{Keyprp} follows from
   an isomorphism
   \begin{align}\label{equation goal}
   H^{q}((X \times \A^1)^{\circ}, \F^p_{(X\times \A^1)^{\circ}}) \cong H^{q}(X^{\circ}, \F^p_{X^{\circ}} ),
   \end{align}
   where we put e.g., $\F_{X^{\circ}}^p:= \F^p_{X}|_{X^{\circ}}$.
   We fix $v_0 \in X^{\circ}$.
   We shall show that 
   $R^i \pi^{\circ}_* \F_{(X \times \A^1)^{\circ},v_0}^r=0 $ 
   ($i \geq 1$) (Corollary \ref{R2 pi_*=0}, Corollary \ref{R1 pi_* =0})
   and 
   $\F_{X^{\circ},v_0}^r \cong \pi^{\circ}_* \F_{(X\times \A^1)^{\circ},v_0}^r$
   (Corollary \ref{pi_*=F}).
   Isomorphism (\ref{equation goal}) follows from these.
   
   In the rest of this section, we assume that $v_0$ is of height $1$.
   When $v_0$ is of height $0$, the proof is much easier. We omit it.
   To simplify notation, we assume that $X$ is affine and irreducible, and $\supp v_0$ is the generic point of $X$. The general case easily follows from this case by using retractions.

 \subsection{Notations}\label{subsec Restrictions to fibers}
  
  We fix elements $a_1,\dots,a_{\rank v_0} \in K(X)$ 
  such that $v_0(a_1),\dots, v_0(a_{\rank v_0})$ form a basis of a $\Q$-vector space $\Q \cdot \Gamma_{v_0}$.
  We may assume that 
  $a_1,\dots, a_{\rank v_0} \in \O(X)$.
  We put 
  $K(X)_{v_0}$  the $v_0$-adic completion of the function field $K(X)$.
  By abuse of notation, for 
  a morphism $\psi$ from $X \times \A^1$ to a toric variety over $L$, 
  let $\psi $ also denote its base change from 
  $$\A^1_{K(X)_{v_0}}  \cong (X\times \A^1) \times_X \Spec K(X)_{v_0}
   $$  to the toric variety over $K(X)_{v_0}$.
   We fix toric structures of affine spaces and tori.

  We put $I_{X\times \A^1 }$ the small category whose objects are 
  closed immersions 
  $$\varphi= (\varphi',  (a_i)_{i=1}^{\rank v_0} ) 
  \colon 
  X\times \A^1
  \to  \A^r \times \G_m^{\rank v_0}  $$
  over $K$
  such that 
  $\Trop(\varphi ( (\pi^{\circ})^{-1} (v_0) ))$ is not a single point, 
  the map $\Trop \circ \varphi$
  induces a bijection
  $$\Trop \circ \varphi |_{
    \Sk_{\varphi} (\A^1_{K(X)_{v_0}}) 
  }  \colon \Sk_{\varphi} (\A^1_{K(X)_{v_0}}) 
  \cong \Trop(\varphi ( \A^1_{K(X)_{v_0}} )) ,$$
   and 
  $\varphi(X \times \A^1)$ intersects with the dense torus $\G_m^{r + \rank v_0} \subset \A^r \times \G_m^{\rank v_0} $, 
  and morphisms from $\varphi_1 $ to $\varphi_2$ 
  $$(\varphi_i \colon X\times \A^1 \to  \A^{r_i} \times  \G_m^{\rank v_0} \ (i=1,2) ) $$
  are
  toric morphisms $\psi' \colon \A^{r_2}  \to \A^{r_1} $ 
  such that 
  $$\xymatrix{
  X\times \A^1 \ar[r]^-{\varphi_1 } \ar[rd]_-{\varphi_2 } &  \A^{r_1} \times \G_m^{\rank v_0} \\
  & \A^{r_2} \times  \G_m^{\rank v_0} \ar[u]_-{\psi' \times  \id_{ \G_m^{\rank v_0}} }
  }$$
  is commutative. 
  By Lemma \ref{existence of nice tropicalization of A1}, 
  there are sufficiently many objects in $I_{X \times \A^1}$.
  We will compute 
   $R^i \pi^{\circ}_* \F_{(X \times \A^1)^{\circ},v_0}^r $ 
  using $I_{X\times \A^1}$. 

  We put $T_{\Sigma'} :=  \A^r$ when we emphasize that $\A^r$ is a toric variety,  and put $\Sigma'$ the corresponding fan. 
  We put $M' $ the free $\Z$-module such that $\Spec K[M'] \subset  T_{\Sigma'}$ is the dense torus. 
  We put 
  $\G_m^{\rank v_0} =\Spec K[\tilde{a}_i^{\pm 1}]_{i=1}^{\rank v_0}$,
  where $\tilde{a}_i $ maps to $a_i$.
  We put $M:=M' \oplus \Z \langle \tilde{a}_i\rangle_{i=1}^{\rank v_0}$.
  We put $T_{\Sigma} := T_{\Sigma'} \times \G_m^{\rank v_0}$ a toric variety, and $\Sigma$ the corresponding fan.
  There is a natural bijection $\Sigma \cong \Sigma'$. 
  For a cone $\sigma \in \Sigma$, let $\sigma'$ denote its image in $\Sigma'$.
  We put $N:= \Hom (M,\Z)$ and $N' := \Hom (M', \Z)$.

  We have 
  $$\Trop(\varphi(\pi^{-1}(v_0) )) = \Trop(\varphi'(\pi^{-1}(v_0)  )) \times \{ (v_0 (a_i))_{i=1}^{\rank v_0} \}.$$
  We often identify 
  $\Trop(\varphi(\pi^{-1}(v_0) )) $ and $ \Trop(\varphi'(\pi^{-1}(v_0)  )) .$
  We also identify 
  $\Trop(\varphi ((\pi^{\circ})^{-1} (v_0)))$
  and points of 
  $\Trop^{\ad} (\varphi ((\pi^{\circ})^{-1} ([v_0])))$
  of height $1$. 

  Let $\Lambda$ be  a sufficiently fine fan structure of $\Trop(\varphi(X \times \A^1))$. 
  Then 
  $$\Xi_1
  := \Lambda \cap \Trop(\varphi(\pi^{-1}(v_0) ))
  := \{ P \cap  \Trop(\varphi(\pi^{-1}(v_0) ))\}_{P \in \Lambda} $$
  is
  a $\Gamma_{v_0}$-rational polyhedral complex
  structure of  
  $\Trop(\varphi(\pi^{-1}(v_0) )) \subset \Trop(T_{\Sigma})$,
  there is a subcomplex $\Xi_1^{\circ} \subset \Xi_1$ 
  whose support is $\Trop(\varphi ((\pi^{\circ})^{-1} (v_0)))$,
  and
  for  a cone $\sigma' \in \Sigma'$, 
  a polyhedral complex structure 
  $\Xi_1 \cap N'_{\sigma',\R}$ of
   $\Trop (\varphi'(\pi^{-1}(v_0)) )\cap N'_{\sigma',\R}$
  is the restriction of 
  a $\Gamma_{v_0}$-admissible tropical fan  $\Xi_{\sigma'}$ 
  in $N'_{\sigma',\R} \times \R_{\geq 0}$ 
  for $\varphi'(\A^1_{K(X)_{v_0}}) \cap O(\sigma')_{K(X)_{v_0}},$ 
  where the polyhedral complex $\Xi_1 $ is also considered as a polyhedral complex structure of 
  $ \Trop(\varphi'(\pi^{-1}(v_0)  )) \subset \Trop(T_{\Sigma'})$.
  (See Subsection \ref{subsec;trop;triv;val} for $\Gamma_{v_0}$-admissible tropical fans.)
  We have a natural map 
  $$ 
  \Trop^{\ad}_{\Lambda} \big( \varphi \big( (\pi^{\circ})^{-1} ([v_0]) \big) \big)
  \ni P \mapsto 
  P \cap \Trop (\varphi ((\pi^{\circ})^{-1} (v_0) ))
  \in 
  \Xi_{1}^{\circ  }.$$ 
  We put 
  $$
  \Trop^{\ad}_{\Xi_1^{\circ}} \circ \varphi
  \colon  (\pi^{\circ})^{-1} ([v_0]) 
  \twoheadrightarrow
  \Xi_{1}^{\circ  }$$
  the composition of this map and $\Trop^{\ad}_{\Lambda} \circ \varphi $.

  For $x \in \Trop^{\ad} (\varphi ((\pi^{\circ})^{-1} ([v_0]) ))$, 
  we put $v_x $ the unique element in 
  $$ (\Trop^{\ad} \circ \varphi)^{-1} (x) \cap 
    \Sk_{\varphi} (\pi^{\circ})^{-1} ([v_0]) .$$
  For $\xi_1 \subsetneq \xi_2 \in \Xi_1^{\circ}$ with $\xi_1 \in N_\R$, 
  we put $u_{\xi_1,\xi_2} \in \Sk_{\varphi} (\pi^{\circ})^{-1} ([v_0])$ 
  the specialization of $v_{\xi_1}$ which maps to $\xi_2$.

  The following is used in the next subsection. 
  \begin{lem}\label{reduction of A1}
    Let $\sigma \in \Sigma$ be a cone, 
     $\xi \in \Xi_1^{\circ} \cap N_{\sigma,\R}$.
    We put $P_\xi \in \Lambda \cap N_{\sigma,\R}$
    the minimal cone such that 
    $$P_\xi \cap \Trop (\varphi ( (\pi^{\circ})^{-1} (v_0) )) = \xi.$$
    Then 
    $$\overline{
   \bigcup_{ v \in (\pi^{\circ})^{-1}( [v_0] ) 
   \cap (\Trop^{\ad}_{\Xi_1^{\circ}} \circ \varphi)^{-1}(\xi) } 
    \Cent_{T_{P_\xi}} (v)}$$
    is irreducible, and its generic point is $\Cent_{T_{P_\xi}} (v_x)$,
    where  
    $x \in \Trop^{\ad} (\varphi ( (\pi^{\circ})^{-1} ([v_0]) ))$ maps to $\xi$, 
    and
  the centers $\Cent_{T_{P_\xi}} $ of valuations and the closure are taken in the affine toric variety $T_{P_\xi} \supset O(\sigma)$ corresponding to the cone $P_\xi$.
  \end{lem}
  \begin{proof}
  Note that $\Cent_{T_{P_\xi}} (v)$ is contained in the closed orbit $O(P_\xi) \subset T_{P_\xi}$.
  We put $
  c(\xi):=\overline{ \R_{\geq 0} (\xi \times \{1\})} $ 
  the cone in  
  the $\Gamma_{v_0}$-admissible tropical fan  $\Xi_{\sigma'}$ 
   such that $c(\xi) \cap (N'_{\sigma',\R} \times \{1 \} )= \xi$.
  We have a commutative diagram 
  $$\xymatrix{
    & O(\sigma)_K  \ar@{_{(}->}[r]
    & T_{\Lambda \cap N_{\sigma,\R},K} 
    & O(P_\xi)_K  \ar@{_{(}->}[l]
    \\
  \varphi'(\A^1_{K(X)_{v_0}}) \cap O(\sigma')_{K(X)_{v_0}} 
  \ar[r]^-{\varphi'} \ar[ru]^-{\varphi}
    & O(\sigma')_{K(X)_{v_0}}  \ar@{_{(}->}[r] \ar[u]
    & \T_{\Xi_{\sigma'}  ,\O_{v_0}} \ar[u]
    & O(c(\xi))_{\kappa(v_0)}  \ar@{_{(}->}[l] \ar[u],
  }$$ 
  where 
  $\T_{\Xi_{\sigma'}  ,\O_{v_0}}$ is the toric scheme over $\O_{v_0}$ corresponding to $\Xi_{\sigma'}$ (see Subsection \ref{subsec;trop;triv;val}, and more precisely, \cite[7.7]{Gub13}), 
  the horizontal arrows except $\varphi'$ are natural inclusions, 
   the first vertical arrow is given by a morphism 
   $$(a_i)_{i=1}^{\rank v_0} \colon K(X)_{v_0} \to \G_m^{\rank v_0},$$
   the second one is the natural extension of the first one, 
   and the third one is the restriction of the second one.
   Hence the second assertion for 
   $x $ of height $1$ and the first assertion 
   follows from Remark \ref{tropical compactifications and initial degeneration}
   and 
   the fact 
   that
   the inverse images of generic points under reduction maps are Shilov boundaries (\cite[Subsection 2.13]{GRW17}). 
   When $x $ is of height $2$, 
   the polyhedron $\xi$ is of dimension $1$,
   and 
   hence 
   $$ 
    \overline{\varphi'(\A^1_{K(X)_{v_0}}) \cap O(\sigma')_{K(X)_{v_0}} }
    \cap O(c(\xi))_{\kappa(v_0)}  $$
   consists of a single point.
   Hence  the second assertion for $x$ of height $2$ holds.
  \end{proof}

  \begin{rem}\label{reduction of A1 2}
  The morphism 
  $$ O(c(\xi))_{\kappa(v_0)}  \to O(P_\xi)_K $$
  in proof of Lemma \ref{reduction of A1} 
  becomes isomorphism after a base extension.
  By Remark \ref{tropical compactifications and initial degeneration}, 
  when $\xi $  is of dimension $1$, 
  for a point $x \in \xi$ of rank $= \rank v_0$, 
  the natural morphism
  \begin{align*}
   & 
    \overline{\varphi'(\A^1_{K(X)_{v_0}}) \cap O(\sigma')_{K(X)_{v_0}} }^{\T_{c(x),\O_{v_0}}}
    \cap O(c(x))_{\kappa(v_0)}    \\
    \to  &
    \overline{\varphi'(\A^1_{K(X)_{v_0}}) \cap O(\sigma')_{K(X)_{v_0}} }^{\T_{\Xi_\sigma', \O_{v_0}}}
    \cap O(c(\xi))_{\kappa(v_0)}  
  \end{align*}
  is a trivial $\G_m^1$-fibration after a base extension,
  where 
  $$c(x):= \R_{\geq 0} (x \times \{1\}) \subset N_{\sigma',\R} \times \R_{\geq 0}$$
  is a cone, 
  and the closures are taken in toric schemes $\T_{c(x),\O_{v_0}}$
  and $\T_{\Xi_\sigma', \O_{v_0}}$ over $\O_{v_0}$, respectively.
  \end{rem}

 \subsection{Resolutions}\label{sub:resolutions}
  In this subsection, we shall give a resolution of  $\F^p_{(X\times \A^1)^{\circ}} |_{(\pi^{\circ})^{-1}(v_0)}$.

  \begin{dfn}
    Let $\sigma \in \Sigma$ be a cone, 
     $\xi \in \Xi_1^{\circ} \cap N_{\sigma,\R}$.
    We put $P_\xi \in \Lambda \cap N_{\sigma,\R}$
    the minimal cone such that 
    $$P_\xi \cap \Trop (\varphi ( (\pi^{\circ})^{-1} (v_0) )) = \xi.$$
  We put 
  $F^p_{``\overline{(\pi^{\circ})^{-1}([v_0])}''}(\xi)$ the coimage of the natural morphism 
  $$\bigwedge^p (M\cap \sigma^{\perp})_\Q 
  \to \prod_{ v \in (\pi^{\circ})^{-1}([v_0]) \cap 
  (\Trop^{\ad}_{\Xi_1^{\circ}} \circ \varphi)^{-1}(\xi) }
   \F^p_{(X \times \A^1)^{\circ} , v} ,$$
  and put
  $$\overline{F}^p_{``\overline{(\pi^{\circ})^{-1}([v_0])}''}(\xi)
  := F^p \big(0_{(N_{\sigma})_{P_{\xi}  ,\R}}, 
  \Trop \big(\overline{
   \bigcup_{ v \in (\pi^{\circ})^{-1}( [v_0] ) 
   \cap (\Trop^{\ad}_{\Xi_1^{\circ}} \circ \varphi)^{-1}(\xi) } 
    \Cent_{T_{P_\xi}} (v)} 
  	\big) \big).
  $$
  \end{dfn}
  
  Let $\tilde{b}_1,\dots,\tilde{b}_{\dim P_\xi} \in M\cap \sigma^{\perp}$ form a basis of the dual of $\Span P_\xi$ with $\tilde{b}_i:= \tilde{a}_i$ ($1 \leq i \leq \rank v_0$). 
  (Note that $\dim P_\xi = \rank v_0 + \dim \xi$ for sufficiently fine $\Lambda$.)
  They give a decomposition 
  $$ (M\cap \sigma^{\perp} )_\Q \cong \Q\langle \tilde{b}_i \rangle_{i=1}^{\dim P_\xi} 
  \oplus  (M\cap \sigma^{\perp} \cap  P_\xi^{\perp})_\Q. 
  $$
  It induces a surjection 
  $$ \bigwedge^p (M\cap \sigma^{\perp} )_\Q
  \twoheadrightarrow \bigoplus_{i=0}^p  
      \bigwedge^i \Q\langle \tilde{b}_i \rangle_{i=1}^{\dim P_\xi} 
  	\otimes \overline{F}^{p-i}_{``\overline{(\pi^{\circ})^{-1}([v_0])}''}(\xi).
  $$

  \begin{lem}\label{Fp overline decomposition}
  We have 
  $$F^p_{``\overline{(\pi^{\circ})^{-1}([v_0])}''}(\xi)
  \cong 
  \bigoplus_{i=0}^p  
      \bigwedge^i \Q\langle \tilde{b}_i \rangle_{i=1}^{\dim P_\xi} 
  	\otimes \overline{F}^{p-i}_{``\overline{(\pi^{\circ})^{-1}([v_0])}''}(\xi).
  $$ 
  \end{lem}
  \begin{proof}
  When $\xi$ is of dimension $0$, 
  since 
  the elements $v(a_1), \dots, v(a_{\rank v_0})$
  extend to a basis of $\Q \cdot \Gamma_v$ 
  for
  $$v \in (\pi^{\circ})^{-1}([v_0]) \cap 
  (\Trop^{\ad}_{\Xi_1^{\circ}} \circ \varphi)^{-1}(\xi),$$
  the assertion follows from Lemma \ref{stalk of F^p}. 
  When $\xi $ is of dimension $1$,
  for
  $$v \in (\pi^{\circ})^{-1}([v_0]) \cap 
  (\Trop^{\ad}_{\Xi_1^{\circ}} \circ \varphi)^{-1}(\xi)$$
  with
  $\rank v   < \rank v_0 +\dim \xi$,
  the elements $v(a_1), \dots, v(a_{\rank v_0})$
  do not extend to a basis of $\Q \cdot \Gamma_v$. 
  However, 
  the assertion still holds by Remark \ref{reduction of A1 2}.
  \end{proof}

  \begin{rem}\label{Fp overline subset morphisms:label}
  For $\xi_1 \subsetneq \xi_2 \in \Xi_1^{\circ}$,
  by Lemma \ref{reduction of A1} and Lemma \ref{Fp overline decomposition},
  there is a natural morphism 
  $$ i_{\xi_1 \subset \xi_2} \colon F^p_{``\overline{(\pi^{\circ})^{-1}([v_0])}''}({\xi_1})
  \to F^p_{``\overline{(\pi^{\circ})^{-1}([v_0])}''}({\xi_2}),$$
  which is surjective when $\xi_1 \in N_\R$,
  and is injective when $\xi_1 \notin N_\R$.
  (When $\xi_1 \notin N_\R$, by Lemma \ref{reduction of A1} (the irreducibility), 
  we have 
  $$\overline{
   \bigcup_{ v \in (\pi^{\circ})^{-1}( [v_0] ) 
   \cap (\Trop^{\ad}_{\Xi_1^{\circ}} \circ \varphi)^{-1}(\xi_1) } 
    \Cent_{T_{P_{\xi_1}}} (v)}
    =
   \overline{
   \bigcup_{ v \in (\pi^{\circ})^{-1}( [v_0] ) 
   \cap (\Trop^{\ad}_{\Xi_1^{\circ}} \circ \varphi)^{-1}(\xi_2) } 
    \Cent_{T_{P_{\xi_2}}} (v)}$$
    via an identification $O(P_{\xi_1})= O(P_{\xi_2})$.)
  \end{rem}
  
  \begin{rem}\label{new Fp bar and Fp KT}
  For 
  $v \in (\pi^{\circ})^{-1} ([v_0]) $,
  we have 
  natural morphisms
  \begin{align*}
  F^p_{``\overline{(\pi^{\circ})^{-1}([v_0])}''}(
  \Trop^{\ad}_{\Xi_1^{\circ}}  (\varphi (v)) )
     \to  & \F^p_{(X \times \A^1)^{\circ}, v}, \\ 
  \overline{F}^p_{``\overline{(\pi^{\circ})^{-1}([v_0])}''}(
  \Trop^{\ad}_{\Xi_1^{\circ}}  (\varphi (v)) )
  \to & K_T^p (\kappa (v)/K).
  \end{align*}
  By Lemma \ref{reduction of A1} and Lemma \ref{Fp overline decomposition},
  they are injective when $v \in \Sk_{\varphi} ( (\pi^{\circ})^{-1} ([v_0]) )$.
  By Lemma \ref{stalk of F^p}, 
  we have 
  \begin{align*}
  \F^p_{(X \times \A^1)^{\circ} , v} 
  \cong & \underrightarrow{\lim}_{\varphi \in I_{X \times \A^1}}
  F^p_{``\overline{(\pi^{\circ})^{-1}([v_0])}''}(
  \Trop^{\ad}_{\Xi_1^{\circ}}  (\varphi (v)) )   ,  \\
  K_T^p (\kappa (v)/K) 
  \cong & \underrightarrow{\lim}_{\varphi \in I_{X \times \A^1}}
  \overline{F}^p_{``\overline{(\pi^{\circ})^{-1}([v_0])}''}(
  \Trop^{\ad}_{\Xi_1^{\circ}}  (\varphi (v)) )   .
  \end{align*}
  \end{rem}
  
  By Remark \ref{Fp overline subset morphisms:label},
  we can define a complex 
  $C^{p,*}_{``\overline{(\pi^{\circ})^{-1}([v_0])}''} 
  ( \Trop(\varphi((\pi^{\circ})^{-1}(v_0) ))$ 
  of cochains with $F^p_{``\overline{(\pi^{\circ})^{-1}([v_0])}''}({\xi}) $-coefficients ($\xi \in \Xi_1^\circ$),  
  similarly to 
  tropical cochains $C^{p,*} (B, \Lambda)$ in Subsection \ref{subsec;trop;cohomology;fan}.
  Similarly to the usual tropical cohomology (see Subsection \ref{subsec:trop:coh:hom}), 
  these complexes for $\varphi \in I_{X\times \A^1}$ induce a complex  of c-soft sheaves 
  $\C^{p,q}_{``\overline{(\pi^{\circ})^{-1}([v_0])}''} $ on $(\pi^{\circ})^{-1}(v_0) $. 
  It is 
  a resolution of  
  $$\F^p_{(X \times \A^1)^{\circ}} |_{(\pi^{\circ})^{-1}(v_0)} \cong 
  \Ker (\C^{p,0}_{``\overline{(\pi^{\circ})^{-1}([v_0])}''} 
  \to \C^{p,1}_{``\overline{(\pi^{\circ})^{-1}([v_0])}''}  ).$$
  In particular, we have the following.
  \begin{cor}\label{stalk cohomology}
  	$$H^q ((\pi^{\circ})^{-1}(v_0) , \F^p_{(X \times \A^1)^{\circ}} |_{(\pi^{\circ})^{-1}(v_0)} )
  	\cong H^q (\C^{p,*}_{``\overline{(\pi^{\circ})^{-1}([v_0])}''} ((\pi^{\circ})^{-1}(v_0)) ).$$
  \end{cor}

  Since $\pi^\circ  \colon (X\times \A^1)^{\circ} \to X^{\circ}$ is proper,
  by Corollary \ref{stalk cohomology}, we have 
  \begin{align*}
    (R^i \pi^{\circ}_* \F_{(X\times \A^1)^{\circ}}^r)_{v_0}  
   \cong & H^i ((\pi^{\circ})^{-1} (v_0), 
   \F^r_{(X\times \A^1)^{\circ}}  |_{(\pi^{\circ})^{-1} (v_0)})  \\
   \cong & H^i (\C^{r,*}_{``\overline{(\pi^{\circ})^{-1}([v_0])}"} ((\pi^{\circ})^{-1} (v_0)) )  \\
   \cong & \underrightarrow{\lim}_{\varphi \in I_{X \times \A^1} } H^i (C^{r,*}_{``\overline{(\pi^{\circ})^{-1}([v_0])}"} (\Trop (\varphi ((\pi^{\circ})^{-1} (v_0) )) )).
  \end{align*}
  
  \begin{cor}\label{R2 pi_*=0}
  $ R^i \pi_* \F_{(X \times \A^1)^{\circ}}^r =0$
  for $i\geq 2$.
  \end{cor}
   
 \subsection{Local descriptions}
  In this subsection,
  we shall describe a natural morphism 
  $$i_{\xi_1 \subset \xi_2} \colon F^p_{``\overline{(\pi^{\circ})^{-1}([v_0])}''}({\xi_1})
  \to F^p_{``\overline{(\pi^{\circ})^{-1}([v_0])}''}({\xi_2})$$
  for a fixed object $\varphi  \colon X \times \A^1 \to T_{\Sigma'} \times \G_m^{\rank v_0} $ in $I_{X\times \A^1}$ 
  and $\xi_1 \subsetneq \xi_2 \in \Xi_1^{\circ}$.
  
  \begin{rem}\label{local descriptions basics:label}
  Let $\sigma \in \Sigma$ be a cone, 
   $\xi  \in \Xi_1^{\circ} \cap N_{\sigma,\R}$, 
  and $x  \in \Trop^{\ad} (
    \varphi ( (\pi^{\circ})^{-1} ([v_0]) )
  )$ a point of rank $= \rank v_0 + \dim \xi$ mapping to $\xi$.
  Let $\tilde{b}_1,\dots,\tilde{b}_{\rank v_0 + \dim \xi} \in M\cap \sigma^{\perp}$ 
  form a basis of the dual of $\Span P_\xi$ with $\tilde{b}_i:= \tilde{a}_i$ ($1 \leq i \leq \rank v_0$). 
  Then 
  we have a commutative diagram 
  $$ \xymatrix{
  (M\cap \sigma^{\perp} )_\Q 
  \ar[r]^-{\cong} \ar[d] &
   \Q\langle \tilde{b}_i \rangle_{i=1}^{\rank v_0 +\dim \xi} 
  \oplus  (M\cap \sigma^{\perp} \cap  P_\xi^{\perp})_\Q 
  \ar[d] \\
  (k(\supp(v_x))^{\times})_\Q 
  \ar[r]^-{\cong} &
    \Q\langle b_i \rangle_{i=1}^{\rank v_0 +\dim \xi} \oplus \Ker (v_x\otimes \Q \colon (k(\supp(v_x))^{\times})_\Q \to \Q \cdot \Gamma_{v_x}),
  }$$
  where the support $\supp (v_x)$ is taken in $\A^1_{K(X)_{v_0}}$,
  and $b_i := \varphi (\tilde{b}_i) \in k(\supp (v_x))$.
  By Lemma \ref{stalk of F^p} and results in Subsection \ref{sub:resolutions},
  this
  induces a commutative diagram 
  \begin{align}
   \xymatrix{
  F^p_{``\overline{(\pi^{\circ})^{-1}([v_0])}''}({\xi})
  \ar[r]^-{\cong} \ar@{^(->}[d] &
    \bigoplus_{i=0}^p 
    \bigwedge^i 
   \Q\langle \tilde{b}_i \rangle_{i=1}^{\rank v_0 +\dim \xi} 
  \otimes 
  \overline{F}^{p-i}_{``\overline{(\pi^{\circ})^{-1}([v_0])}''}({\xi})
  \ar@{^(->}[d] \\
    \F^p_{(X \times \A^1)^{\circ}, v_x}
  \ar[r]^-{\cong} &
    \bigoplus_{i=0}^p 
    \bigwedge^i 
     \Q\langle b_i \rangle_{i=1}^{\rank v_0 +\dim \xi} 
    \otimes K^{p-i}_T (\kappa (v_x) /K).  
  } \label{eq compatibility Fp KTp decomposition}
  \end{align}
  \end{rem}

  We put $ \tilde{\pi}_{\xi_1,\xi_2} := \tilde{b}_{\rank v_0 +1} \in M'$ for $\xi = \xi_2$.
  Then by a decomposition in (\ref{eq compatibility Fp KTp decomposition}) for $\xi =\xi_1, \xi_2$,
   the morphism 
  $i_{\xi_1 \subset \xi_2}$
  is given by a morphism 
  \begin{align}
    \overline{F}^{p-i}_{``\overline{(\pi^{\circ})^{-1}([v_0])}"}(\xi_1) 
   \to  
   \overline{F}^{p-i}_{``\overline{(\pi^{\circ})^{-1}([v_0])}"}(\xi_2)
   \oplus
  ( \Q \langle \tilde{\pi}_{\xi_1,\xi_2} \rangle \otimes 
   \overline{F}^{p-i-1}_{``\overline{(\pi^{\circ})^{-1}([v_0])}"}(\xi_2), \label{eq morphism Fp overline subset}
  \end{align}
  which can be described using Remark \ref{new Fp bar and Fp KT} as follows.

  \begin{rem}\label{compatibility local description type 2}
  We assume $\xi_1 \in N_\R$. 
  We take $ \tilde{\pi}_{\xi_1,\xi_2}  \in M \cap P_{\xi_1}^{\perp}$. 
  Then morphism (\ref{eq morphism Fp overline subset})
  is given by a decomposition
  $$
  (M \cap P_{\xi_1}^{\perp}) \otimes \Q 
  \cong 
  \Q \langle \tilde{\pi}_{\xi_1,\xi_2} \rangle \oplus 
  (M \cap P_{\xi_2}^{\perp}) \otimes \Q .$$
  We put $\pi_{\xi_1,\xi_2} \in K(X)(T)$ the image of $ \tilde{\pi}_{\xi_1,\xi_2}$. 
  Then we have a morphism
  \begin{align*}
     K_T^{p-i}(\kappa(v_{\xi_1})/K)     
   \to
   &  K_T^{p-i}(\kappa(u_{\xi_1,\xi_2})/K) \oplus (\Q\langle \pi_{\xi_1,\xi_2} \rangle \otimes K_T^{p-i-1}(\kappa(u_{\xi_1,\xi_2})/K) ).
  \end{align*}
  	given by a  decomposition 
  \begin{align*}
  & \Ker (v_{\xi_1}\otimes \Q \colon (K(X)(T)^{\times})_\Q \to \Q \cdot \Gamma_{v_{\xi_1}}) \\
  = &  \Q\langle \pi_{\xi_1,\xi_2} \rangle \oplus
  \Ker (u_{\xi_1,\xi_2}\otimes \Q \colon (K(X)(T)^{\times})_\Q \to \Q \cdot \Gamma_{u_{\xi_1,\xi_2}}).
  \end{align*}
  By construction, this morphism can be written as 
  $$
    ( s_{\overline{u_{\xi_1,\xi_2}}}^{\overline{\pi_{\xi_1,\xi_2}}}, \partial_{\overline{u_{\xi_1,\xi_2}}}) 
    \colon 
     K_T^{p-i}(\kappa(v_{\xi_1})/K)      
     \to 
     K_T^{p-i}(\kappa(u_{\xi_1,\xi_2})/K) \oplus K_T^{p-i-1}(\kappa(u_{\xi_1,\xi_2})/K)  $$
  by an identification 
  $\Q \langle \pi_{\xi_1,\xi_2} \rangle \ni \pi_{\xi_1,\xi_2} 
  \mapsto \overline{u_{\xi_1,\xi_2}} ( \overline{ \pi_{ \xi_1,\xi_2} } ) \in \Q,$
  where 
  $\overline{u_{\xi_1,\xi_2}}$ is the normalized discrete valuation 
  of $\kappa (v_{\xi_1})$ corresponding to $u_{\xi_1,\xi_2}$
  in the sense of Remark \ref{rem;val;vert;special;resi;fld;val},
  the element $ \overline{ \pi_{ \xi_1,\xi_2} } \in \kappa (v_{\xi_1})$
  is the reduction, 
  and 
  $s_{\overline{u_{\xi_1,\xi_2}}}^{\overline{\pi_{\xi_1,\xi_2}}}
  := \partial_{\overline{u_{\xi_1,\xi_2}}} \circ  \overline{\pi_{\xi_1,\xi_2}}\wedge -$.
  By construction,  morphism (\ref{eq morphism Fp overline subset}) is compatible with this morphism of  $K_T$
  through diagram (\ref{eq compatibility Fp KTp decomposition}).
  \end{rem}

  \begin{rem}\label{local computation type 1}
  We assume $\xi_1 \notin N_\R$.
  Then morphism (\ref{eq morphism Fp overline subset})
  factors through 
  \begin{align*}
   \overline{F}^{p-i}_{``\overline{(\pi^{\circ})^{-1}([v_0])}"}(\xi_1) 
  \hookrightarrow & 
  \overline{F}^{p-i}_{``\overline{(\pi^{\circ})^{-1}([v_0])}"}(\xi_2)  \\
  \subset  &
  \overline{F}^{p-i}_{``\overline{(\pi^{\circ})^{-1}([v_0])}"}(\xi_2)  
  \oplus (\Q \langle \tilde{\pi}_{\xi_1,\xi_2} \rangle \otimes_\Q
  \overline{F}^{p-i-1}_{``\overline{(\pi^{\circ})^{-1}([v_0])}"}(\xi_2) .  
  \end{align*}
  Since $\Lambda$ (and hence $\Xi_1^{\circ}$) is sufficiently fine,
  by 
  Lemma \ref{mpd near type 1 point}
  and Lemma \ref{rem;P1;adic;formal;model} 
  (\ref{enu:3:rem:P1:adic:formal:model}),
  for $$u \in \Sk_{\varphi} ( (\pi^{\circ})^{-1} ([v_0]) ) 
  \cap (\Trop^{\ad}_{\Xi_1^{\circ}} )^{-1} (\xi_2)$$
  of height $2$, 
  we have $\kappa (u) \cong \kappa (v_{\xi_1})$.
  In paticular, we have 
  $ K_T^{p-i}(\kappa(v_{\xi_1})/K)    \cong
  K_T^{p-i}(\kappa(u)/K)  $.
  By construction,
  the above morphism 
   $$ \overline{F}^{p-i}_{``\overline{(\pi^{\circ})^{-1}([v_0])}"}(\xi_1) 
  \hookrightarrow 
  \overline{F}^{p-i}_{``\overline{(\pi^{\circ})^{-1}([v_0])}"}(\xi_2)  $$
  is compatible with this isomorphism
  through diagram (\ref{eq compatibility Fp KTp decomposition}).
  \end{rem}
   
 \subsection{Proof of $\F_{X^{\circ},v_0}^r \cong \pi^{\circ}_* \F_{(X\times \A^1)^{\circ},v_0}^r$}
  In this subsection, we shall  show $\F_{X^{\circ},v_0}^r \cong \pi^{\circ}_* \F_{(X\times \A^1)^{\circ},v_0}^r$.
  Let  $s_0 \colon X \to X \times \A^1$ be the section at $0$.
  Since the composition  $ \pi \circ s_0$ is the identity map,
  the morphism
  $\F_{X^{\circ},v_0}^r \rightarrow \pi^{\circ}_* \F_{(X\times \A^1)^{\circ},v_0}^r $
  is injective. 
  We shall show surjectivity. 
  Let 
  $$\alpha \in \F_{ (X \times \A^1)^{\circ } }^r 
  |_{ (\pi^{\circ})^{-1} (v_0) } 
  ( (\pi^{\circ})^{-1} (v_0)) $$ 
  be such that $s_o^* ( \alpha )=0$.
  We shall show that $\alpha =0$.
  There is $\varphi \in I_{X \times \A^1}$ 
  such that 
  $\alpha$ is the image of some element $\tilde{\alpha}$ in 
  $$H^0 (C^{r,*}_{``\overline{(\pi^{\circ})^{-1}([v_0])}"} (\Trop (\varphi ((\pi^{\circ})^{-1} (v_0) )) )).$$

  The following is the main tool in proof.
  Since  tropial Milnor $K$-groups form a cycle module (Theorem \ref{trKcymod}), by
  Lemma \ref{rem;P1;adic;formal;model} (\ref{enu:1:rem:P1:adic:formal:model})  and  \cite[Proposition 2.2]{Ros96}, 
  for a point $w \in (\pi^{\circ})^{-1}([ v_0])$ of type $2$, 
  we have 
  	\begin{align}
   & \Ker \big(  K_T^k (\kappa(w)/K)\xrightarrow{  (\partial_{\overline{u}})_{u \in  \overline{\{w\}} \setminus \{w, u_{w,\infty}\}  } } 
    \bigoplus_{u\in  \overline{\{w\}} \setminus \{w, u_{w,\infty}\}} K_T^{k-1}(\kappa(u)/K) \big) \notag \\
    \cong  & K_T^k (\kappa (w) \cap \kappa (v_0)^{\alg} /K),  \label{R0 A1 homtopy local computation}
  	\end{align} 
  where $u \in (\pi^{\circ})^{-1} ([v_0])$ runs through all non-trivial specializations of $w$ except for $u_{w,\infty}$. 
  (Recall that when $w$ corresponds to $(a,r)$, the valuation $u_{w,\infty}$ is a specialization of $w$ corresponding to $(a,r,-1)$.)

  For $\xi \in \Xi_1^{\circ}$, 
  we put 
  $\tilde{\alpha}_\xi $
  its restriction to 
  $$ H^0 (C^{r,*}_{``\overline{(\pi^{\circ})^{-1}([v_0])}"} ( \relint \xi ) )
  \cong F^{r}_{``\overline{(\pi^{\circ})^{-1}([v_0])}"}(\xi).$$
  For $v \in (\pi^{\circ})^{-1}( [v_0] )$, 
  we put $\alpha_v \in \F^r_{(X\times \A^1)^{\circ},v}$ the stalk of $\alpha$ at $v$.

  We shall use 
    Remark \ref{local descriptions basics:label},
   Remark \ref{compatibility local description type 2}, 
   and 
     Remark \ref{local computation type 1} freely.
  \begin{lem}\label{Fu,pi}
  For $\xi_1 \in \Xi_1^{\circ} \cap N_\R$ of dimension $1$, 
  via the decomposition in Lemma \ref{stalk of F^p},
  we have 
  $$\alpha_{v_{\xi_1}} \in 
  \bigoplus_{i=0}^r  
      \bigwedge^i \Q\langle a_j \rangle_{j=1}^{\rank v_0} 
    \otimes K^{r-i}_T (\kappa ( v_{\xi_1} ) \cap \kappa (v_0)^{\alg} /K). $$ 
  \end{lem}
  \begin{proof}
  	Suppose the assertion does not hold, i.e., 
    by
     (\ref{R0 A1 homtopy local computation}), 
     there exists a polyhedron 
    $\xi_2 \in \Xi_1^{\circ}$ of dimension $1$
    such that 
  $$\tilde{\alpha}_{\xi_2} \notin 
  \bigoplus_{i=0}^r  
      \bigwedge^i \Q\langle \tilde{a}_j \rangle_{j=1}^{\rank v_0} 
    \otimes \overline{F}^{r-i}_{``\overline{(\pi^{\circ})^{-1}([v_0])}"}(\xi_2).$$
     Since $\tilde{\alpha}$ is a cocycle, 
     we may assume that 
  $\partial_{\overline{u'}}\alpha_{v_{\xi'_1}} = 0$ 
  for a non-trivial specialization $u'\in  (\pi^{\circ})^{-1}([v_0]) \setminus \{u_{v_{\xi'_1},\infty}\} $ of $v_{\xi'_1}$,
  where 
     $\xi'_1 \in \Xi_1^{\circ}$ is the $0$-dimensional polyhedron contained in $\xi_2 \cap N_\R$
     such that  $u_{\xi'_1,\xi_2}= u_{v_{\xi'_1},\infty }$. 
  However, 
  this contradicts to (\ref{R0 A1 homtopy local computation}).
   Hence the assertion holds.
  \end{proof}

  \begin{prp}
  We have $\alpha=0$.
  \end{prp}
  \begin{proof}
  	Suppose $\alpha \neq 0$.
  	Then there exists 
    $ \xi_1 \in \Xi_1^{\circ}$
    of dimension $0$ with $\xi_1 \in N_\R$
   such that 
  $\tilde{\alpha}_{\xi_1} \neq 0 .$
  We put $\xi_2 := \Trop^{\ad}_{\Xi_1^{\circ}} (
   \varphi (u_{v_{\xi_1},0})
    )$
    a $1$-dimensional polyhedron containing $\xi_1$
     (see above Lemma \ref{rem;P1;adic;formal;model} for $u_{v_{\xi_1},0}$).
  Since $s_0^* (\alpha) =0$ and $\tilde{\alpha}$ is a cocycle, 
  we may assume that $\tilde{\alpha}_{\xi_2}=0$.
  However, 
  by 
  $\kappa (v_{\xi_1}) \cap \kappa (v_0)^{\alg}  =\kappa (u_{v_{\xi_1},0})$
  (Corollary \ref{enu:2:rem:P1:adic:formal:model})
  and 
  Lemma \ref{Fu,pi}, 
  this is a contradiction.
   Hence the assertion holds.
  \end{proof}
  
  Consequently, we have 
  \begin{cor}\label{pi_*=F}
  $\F_{X^{\circ},v_0}^r \cong \pi^{\circ}_* \F_{(X\times \A^1)^{\circ},v_0}^r$.
  \end{cor} 

 \subsection{Proof of $R^1 \pi^{\circ}_* \F_{(X \times \A^1)^{\circ},v_0}^r=0 $}
  In this subsection, 
  we shall show $R^1 \pi^{\circ}_* \F_{(X \times \A^1)^{\circ},v_0}^r=0 $.
  More presicely, 
  for $\varphi \in I_{X \times \A^1}$
  and a cocycle
  $$ \beta  
  \in 
  C^{r,1}_{\Trop} ( \Trop (\varphi( (\pi^{\circ})^{-1} ([v_0])) )) $$
  (i.e., $\partial \beta =0$),
  we shall show that there exists $\varphi_{\new} \in I_{X \times \A^1}$
  with a morphism $\psi \colon \varphi \to \varphi_{\new}$ in $I_{X \times \A^1}$, 
  such that 
  $$\overline{\psi^* (\beta)} =0\in 
  H^{r,1}_{\Trop} ( \Trop (\varphi_{\new} ( (\pi^{\circ})^{-1} ([v_0])) )).$$

  We shall use 
    Remark \ref{local descriptions basics:label},
   Remark \ref{compatibility local description type 2}, 
   and 
     Remark \ref{local computation type 1} freely.

  \begin{lem}\label{R1Fp=0 induction 1 local}
  For a polyhedron $\xi_1 \in \Xi_1^{\circ} \cap N_\R$ of dimension $0$,
  there exist $\varphi_{\new} = (\varphi, (g_i)_{i} )\in I_{X \times \A^1}$
  given by finitely many polynomials $g_i$ of degree $< \mpd (v_{\xi_1})$ 
  and 
  $$ \beta'_{\xi_1} 
  \in 
  F^{r}_{ ``\overline{(\pi^{\circ})^{-1}([v_0])}"} 
  (\Trop (\varphi_{\new} ( 
     v_{\xi_1}
    )))$$
  such that 
  for 
  $$\xi_{\new,2} \in \Xi_{\new, 1}^{\circ} 
  \setminus 
  \{ \Trop_{\Xi_{\new,1}^{\circ}}^{\ad} (\varphi_{\new} ( u_{v_{\xi_1}, \infty }
   )) \} $$
  of dimension $1$ containing $\Trop (\varphi_{\new} ( v_{\xi_1}   )) $,
    we have 
  $$\psi^*(\beta)_{\xi_{\new,2}} + 
   i_{\Trop (\varphi_{\new} (v_{\xi_1})) \subset \xi_{\new,2} }
    (\beta'_{\xi_1} )
  \in 
  \bigoplus_{i=0}^r  
      \bigwedge^i \Q\langle \tilde{a}_j \rangle_{j=1}^{\rank v_0} 
    \otimes \overline{F}^{r-i}_{``\overline{(\pi^{\circ})^{-1}([v_0])}"}(\xi_{\new,2}) ,
     $$
  where 
  for example, $\Xi_{\new,1}^{\circ}$ is defined similarly to $\Xi_1^{\circ}$,
  $$\psi^*(\beta)_{\xi_{\new,2}}
    := \psi^*(\beta) ([\xi_{\new,2}]) \in F^p_{``\overline{(\pi^{\circ})^{-1}([v_0])}"}(\xi_{\new,2}) $$
  (where the orientation of the chain $[\xi_{\new,2}]$ is inward from $\Trop(\varphi_{\new} (v_{\xi_1}))$),
  and
  $\psi \colon \varphi \to \varphi_{\new}  $ in $I_{X \times \A^1}$ is given by a projection.
  \end{lem}
  \begin{proof}
  Since  tropial Milnor $K$-groups form a cycle module (Theorem \ref{trKcymod}), by
  Lemma \ref{rem;P1;adic;formal;model} (\ref{enu:1:rem:P1:adic:formal:model})  and  \cite[Proposition 2.2]{Ros96}, 
  for a point $w \in (\pi^{\circ})^{-1}([ v_0])$ of type $2$, 
  we have a surjection
    \begin{align*}
   &   K_T^k (\kappa(w)/K)\xrightarrow{  (\partial_{\overline{u}})_{u \in  \overline{\{w\}} \setminus \{w, u_{w,\infty}\}  } } 
    \bigoplus_{u\in  \overline{\{w\}} \setminus \{w, u_{w,\infty}\}} K_T^{k-1}(\kappa(u)/K) ,
    \end{align*} 
  where $u \in (\pi^{\circ})^{-1} ([v_0])$ runs through all non-trivial specializations of $w$ except $u_{w,\infty}$. 
   Hence the assertion follows form Lemma \ref{rem;P1;adic;formal;model}
    (\ref{enu:3:rem:P1:adic:formal:model}).
  \end{proof}
  
  \begin{prp}\label{lem;prf;infty;Rpi}
  There exist $\varphi_{\new} \in I_{X \times \A^1}$
  with a morphism $\psi \colon \varphi \to \varphi_{\new}$ in $I_{X \times \A^1}$ 
  and a cocycle
  $$ \beta' 
  \in 
  C^{r,1}_{\Trop} ( \Trop (\varphi_{\new}( (\pi^{\circ})^{-1} ([v_0])) )) $$
  such that 
  $$\overline{\beta'} = \overline{\psi^* (\beta)} \in 
  H^{r,1}_{\Trop} ( \Trop (\varphi_{\new} ( (\pi^{\circ})^{-1} ([v_0])) ))$$
  and 
  for $ \xi_{\new,2} \in \Xi_{\new,1}^{\circ}$ of dimension $1$,
  	we have 
  $$ \beta'_{\xi_{\new,2}} 
  \in 
  \bigoplus_{i=0}^r  
      \bigwedge^i \Q\langle \tilde{a}_j \rangle_{j=1}^{\rank v_0} 
    \otimes \overline{F}^{r-i}_{``\overline{(\pi^{\circ})^{-1}([v_0])}"}(\xi_{\new,2})
    . $$
  \end{prp}
  \begin{proof}
  This follows from Lemma \ref{tropicalization A1 as retraction}  and Lemma \ref{R1Fp=0 induction 1 local}
  by induction on $\mpd (v_{\xi_1})$ in the descending order.
  \end{proof}

  By Proposition \ref{lem;prf;infty;Rpi},
  we may assume that 
  for $ \xi_2 \in \Xi_1^{\circ}$ of dimension $1$,
    we have 
  $$\beta_{\xi_2}
  \in 
  \bigoplus_{i=0}^r  
      \bigwedge^i \Q\langle \tilde{a}_j \rangle_{j=1}^{\rank v_0} 
    \otimes \overline{F}^{r-i}_{``\overline{(\pi^{\circ})^{-1}([v_0])}"}(\xi_2)
    . $$

  \begin{lem}\label{R1Fp=0 induction 2 local}
  For a polyhedron $\xi_1 \in \Xi_1^{\circ}$ of dimension $0$ with $\xi_1 \in N_\R$ (resp. $\xi_1 \notin N_\R$),
  there exist $\varphi_{\new} = (\varphi, (g_i)_{i} )\in I_{X \times \A^1}$
  given by finitely many polynomials $g_i$ of degree $< \mpd (v_{\xi_1})$ 
  (resp. $< [K(X)_{v_0}(v_{\xi_1}): K(X)_{v_0}]$)
  and 
  $$ \beta'_{\xi_1} 
  \in 
  F^{r}_{ ``\overline{(\pi^{\circ})^{-1}([v_0])}"} 
  (\Trop (\varphi_{\new} ( 
     v_{\xi_1}
    )))$$
  such that 
  we have 
  \begin{align*}
  \psi^*(\beta)_{\Trop^{\ad}_{\Xi_{\new,1}^{\circ}} (\varphi_{\new} ( u_{v_{\xi_1}, \infty }
   ))}=  &
   i_{\Trop (\varphi_{\new} (v_{\xi_1})) \subset
   \Trop_{\Xi_{\new,1}^{\circ}}^{\ad} (\varphi_{\new} ( u_{v_{\xi_1}, \infty }) )
   }
    (\beta'_{\xi_1} )  \\
  \in  &
  \bigoplus_{i=0}^p  
      \bigwedge^i \Q\langle \tilde{a_j} \rangle_{j=1}^{\rank v_0} 
    \otimes \overline{F}^{p-i}_{``\overline{(\pi^{\circ})^{-1}([v_0])}"}
    ( \Trop^{\ad}_{\Xi_{\new,1}^{\circ}} (\varphi_{\new} ( u_{v_{\xi_1}, \infty }
   )) )
  \end{align*}
  when $u_{v_{\xi_1}, \infty} \in (\pi^{\circ})^{-1}([v_0])$,
     and 
  for 
  $$\xi_{\new,2} \in \Xi_{\new, 1}^{\circ} 
  \setminus 
  \{ \Trop_{\Xi_{\new,1}^{\circ}}^{\ad} (\varphi_{\new} ( u_{v_{\xi_1}, \infty }
   )) \} $$
  of dimension $1$ containing $\Trop (\varphi_{\new} ( v_{\xi_1}   )) $,
    we have 
  $$ 
   i_{\Trop (\varphi_{\new} (v_{\xi_1})) \subset \xi_{\new,2} }
   (\beta'_{\xi_1} )
  \in 
  \bigoplus_{i=0}^r  
      \bigwedge^i \Q\langle \tilde{a}_j \rangle_{j=1}^{\rank v_0} 
    \otimes \overline{F}^{r-i}_{``\overline{(\pi^{\circ})^{-1}([v_0])}"}(\xi_{\new,2})
     $$
  (resp. such that we have 
  \begin{align*}
  \psi^*(\beta)_{\xi_{\new,2}}
   =  &
   i_{\Trop (\varphi_{\new} (v_{\xi_1})) \subset \xi_{\new,2} }
    (\beta'_{\xi_1} )  \\
  \in  &
  \bigoplus_{i=0}^r  
      \bigwedge^i \Q\langle \tilde{a}_j \rangle_{j=1}^{\rank v_0} 
    \otimes \overline{F}^{r-i}_{``\overline{(\pi^{\circ})^{-1}([v_0])}"}
    (\xi_{\new,2})
  \end{align*}
  where $\xi_{\new,2} \in \Xi_{\new,1}^{\circ}$
  is the unique $1$-dimensional polyhedron containing $\Trop (\varphi_{\new } (v_{\xi_1}))$),
  where $\psi \colon \varphi \to \varphi_{\new}$ is a projection.
  \end{lem}
  \begin{proof}
  This follows from 
  Lemma \ref{rem;P1;adic;formal;model}
    (\ref{enu:3:rem:P1:adic:formal:model}) and 
    Corollary \ref{enu:2:rem:P1:adic:formal:model}.
     \end{proof}

  \begin{prp}
  There exists $\varphi_{\new} \in I_{X \times \A^1}$
  with a morphism $\psi \colon \varphi \to \varphi_{\new}$ in $I_{X \times \A^1}$, 
  such that 
  $$\overline{\psi^* (\beta)} =0\in 
  H^{r,1}_{\Trop} ( \Trop (\varphi_{\new} ( (\pi^{\circ})^{-1} ([v_0])) )).$$
  \end{prp}
  \begin{proof}
  This follows from Lemma \ref{tropicalization A1 as retraction}
  and  Lemma \ref{R1Fp=0 induction 2 local}
  by induction on $\mpd (v_{\xi_1})$ in the increasing order.
  \end{proof}
  
  Consequently, we have 
  \begin{cor}\label{R1 pi_* =0}
  $R^1 \pi^{\circ}_* \F_{(X \times \A^1)^{\circ},v_0}^r=0 $ 
  \end{cor}

\section{The existence of corestriction maps}\label{section finite fields}
  In this section, we show the existence of corestriction maps (Proposition \ref{corestriction}), which is used to prove the main theorem (Theorem \ref{thm,trop,main}) over finite fields.
  Let $L/K$ be an extension of trivially valued finite fields,
  and $X$ be a smooth irreducible algebraic variety over $K$.
  We put $\pi \colon X_L \to X$ the base change.
  
  Since $R^i \pi_* \F^p_{X_L} =0 $ for $i\geq 1$,
  Proposition \ref{corestriction} follows from the following.
  \begin{prp}\label{corestriction of Fp}
  There is a morphism
  $$ \cores \colon \pi_* \F^{p}_{X_L} \to  \F^{p}_{X}$$
  of sheaves on $X^{\Ber}$
  such that 
  $\cores \circ \res = [L:K]$, 
  where 
  $ \res \colon \F^{p}_{X} \to \pi_* \F^{p}_{X_L}$
  is the natural morphism.
  \end{prp}
  
  Let $v \in X^{\Ber}$ be a valuation. 
  We put $x:= \supp v \in X$ and  $r:= \rank v$. 
  We fix elements
  $a_1,\dots, a_{ r} \in k(x)^{\times}$ 
  such that $v(a_1), \dots, v(a_r )$ form a basis of a $\Q$-vector space $\Q \cdot \Gamma_v$.
  We put 
  $$\cores_v \colon \bigoplus_{w \in \pi^{-1}(v)}
  \F^p_w \to \F^p_v$$
  the sum of morphisms
  $$  \F^p_w \cong
   \bigoplus_i \bigwedge^i \Q \langle a_j \rangle_j
   \otimes K_T^{p-i} (\kappa (w)/K)
   \to 
   \bigoplus_i \bigwedge^i \Q \langle a_j \rangle_j
   \otimes K_T^{p-i} (\kappa (v)/K)
   \cong \F^p_v$$
  given by identity maps on $\Q \langle a_j \rangle_j$ 
  and norm homomophisms $K_T^{p-i} (\kappa (w)/K) \to K_T^{p-i} (\kappa (v)/K)$,
  where the isomorphisms are given in Lemma \ref{stalk of F^p}.
  
  Since every section of $ \pi_* \F^{p}_{X_L}$
  is locally the image of an element of 
  $\pi_* \supp^* \bigwedge^p (\O_{X_L}^{\times})_\Q$, 
  Proposition \ref{corestriction of Fp} 
  follows from 
  by the following,
  where $\supp \colon X_L^{\Ber} \to X_L$ is the map of taking supports of valuations, and $\O_{X_L}^{\times}$ is the sheaf of invertible algebraic functions.

  \begin{prp}\label{commutativity of cores}
  We have a commutative diagram 
  \begin{align*}
  \xymatrix{
   \Im ( 
  \bigwedge^p (\O_{X,x} \otimes_K L)^{\times }_\Q \to 
  	\bigoplus_{ Y \subset X_L } K_T^p (K(Y)/K) ) 
  	\ar[r] \ar@{->>}[d] &
  \Im ( 
  \bigwedge^p (\O_{X,x} ^{\times })_\Q \to
  	K_T^p (K(X)/K) ) 
  	 \ar@{->>}[d] \\
  	\bigoplus_{ y \in \pi^{-1}(x) } K_T^p (k(y)/K)  
  	\ar[r] \ar@{->>}[d] &
  	K_T^p (k(x)/K)   \ar@{->>}[d] \\
  \bigoplus_{w \in \pi^{-1} (v)}
   \F^p_w
   \ar[r] &
   \F^p_v ,
  }
  \end{align*}
  where 
  $Y$ runs through all irreducible components of $X_L$, 
  the horizontal arrows are sums of norm homomophisms and $\cores_v$, 
  and the vertical arrows are 
  factorizations of natural maps 
  from $\bigwedge^p (\O_{X,x} ^{\times } \otimes_K L )_\Q $ and 
  $\bigwedge^p (\O_{X,x} ^{\times })_\Q $.
  \end{prp}
  
  The higher vertical arrows  in Proposition \ref{commutativity of cores} 
  are well-defined by Remark \ref{rem;val;vert;special;resi;fld;val} and the existence of extensions of valuations under extensions of fields.
  By \cite[Proposition 5.7]{Ker09} (compatibility of norm homomophisms of Milnor K-groups of semi-local rings), 
  the first horizontal arrow in Proposition \ref{commutativity of cores} 
  is well-defined. 
  Consequently, all morphisms in Proposition \ref{commutativity of cores} are well-defined.

  \begin{rem}\label{surjectivity in commutativity of cores}
  The surjectivity of the vertical arrows 
  except for the first lower one
  in Proposition \ref{commutativity of cores}
  is trivial. 
  The surjectivity of the first lower vertical arrow can be seen as follows.
  Let $y \in \pi^{-1}(x)$ be a point.
  By \cite[Chapter VI, Section 7.2, Theorem 1, Corollary 1]{Bou72}, 
  we have a surjection 
  $$ k(y)^{\times} \twoheadrightarrow  
  \bigoplus_{ 
  	\substack{w \in \pi^{-1}(v) 
  	          \\    \supp w =y }} 
  \Gamma_w.$$
  Hence we have a splitting 
  $$ k(y)^{\times}_\Q \cong \bigoplus_{ 
    \substack{w \in \pi^{-1}(v) 
              \\    \supp w =y }} 
  \Gamma_{w,\Q}
  \oplus \bigcap_{\substack{w \in \pi^{-1}(v) 
              \\    \supp w =y }}   \O_{w,\Q}^{\times} .$$
  Since $\bigcap_{ 
    \substack{w \in \pi^{-1}(v) 
        \\    \supp w =y }}   \O_w$ is a semi-local ring,
  we have a surjection 
  $$ \bigcap_{ 
    \substack{w \in \pi^{-1}(v) 
              \\    \supp w =y }}   \O_w^{\times} 
          \twoheadrightarrow 
          \bigoplus_{ 
    \substack{w \in \pi^{-1}(v) 
              \\    \supp w =y }}  \kappa (w)^{\times} . $$
  Hence the first lower vertical arrow in Proposition \ref{commutativity of cores}
   is surjective.
  \end{rem}
  
  In the rest of this section, we shall show the commutativity of the second squeare of the diagram in Proposition  \ref{commutativity of cores}.
  The commutativity of the first one follows similarly, and we omit it. 
  These complete proof of Proposition \ref{commutativity of cores} (and hence Proposition \ref{corestriction of Fp}).
  
  Recall that 
  for $w \in \pi^{-1}(v)$ and $y:= \supp w$,
  the morphism 
  $$ K_T^p (k(y)/K)   \to \F^p_{X_L,w} \cong 
  \bigoplus_i 
  \bigwedge^i \Q \langle a_j \rangle_j 
  \otimes
    K_T^{p-i} (\kappa (w) /K) 
    $$
  is given by  a decomposition
  $k(y)^{\times}_\Q \cong \Q \langle a_j \rangle_j \oplus (\O_{w}^{\times})_\Q$
  (Lemma \ref{stalk of F^p}).
  We also have 
  $\O_{w}^{\times} \cong \underrightarrow{\lim}_{X'} \O_{  X'_L, \Cent_{X'_L} w}^{\times} $, 
  where $X'$ runs through all proper algebraic varieties 
  whose function field is $k(x)$.
  The same statements hold for $x$ and $v$ instead of $y$ and $w$.
  Hence it suffices to show that 
  the restriction 
  $$\xymatrix{
    \bigoplus_{ y \in \pi^{-1}(x) } 
  K_T^p (k(y)/K)|_{X'}
    \ar[r] \ar[d] &
  K_T^p (k(x)/K)|_{X'}
      \ar[d] \\
  \bigoplus_{w \in \pi^{-1} (v)}
  \bigoplus_i 
  \bigwedge^i \Q \langle a_j \rangle_j 
  \otimes
    K_T^{p-i} (\kappa (w) /K)  
    \ar[r] &
  \bigoplus_i 
  \bigwedge^i \Q \langle a_j \rangle_j 
  \otimes
    K_T^{p-i} (\kappa(v)/K)  
  }$$
  of the second squeare of the diagram in Proposition  \ref{commutativity of cores}
  is commutative, 
  where 
  we put  $x':= \Cent_{X'}v$ the center of $v$, 
  by abuse of notation, let $\pi$ also denote the base change $\pi \colon X'_L \to X'$,
  \begin{align*}
  K_T^p (k(y)/K)|_{X'} := & \bigcap_{ \substack{ y' \in \pi^{-1}(x') \cap  \overline{ \{y \}} } }
    \Im \big( \bigwedge^p 
    (\Q \langle a_j \rangle_j \oplus (\O_{X'_L, y'}^{\times})_\Q ) \to K_T^p (k(y)/K)  
      \big) ,   \\
  K_T^p (k(x)/K)|_{X'} := &
  \Im \big( \bigwedge^p  
      (\Q \langle a_j \rangle_j \oplus
    (\O_{X', x'}^{\times})_\Q ) 
    \to 
    K_T^p (k(x)/K  ) \big) .
  \end{align*}

  We may assume that there is
  an open neighborhood $U' \subset X'$ of $x'$, 
  a basis $a_j' \in \O (U') $ 
  of $\Q \langle a_j \rangle_{j=1}^r $,
  and 
  an $i$-codimensional point $x'_i \in U'$ ($ 0 \leq i \leq r$)
  such that 
  $$\overline{ \{x'_0\} } \supset 
  \overline{\{ x'_1 \} } \supset
  \dots \supset 
  \overline{ \{ x'_r\} } \ni x'$$
  and 
  $$
  \partial_{x'_{r -1}}^{x'_{r}}
  ( \dots (\partial_{x'_0}^{x'_1}
  (a'_1 \wedge \dots \wedge a'_{r})) \dots ) 
  \neq 0 
  \in K_T^0 (k(x'_{r})/K) \cong \Q.$$
  (Recall that $
  \partial_{x'_{i -1}}^{x'_{i}} 
  $ is given by residue homomorphisms (Subsection \ref{Subsection tropical K group cycle module}).)

  It suffices to prove the following two Lemmas.
  
  \begin{lem}\label{compatibility of norm hom and certain specializations}
  We have a commutative diagram
  \begin{align}
  \xymatrix{
    \bigoplus_{ y \in \pi^{-1}(x) } K_T^p (k(y)/K)  |_{X'}
    \ar[r] \ar[d] &
    K_T^p (k(x)/K) |_{X'}   \ar[d] \\
    \bigoplus_{ y'_r \in \pi^{-1}(x_r') } 
  \bigoplus_i 
  \bigwedge^i \Q \langle a_j \rangle_j 
  \otimes
    K_T^{p-i} (k(y'_r)/K)  
    \ar[r]  &
  \bigoplus_i 
  \bigwedge^i \Q \langle a_j \rangle_j 
  \otimes
    K_T^{p-i} (k(x'_r)/K)   ,
  }\label{diagram for cores 2}
  \end{align}
  where  the vertical arrows are the natural morphisms, and
  the horizontal arrows are given by identity maps on $\Q \langle a_j \rangle_j$ and norm homomophisms.
  \end{lem}
  \begin{proof}
  For $i$ and $J=\{j_1 , \dots, j_{i } \} \subset [r]:=\{1,\dots,r\}$, 
  we put 
  $$K_T^p(k(x)/K) \to
   \Q \langle a'_{j_1} \wedge \dots \wedge a'_{j_{ i }} \rangle
   \otimes 
  K_T^{p-i} (k(x'_r)/K)$$
  the morphism given by 
  $$f \mapsto 
  \epsilon_J
  a'_{j_1} \wedge \dots \wedge a'_{j_{ i }} \otimes
  \frac{
  \partial_{x'_{r-1}}^{x'_{r}}
  ( \dots (\partial_{x'_{0}}^{x'_{1}}
  ( \bigwedge_{i \in [r] \setminus J} a'_i \wedge f)) \dots )
  }{  
  \partial_{x'_{r -1}}^{x'_{r}}
  ( \dots (\partial_{x'_0}^{x'_1}
  (a'_1 \wedge \dots \wedge a'_{r})) \dots ) 
  }
  $$
  (suitable $\epsilon_J \in \{ \pm 1 \}$).
  The second vertical arrow  
   in (\ref{diagram for cores 2})
   is the restriction of the morphism
  $$K_T^p (k(x)/K) \to \bigoplus_i 
  \bigwedge^i \Q \langle a_j \rangle_j
   \otimes	K_T^{p-i} (k(x'_r)/K) $$
   defined as the sum of the above morphisms
   by identifying 
  \begin{align*}
  \bigwedge^i \Q \langle a_j \rangle_j
   \otimes K_T^{p-i} (k(x'_r)/K) 
   \cong &
  \bigwedge^i \Q \langle a'_j \rangle_j
   \otimes K_T^{p-i} (k(x'_r)/K) \\
   \cong  &
   \bigoplus_{  J=\{ j_1 , \dots, j_r \} \subset [r] }
   \Q \langle a'_{j_1} \wedge \dots \wedge a'_{j_{ i }} \rangle
   \otimes K_T^{p-i} (k(x'_r)/K) .
  \end{align*}

  Similarly, the first vertical arrow
  in (\ref{diagram for cores 2})
  is the restriction of the morphism
  $$\bigoplus_{ y \in \pi^{-1}(x) } K_T^p (k(y)/K)  
  	\to 
  	\bigoplus_{ y'_r \in \pi^{-1}(x_r') } 
  \bigoplus_i 
  \bigwedge^i \Q \langle a_j \rangle_j 
  \otimes
  	K_T^{p-i} (k(y'_r)/K)  $$
  defined as the sum of morphisms
  $$\bigoplus_{ y \in \pi^{-1}(x) } K_T^p (k(y)/K)  
  	\to 
   \Q \langle a'_{j_1} \wedge \dots \wedge a'_{j_{ i }} \rangle
  \otimes
  	K_T^{p-i} (k(y'_r)/K)  $$
  ($i$, $J=\{j_1 , \dots, j_{i } \} \subset [r]$, and $y'_r \in \pi^{-1}(x_r')$)
  given by 
  $$f =( f_y)_{y \in \pi^{-1}(x)} \mapsto 
  \epsilon_J
  a'_{j_1} \wedge \dots \wedge a'_{j_{ i }} 
  \otimes
  \sum_{ \substack{ (y'_0, \dots,y'_{r-1}) \\  \overline{y'_{r-1}} \ni y'_r} }
  \frac{
  \partial_{y'_{r-1}}^{y'_{r}}
  ( \dots (\partial_{y'_{0}}^{y'_{1}}
  ( \bigwedge_{i \in [r] \setminus J} a'_i \wedge f_{y'_0})) \dots )
  }{  
  \partial_{y'_{r -1}}^{y'_{r}}
  ( \dots (\partial_{y'_0}^{y'_1}
  (a'_1 \wedge \dots \wedge a'_{r})) \dots ) 
  },
  $$
  where 
  $(y'_0,\dots, y'_{r-1})$
  runs through $y'_i \in \pi^{-1}(x'_i)$ ($0\leq i \leq r-1$)
  such that 
  $$\overline{ \{y'_0 \} } \supset 
  \overline{\{ y'_1 \} } \supset
  \dots \supset 
  \overline{ \{ y'_{r-1} \} } \ni  y'_r.$$

  By \cite[Proposition 4.6 (2)]{Ros96},
  we have 
  \begin{align*}
  \partial_{x'_{r -1}}^{x'_{r}}
  ( \dots (\partial_{x'_0}^{x'_1}
  (a'_1 \wedge \dots \wedge a'_{r})) \dots ) 
  = 
  \partial_{y'_{r -1}}^{y'_{r}}
  ( \dots (\partial_{y'_0}^{y'_1}
  (a'_1 \wedge \dots \wedge a'_{r})) \dots ). 
  \end{align*}
  Consequently, 
  by compatibility of norm homomophisms and residue homomophisms
  (\cite[Proposition 4.6 (1)]{Ros96}),
  diagram (\ref{diagram for cores 2}) 
  is commutative.
  \end{proof}
  
  We put 
  \begin{align*}
  K_T^{p-i} (k(y'_r)/K)|_{\O_{X'_L, \pi^{-1}(x') }^{\times}} := & \bigcap_{ \substack{ y' \in \pi^{-1}(x') \cap  \overline{ \{y'_r \}} } }
    \Im \big( \bigwedge^{p-i} 
    ( (\O_{X'_L, y'}^{\times})_\Q ) \to K_T^{p-i} (k(y'_r)/K)  
      \big) ,   \\
  K_T^{p-i} (k(x'_r)/K)|_{\O_{X', x'}^{\times}} := &
  \Im \big( \bigwedge^{p-i}  
      ( (\O_{X', x'}^{\times})_\Q ) 
    \to 
    K_T^{p-i} (k(x'_r)/K  ) \big) .
  \end{align*}

  \begin{lem}
  We have a natural commutative diagram 
  \begin{align*}
  \xymatrix{
    \bigoplus_{ y'_r \in \pi^{-1}(x_r') } 
    K_T^{p-i} (k(y'_r)/K)  |_{\O_{X'_L, \pi^{-1}(x')}^{\times}}
    \ar[r] \ar[d] &
    K_T^{p-i} (k(x'_r)/K) |_{\O_{X', x'}^{\times}}  \ar[d] \\
  \bigoplus_{w \in \pi^{-1} (v)}
    K_T^{p-i} (\kappa (w) /K)  
    \ar[r] &
    K_T^{p-i} (\kappa(v)/K)  ,
  }
  \end{align*}
  where  
  the horizontal arrows are the sum of norm homomophisms.
  \end{lem}
  \begin{proof}
  Similar to Lemma \ref{compatibility of norm hom and certain specializations}, 
  we have a commutative diagram
  \begin{align*}
  \xymatrix{
    \bigoplus_{ y'_r \in \pi^{-1}(x_r') } 
    K_T^{p-i} (k(y'_r)/K)  |_{\O_{X'_L,  \pi^{-1}(x')}^{\times}}
    \ar[r] \ar[d] &
    K_T^{p-i} (k(x'_r)/K) |_{\O_{X', x'}^{\times}}  \ar[d] \\
  \bigoplus_{y' \in \pi^{-1} (x')}
    K_T^{p-i} (\kappa (y') /K)  
    \ar[r] &
    K_T^{p-i} (\kappa(x')/K)  .
  }
  \end{align*}
  The assertion follows from this diagram and the compatibility of norm homomophisms and extensions of fields.
  \end{proof}

\section{``Monodromy weight'' spectral sequences for geometric strictly semi-stable reductions}\label{sec ssreduction}
  In this section, we shall give spectral sequences (Corollary \ref{spectral sequence for ss reduction}) for geometric strictly semi-stable reductions. 
  For this, we also introduce log tropical cohomology (Subsection \ref{subsec log tropical cohomology}).
  Our theory is parallel to monodromy weight spectral sequences for singular cohomology of degenerations of complex algebaric varieties, see e.g., \cite[Chapter 4 and 11]{PS08}.
   
 \subsection{Log tropical cohomology}\label{subsec log tropical cohomology}
  In this subsection, we introduce and study log tropical cohomology in the trivially valued case.
  Let $X$ be a smooth algebraic variety over a trivially valued field $K$, 
  and
  $D=\bigcup_{i\in I} D_i$ a simple normal crossing divisor in $X$.
  We put $U := X\setminus D$ the complement and $i\colon U \hookrightarrow X$ the inclusion.
  
  First, we shall study cohomology of sheaves of log tropical Milnor $K$-groups.
  We have a distinguished triangle 
  $$ R \Gamma_D (\K^p_{T,X}) \to \K^p_{T,X} \to Ri_* \K^p_{T,U} \to^{[1]} \cdot $$
  in the derived category of sheaves  
  on $X_{\Zar}$.
  In particular, we have 
  \begin{align*}
  	R^0 \Gamma_D (\K^p_{T,X}) & = 0, \\
  	R^1 \Gamma_D (\K^p_{T,X}) & \cong i_* \K^p_{T,U} /  \K^p_{T,X},  \\
  	R^i \Gamma_D (\K^p_{T,X}) & \cong R^{i-1} i_* \K^p_{T,U}  \quad (i \geq 2).
  \end{align*}
  
  The following is same as the case of Milnor $K$-groups (\cite[Lemma 2.1]{RS18}).
  \begin{lem}\label{vanish of R^i Gamma_D (K^p_T,X)}
     $R^i \Gamma_D (\K^p_{T,X})=0$ for $i \geq 2$.
  \end{lem}
  \begin{proof}
  	When $\# I=1$, i.e., $D $ is irreducible, 
  	we have 
  	$$R \Gamma_{D} (\K^p_{T,X}) \cong
        \Gamma_{D} C_{X}^{p,*} \cong 
  	  j_* C_{D}^{p-1,* }[-1],$$
  	where $C_X^{p,*} $ is the Gersten resolution of $\K_{T,X}^p$ and $j\colon D \hookrightarrow X$ is the inclusion.
    We assume that $\# I \geq 2$. 
  	We fix $r \in I$. We put $D':= \bigcup_{i \in I \setminus \{r\}} D_i .$ 
  	We have a distinguished triangle 
  	$$R \Gamma_{D'} (\K^p_{T,X}) \to 
  	R \Gamma_{D} (\K^p_{T,X})  \to 
  	R \Gamma_{D\setminus D'} (\K^p_{T,X})  \to^{[1]} \cdot. $$
  	We put 
  		$$ D_r \setminus D' \to^{j_r}  D_r \to^{i_r}  X .$$
  	Then we have 
  	\begin{align*}
  		R \Gamma_{D\setminus D'} (\K^p_{T,X}) 
  	= &	R \Gamma_{D_r \setminus D'} (\K^p_{T,X}) \\
  	\cong &	R i_{r,*} R j_{r,*} j_r^{*} i_r^{!} (\K^p_{T,X}) \\
  	\cong &	R i_{r,*} R j_{r,*}  (\K^{p-1}_{T,D_r \setminus D' })[-1]. 
  	\end{align*}
  	Hence the assertion follows by induction on $\# I$.
  \end{proof}
  
  \begin{cor}\label{log trop Mil K coh = trop Mil K}
  	$$H^q(U_{\Zar}, \K_{T,U}^p) \cong H^q (X_{\Zar}, i_* \K_{T,U}^p).$$
  \end{cor}
  We put $\K_{T,X}^p (\log D):= i_* \K_{T,U}^p$, the sheaf of $p$-th \emph{log tropical Milnor $K$-groups}.
  
  Let $J\subset I$ be a subset such that $\bigcap_{j\in J} D_j \neq \emptyset$.
  We fix an order $J= \{j_1,\dots,j_{\# J}\} $.
  We have a composition of morphisms 
  \begin{align*}
   i_*  \K_{T,U}^p  &
  \to R^1 \Gamma_D (\K^p_{T,X}) 
  \to R^1 \Gamma_{D_{j_1} \setminus  (\bigcup_{i \neq j_1 }D_i )  } (\K^p_{T,X}) 
  \to 
  h_{j_1,* } \K^{p-1}_{T, D_{j_1} \setminus (\bigcup_{i \neq j_1 }D_i )   } \\
  & \to 
  h_{\{j_1,j_2\} ,* } \K^{p-2}_{T, (D_{j_1} \cap D_{j_2}) \setminus  (\bigcup_{i \in I \setminus \{ j_1,j_2\} }D_i )   } 
  \to \dots \to 
  h_{J ,* } \K^{p- \# J}_{T, \bigcap_{j \in J} D_j  \setminus  (\bigcup_{i \in I \setminus J }D_i )} ,
  \end{align*}
  where 
  $$h_{\{j_1, \dots, j_s\} } \colon  
  \bigg(\bigcap_{1\leq i \leq s} D_{j_i} \bigg) 
  \setminus 
   \bigg(\bigcup_{i \in I \setminus  \{ j_1 , \dots, j_s \} }D_i \bigg) 
  \to X$$ 
  ($1 \leq s \leq  \# J$)
  is the inclusion.
  We put 
  \begin{align*}
   W_r \K_{T,X}^p (\log D)
  :=  \Ker \bigg( 
  \K_{T,X}^p (\log D)  
  \to 
  \bigoplus_{\substack{J \subset I \\ \# J= r+1}}
  h_{J ,* } \K^{p-r-1}_{T, \bigcap_{j \in J} D_j  \setminus  (\bigcup_{i \in I \setminus J }D_i )} ,
  \bigg).
  \end{align*}
  This is an increasing filtration, called the \emph{weight filtration}.
  For example, 
  $W_0 \K_{T,X}^p (\log D) = 
   \K_{T,X}^p$.
  We put 
  $$ 
   \Gr^W_r \K_{T,X}^p (\log D)
  := W_r \K_{T,X}^p (\log D) / W_{r-1} \K_{T,X}^p (\log D).$$
  
  For each $J\subset I $, we put $D_J := \bigcap_{j \in J} D_j$ and $j_J \colon D_J \hookrightarrow X$.
  
  \begin{lem}\label{gr W K (log D)}
  	For $r \geq 0$, 
  	we have 
   $$\Gr^W_r \K_{T,X}^p (\log D)
  \cong \bigoplus_{\substack{J \subset I \\ \# J= r}}
  j_{J ,* } \K_{T, D_J  }^{p-r} .$$
  \end{lem}
  \begin{proof}
  Obviously, we have an injection 
  $$\Gr^W_r \K_{T,X}^p (\log D)
  \hookrightarrow \bigoplus_{\substack{J \subset I \\ \# J= r}}
  j_{J ,* } \K_{T, D_J  }^{p-r} .$$
  We shall show surjectivity.
  Let $x \in X$, and $I_x := \{ i \in I \mid x \in D_i \}$.
  Let 
  $$(a_J)_J \in 
  \bigg(\bigoplus_{\substack{J \subset I \\ \# J= r}}
  j_{J ,* } \K_{T, D_J  }^{p-r} \bigg)_x
  = 
  \bigg(\bigoplus_{\substack{J \subset I_x \\ \# J= r}}
  j_{J ,* } \K_{T, D_J  }^{p-r} \bigg)_x $$
  be an element.
  For $L \subset I_x$ with $D_L \neq \emptyset$, 
  we put 
  $\eta_{D_{L}, x}$ the generic point of the irreducible component of $D_{L}$ containing $x$. 
  For each $J \subset I_x$ with $\# J =r$, 
  let $C_{D_{J \setminus \{j_r \}}}^{p-r+1,*}$ be 
  the Gersten resolution (Corollary \ref{corshftrK}) of $\K_{T, D_{J \setminus \{j_r \}}  }^{p-r+1}$.
  Then there exists 
  $$a_{J,J\setminus \{ j_r \}} \in
   K_T^{p-r+1}  (k(\eta_{D_{J \setminus \{j_r \}}, x})/K)
   = C_{D_{J \setminus \{j_r \}} , x}^{p-r+1,0}$$
  such that 
  $$d a_{J,J\setminus \{ j_r \}} = a_J \in C_{D_{J \setminus \{j_r \}} ,x }^{p-r+1,1},$$
  where 
  we consider $a_J $ as an element in $ C_{D_{J \setminus \{j_r \}} ,x }^{p-r+1,1}$
  via natural inclusions 
  $$ (j_{J ,* } \K_{T, D_J  }^{p-r})_x 
     \subset K_T^{p-r}  (k(\eta_{D_{J \setminus \{j_r \}}, x})/K)
     \subset
     C_{D_{J \setminus \{j_r \}} ,x }^{p-r+1,1}.$$

  Similarly, 
  since 
  $$a_{J, J\setminus \{ j_r \}}
   \in h_{J,r-1,*} C_{D_{J \setminus \{j_{r-1}, j_r \}}  \setminus D_{j_r}, x}^{p-r+2,1}$$
   satisfies 
  $$d a_{J, J\setminus \{ j_r \}}=0
   \in h_{J,r-1,*} C_{D_{J \setminus \{j_{r-1}, j_r \}}  \setminus D_{j_r}, x}^{p-r+2,2}, $$
  where we put 
  $$ h_{J,r-1} \colon D_{J \setminus \{j_{r-1}, j_r \}}  \setminus D_{j_r} 
  \hookrightarrow D_{J \setminus \{j_{r-1}, j_r \}} ,$$
  by the log-version of the Gersten resolution (Lemma \ref{vanish of R^i Gamma_D (K^p_T,X)}),
  there exists 
  $$a_{J, J\setminus \{j_{r-1}, j_r \}}
  \in
   K_T^{p-r+2}  (k(\eta_{D_{J \setminus \{j_{r-1}, j_r \}}, x})/K)
   =h_{J,r-1,*} C_{D_{J \setminus \{j_{r-1}, j_r \}}  \setminus D_{j_r}, x}^{p-r+2,0}$$
  such that 
  $$d a_{J, J\setminus \{ j_{r-1},j_r \}} = a_{J, J\setminus \{ j_r \}} 
   \in h_{J,r-1,*} C_{D_{J \setminus \{j_{r-1}, j_r \}}  \setminus D_{j_r}, x}^{p-r+2,1}.$$
  By repeating this argument, 
  we get $a_{J, \emptyset} \in \K_{T, X ,x }^p $
  such that 
  $$ W_r \K_{T,X}^p (\log D) \ni \sum_{\substack{J \subset I_x  \\ \# J =r}} 
  a_{J,\emptyset }  
  \mapsto 
  (a_J)_J \in 
  (\bigoplus_{\substack{J \subset I_x \\ \# J= r}}
  j_{J ,* } \K_{T, D_J  }^{p-r} )_x   .
  $$
  \end{proof}

  \begin{cor}\label{weight spectral sequence for smooth open varieties}
  	For $r \geq 0$, 
  	the $E_1$-terms of the spectral sequence 
  	$$E_1^{p,q}= H^{p+q} ( \Gr^W_{-p} \K_{T,X}^r (\log D) )  \Rightarrow H^{r,p+q}_{\Trop } (U)$$
      induced by the weight filtration are isomorphic to
  	$$ \bigoplus_{\substack{J \subset I \\ \# J= -p}}
          H^{r+p, p+q } ( D_J ) .$$
  \end{cor}

  Next, we shall study log tropical cohomology.
   We put $\F^p_X (\log D)$ the subsheaf of $i_* \F^p_U$ on $X^{\Ber}$ whose stalk at each point $x\in X^{\Ber}$ is the image of  
  $ \bigwedge^p (\O_{X, \supp (x)}[f_i^{-1}]_{i \in I_x})^{\times }_\Q$ under the natural morphism,
  where 
  $f_i \in \O_{X,\supp (x)}$ is a local equation of $D_i$, 
   and $I_x := \{ i \in I  \mid x \in D_i^{\Ber} \}$.
   We call $ H^q (X^{\Ber}, \F^p_X (\log D)) $ \emph{log tropical cohomology} of $(X,D)$.
  By definition, 
  we have 
  \begin{align}
  (\F^p_X (\log D))_x 
  \cong \bigoplus_{ i=0}^p 
  \bigwedge^i \Q \langle f_i \rangle_{i \in I_x}  
  \otimes  \F^{p-i}_x  \label{log Fp local description} .
  \end{align}

  Let $J:= \{j_1,\dots,j_{\# J}\} \subset I$ be a subset such that $\bigcap_{j\in J} D_j \neq \emptyset$.
  We put $D_J \cap D := D_J \cap \big(\bigcup_{i \in I \setminus J} D_i \big)$ a simple normal crossing divisor in $D_J$.
  By (\ref{log Fp local description}), 
  for $ i_0 \in I$, 
  the residue homomorphism
  $$ \partial_\eta^{\eta_{D_i}} \colon K^p_T (K(X)/K) \to K^{p-1}_T (K(D_i)/K)$$
  (where $\eta \in X$ and $\eta_{D_i} \in D_i$ are the generic point)
  induces a morphism 
  $$\F^p_X (\log D) \to j_{\{i_0\},*} \F^{p-1}_{D_{i_0}} (\log (D_{i_0} \cap  D ) ) .$$
  We get a composition of morphisms 
  \begin{align*}
  	\F^p_X (\log D)&  \to j_{\{j_1\},*} \F^{p-1}_{D_{j_1}} (\log (D_{j_1} \cap D ) )  
  	\\
  & \to 
  j_{\{j_1,j_2\},*} \F^{p-1}_{D_{\{j_1,j_2\}}} (\log (D_{\{j_1,j_2\}} \cap D ) )
  \\ 
  & \to \dots \to 
  j_{J,*} \F^{p-r}_{D_J} (\log (D_J \cap D ) ).
  \end{align*}
  We put 
  \begin{align*}
   W_r \F^p_X (\log D) 
  :=  \Ker \bigg( 
  \F^p_X (\log D)  
  \to 
  \bigoplus_{\substack{J \subset I \\ \# J= r+1}}
  j_{J,*} \F^{p-r-1}_{D_J} (\log (D_J \cap D ) ),
  \bigg).
  \end{align*}
  This is an increasing filtration, called the \emph{weight filtration}.
  For example, 
  $W_0 \F^p_X (\log D) = 
   \F^p_X  $.
  We put 
  $$ 
   \Gr^W_r \F^p_X (\log D) 
  := W_r \F^p_X (\log D)  / W_{r-1} \F^p_X (\log D) .$$

  \begin{lem}\label{gr W F (log D)}
  	For $r \geq 0$, 
  	we have 
   $$\Gr^W_r \F^p_X (\log D) 
  \cong \bigoplus_{\substack{J \subset I \\ \# J= r}}
  j_{J ,* } \F_{D_J  }^{p-r} .$$
  \end{lem}
  \begin{proof}  
  	  We obviously have an injection from the left-hand side to the right-hand side.
  	  We shall show surjectivity. 
  Let 
  $$(a_J)_J \in 
  \bigg(\bigoplus_{\substack{J \subset I \\ \# J= r}}
  j_{J ,* } \F_{D_J  }^{p-r} \bigg)_x
  = 
  \bigg(\bigoplus_{\substack{J \subset I_x \\ \# J= r}}
  j_{J ,* } \F_{ D_J  }^{p-r} \bigg)_x $$
  be an element.
  For each $J \subset I_x$ with $\# J =r$, 
  the element $a_J$ is the image of some $\tilde{a}_J \in \bigwedge^{p-r} (\O_{X,\supp (x)}^{\times})_\Q $.
  Then 
  $$\sum_{\substack{J \subset I_x \\ \# J= r}}
  f_{j_1} \wedge \dots \wedge f_{j_r} \wedge \tilde{a}_J
  \in (W_r \F^p_X (\log D))_x$$
  maps to $(a_J)_J$.
  \end{proof}

  \begin{cor}\label{log Fp coh = log trop Mil K}
  	We have a natural isomorphism
      $$ H^q (X^{\Ber}, \F^p_X (\log D)) \cong  H^q (X_{\Zar}, \K_{T,X}^p (\log D)).$$
  \end{cor}
  \begin{proof}
  By Theorem \ref{thm,trop,main} (i.e., the case of $D = \emptyset $), 
 	Lemma \ref{gr W K (log D)}, and Lemma \ref{gr W F (log D)}, 
 	we have 
 	$R^i \pi_* \F^p_X (\log D) =0 $ for $i \geq 1$, 
 	and  a natural morphism $ \pi_*  \F^p_X (\log D) \to \K_{T,X}^p (\log D) $
  compatible with the weight filtrations is an isomorphism,
 	where $\pi \colon X^{\Ber} \to X$ is the map taking supports.
  \end{proof}

  By Theorem \ref{thm,trop,main},
   Corollary \ref{log trop Mil K coh = trop Mil K}, and 
  Corollary \ref{log Fp coh = log trop Mil K}, 
  we have the following
  \begin{cor}\label{coh log Fp = coh FpU}
  $$ H^{q} ( U^{\Ber}, \F^p_U ) 
   \cong H^{q} ( X^{\Ber}, \F^p_X (\log D) ).$$
  \end{cor}
  
 \subsection{``Monodromy weight'' spectral sequences for geometric strictly semi-stable reductions}
  In this subsection, we shall give ``monodromy weight'' spectral sequences for geometric strictly semi-stable reductions (Corollary \ref{spectral sequence for ss reduction}). 
  
  Let 
  $\pi \colon X \to C$ be a flat, generically smooth, projective morphism from a smooth algebraic variety $X$ to a smooth algebraic curve $C$ over a trivially valued field $K$.
  Let $c \in C$ be a closed point.
  We assume that $X_c:= \pi^{-1}(c)$ is a simple normal crossing divisor.
  We put $\hat{K}_c$ the fraction field of the formal completion $\hat{\O_{C,c}}$ of the local ring $\O_{C,c}$,
  and 
  $X_{\hat{K}_c} := X \times_C \Spec \hat{K}_c$
  the base change.
  The field $\hat{K}_c$ is equipped with a natural normalized discrete valuation $\hat{K}_c^{\times} \twoheadrightarrow \Z$ (taking orders of zeros and poles at $c$).
  We will give a spectral sequence converging to 
  $H_{\Trop}^{p,q} (X_{\hat{K}_c})$. 
  
  We shall show that $H_{\Trop}^{p,q} (X_{\hat{K}_c})$ is isomorphic to cohomology of a relative log version of $\F^p$ on $X_c^{\Ber}$ (Proposition \ref{ss reduction trop coh = log coh}).
  We may assume that $c \in C$ is the zero set of some $t_c \in \O(C)$.
  We put 
  $X^{\Ber}_{0 \leq t_c < 1}$  (resp. $X^{\Ber}_{0 < t_c < 1}$)
  the subset of $X^{\Ber}$ consisting of 
  valuations $v$ with $ 0 < v(t_c) \leq \infty$ (resp. $ 0 < v(t_c) < \infty$).
  Note that for $0 < R < \infty$,
  there is a unique point $v_{c,R} \in C^{\Ber}$ such that $v(t_c) =R$,
  and 
  we have 
  a natural homeomorphism $X_{\hat{K}_c}^{\Ber} \cong \pi^{-1}( v_{c,R} )$.
  We put $U:= X \setminus X_c$.
  Then by Mayer-Vietrious for $ X^{\Ber} =U^{\Ber} \cup X^{\Ber}_{0 \leq t_c <1}$ and $\F^p_X (\log X_c)$, 
  we have 
  \begin{align*}
  \dots \to H^{q} ( X^{\Ber}_{0 < t_c < 1}, \F^p_X ) 
  & \to H^{q} ( X^{\Ber}_{0 \leq t_c < 1}, \F^p_X (\log X_c ))
  \oplus H^{q} ( U^{\Ber}, \F^p_U ) \\
  & \to H^{q} ( X^{\Ber}, \F^p_X (\log X_c) ) 
  \to H^{q+1} ( X^{\Ber}_{0 < t_c < 1}, \F^p_X ) \to \dots .
  \end{align*}
  Hence by Corollary \ref{coh log Fp = coh FpU},
  we have 
  \begin{align}
   H^{q} ( X^{\Ber}_{0 < t_c < 1}, \F^p_X ) 
   \cong H^{q} ( X^{\Ber}_{0 \leq t_c < 1}, \F^p_X (\log X_c )). \label{eq tro coh relative log1}
  \end{align}
  
  We put 
  $$ \F^1_C (\log c) \wedge \F^{p-1}_X (\log X_c)$$
  the subsheaf of $\F^p_X (\log X_c)$
  generated by the image of 
  $ \pi^{*} \F^1_C (\log c) \otimes \F^{p-1}_X (\log X_c)$.
  
  \begin{lem}\label{Fp log computation}
  	For $p \geq 0$, there is an exact sequenece 
  	$$0 \to \F^1_C (\log c) \wedge \F^{p-1}_X (\log X_c)
  	   \to   \F^p_X (\log X_c)
  	   \to \F^1_C (\log c) \wedge \F^{p}_X (\log X_c) 
  	   \to 0$$
  	on $X^{\Ber}_{0 \leq t_c < 1}$, 
  	where the third morphism is given by $ t_c \wedge -$,
    and 
  	we put $\F^1_C (\log c) \wedge \F^{-1}_X (\log X_c) := 0 $.
  \end{lem}
  \begin{proof}
      The morphism $\pi \colon X^{\Ber}_{0 < t_c <1} \to C^{\Ber}$  induces a homeomorphism 
  	$X^{\Ber}_{0 < t_c <1} \cong \pi^{-1}(v_{c,R}) \times (0,1)$ ($ 0 < R < \infty$).
  	Hence the exactness of the sequence on $X^{\Ber}_{0 < t_c <1} $ holds.
  	The exactness on  $X^{\Ber}_{c}$ follows from (\ref{log Fp local description}).
  \end{proof}
  
  As a corollary of (\ref{eq tro coh relative log1}) 
  and Lemma \ref{Fp log computation},
  by induction, we have 
  \begin{align}
   H^{q} ( X^{\Ber}_{0 < t_c < 1}, \F^1_C \wedge \F^p_{X} ) 
   \cong H^{q} ( X^{\Ber}_{0 \leq t_c < 1}, \F^1_C (\log c) \wedge \F^{p}_X (\log X_c) ). \label{eq tro coh relative log c wedge log2}
  \end{align}
  We put 
   $ \F^p_{X/C} :=  \F^p_{X}/  \F^1_C \wedge \F^{p-1}_{X}$,
   a sheaf on $X^{\Ber}_{0 < t_c < 1}$,
   and  put 
   $$\F^p_{X/C} (\log X_c ) := \F^p_X (\log X_c) / \F^1_C (\log c) \wedge \F^{p-1}_X (\log X_c),$$
   a sheaf on $X^{\Ber}_{0 \leq t_c < 1}$.
   Recall that we have a homeomorphism
  $$X^{\Ber}_{0 < t_c <1} \cong 
  \pi^{-1}(v_{c,R}) \times (0,1) $$
  ($ 0 < R < \infty$).
  \begin{lem}\label{tropical charts defined over dense fields are enough}
  By an identification $ \pi^{-1}(v_{c,R})
  \cong  X_{\hat{K}_c}^{\Ber} $
  ($ 0 < R < \infty$), 
  we have
  $$ \F^p_{X/C}|_{\pi^{-1}(v_{c,R})} \cong \F^p_{X_{\hat{K}_c}}.$$
  \end{lem}
  \begin{proof}
  It is enough to show that every section of 
  $ \F^p_{X_{\hat{K}_c}}$ 
  is given by tropical charts of $X_{\hat{K}_c}$ defined over $K(C)$.
  Over a non-trivially valued field $\hat{K}_c$, 
  invertible algebraic functions (instead of general algebraic functions) over $\hat{K}_c$  give sufficiently many  
  tropical charts 
  for $ \F^p_{X_{\hat{K}_c}}$ (\cite{Gub16}, \cite{Jel16}).
  However, for each point $x \in X_{\hat{K}_c}^{\Ber}$ and an algebraic function $f $ invertible at $x$, 
  by ultrametric inequality, 
  there exists an algebraic function $g$ defined over $K(C)$ 
  such that $\Trop \circ f = \Trop \circ g$ on an open neighborhood of $x$.
  Hence the assertion holds.
  \end{proof}

  By (\ref{eq tro coh relative log1}), 
  (\ref{eq tro coh relative log c wedge log2}), and
  Lemma \ref{tropical charts defined over dense fields are enough},
  we have 
  $$ H^{q} ( X^{\Ber}_{\hat{K}_c}, \F^p_{X_{\hat{K}_c}}) 
   \cong H^{q} ( X^{\Ber}_{0 \leq t_c < 1},\F^p_{X/C} (\log X_c ) ). $$
  Similarly to $\F^p$ and $\C_T^{p,*}$, we have a c-soft resolution $\C^{p,*}_{T,X/C} (\log X_c)$ of $\F^p_{X/C} (\log X_c ) $ 
  which are given by tropical charts (given by closed immersions of $X$ to toric varieties) and cochains with ``$\F^p_{X/C} (\log X_c ) $''-coefficients.
  By using retractions of tropicalizations (like Lemma \ref{lem;X;circ;trocoho}), we have 
  $$H^{q} ( X^{\Ber}_{0 \leq t_c < 1},\F^p_{X/C} (\log X_c ) )
  \cong H^{q} ( X^{\Ber}_{c},\F^p_{X/C} (\log X_c )|_{X_c^{\Ber}} ).$$
  Consequently, we have the following.
  \begin{prp}\label{ss reduction trop coh = log coh}
  $$	H^{q} ( X^{\Ber}_{\hat{K}_c}, \F^p_{X_{\hat{K}_c}}) 
  \cong H^{q} ( X^{\Ber}_{c},\F^p_{X/C} (\log X_c )|_{X_c^{\Ber}} ).$$
  \end{prp}
  
  We shall construct a spectral sequence converging to $H^{q} ( X^{\Ber}_{c},\F^p_{X/C} (\log X_c )|_{X_c^{\Ber}} )$, the \emph{``monodromy weight'' spectral sequence}.
  We put 
  $$A^{p,*} := (\F^{p+*+1}_{X} (\log X_c )/ W_* \F^{p+*+1}_{X} (\log X_c )) |_{X_c^{\Ber}}$$
  ($* \geq 0$)
  and 
  $$d \colon A^{p,q } \ni \alpha \mapsto t_c \wedge \alpha \in A^{p,q+1}.$$
  By (\ref{log Fp local description}), 
  a morphism 
  $$\F^p_{X/C} (\log X_c ) |_{X_c^{\Ber}} \ni \alpha 
  \mapsto (-1)^p t_c \wedge \alpha 
   \in A^{p,0}$$
  induces a quasi-isomorphism 
  $$\F^p_{X/C} (\log X_c )|_{X_c^{\Ber}} \cong A^{p,*}.$$
  We put 
  $ W(M)_r A^{p,q} $ the image of $W_{r+2q+1} \F^{p+q+1}_{X} (\log X_c ) |_{X_c^{\Ber}}  $ to $A^{p,q}$.
  Then we have $d W(M)_r \subset W(M)_{r-1}$.
  We put $X_{c,i} $ ($i \in I$) the irreducible components of $X_c$, 
  and for $J \subset I$, we put $X_{c,J} :=\bigcap_{i \in J} X_{c,i}$.
  By Lemma \ref{gr W F (log D)},
  we have 
  \begin{align*}
   \Gr^{W(M)}_r A^{p,*} 
   \cong & 	
    \bigoplus_{\substack{ 0 \leq u \\ u+1 \leq r + 2u+1 } }
    \Gr^W_{r +2u +1} \F_X^{p+u+1} (\log X_c) |_{X^{\Ber}_c} [-u] \\
   \cong & 	
    \bigoplus_{ \max \{0, -r \} \leq u  }
    \bigoplus_{\substack{ J \subset I \\ \# J = r + 2u+1 } }
     j_{J,*} \F_{X_{c,J}}^{p-r-u}|_{X^{\Ber}_c}  [-u]. \\
  \end{align*}

  \begin{cor}\label{spectral sequence for ss reduction}
  	For $r \geq 0$, 
  	the $E_1$-terms of the spectral sequence 
  	$$E_1^{p,q}= H^{p+q} (X_c^{\Ber}, \Gr^{W(M)}_{-p} A^{r,*}  )  \Rightarrow H^{r,p+q}_{\Trop } (X_{\hat{K}_c}^{\Ber})$$
      induced by the ``monodromy weight'' filtration $(W(M)_s A^{r,*})_{s } $ are isomorphic to
    $$\bigoplus_{ \max \{0,p\} \leq u  }
    \bigoplus_{\substack{ J \subset I \\ \# J = -p + 2u+1 } }
  H^{r+p-u, p+q -u}_{\Trop} (X_{c,J}). $$
  \end{cor}
  
  The rest of this subsection is 
  devoted to show that the spectral sequence in Corollary \ref{spectral sequence for ss reduction}
  is compatible with
   the monodromy weight spectral sequence for singular cohomology when $K=\mathbb{C}$ (\cite[p.269 and Corollary 11.23]{PS08}).

   We shall give an expression of the spectral sequence in Corollary \ref{spectral sequence for ss reduction} using Zariski sheaves.
  We put 
  $$A_{K_T}^{p,*} := (\K_{T,X}^{p+*+1} (\log X_c )/ W_* \K_{T,X}^{p+*+1} (\log X_c )) |_{X_c}$$
  ($* \geq 0$)
  and 
  $$d \colon A_{K_T}^{p,q } \ni \alpha \mapsto t_c \wedge \alpha \in A_{K_T}^{p,q+1},$$
  a complex of Zariski sheaves on $X_c$.
  We put 
  $ W(M)_r A_{K_T}^{p,q} $ the image of the sheaf 
  $W_{r+2q+1} \K^{p+q+1}_{T,X} (\log X_c ) |_{X_c}  $ to $A_{K_T}^{p,q}$.
  Then we have $d W(M)_r \subset W(M)_{r-1}$.
  By Lemma \ref{gr W K (log D)}, 
  we have 
  \begin{align*}
   \Gr^{W(M)}_r A_{K_T}^{p,*} 
   \cong & 	
    \bigoplus_{\substack{ 0 \leq u \\ u+1 \leq r + 2u+1 } }
    \Gr^W_{r +2u +1} \K_{T,X}^{p+u+1} (\log X_c) |_{X_c} [-u] \\
   \cong & 	
    \bigoplus_{ \max \{0, -r\}  \leq u }
    \bigoplus_{\substack{ J \subset I \\ \# J = r + 2u+1 } }
     j_{J,*} \K_{T,X_{c,J}}^{p-r-u}  |_{X_c}[-u]. \\
  \end{align*}
  Let $\Psi \colon X_c^{\Ber} \to X_c$ be the map of taking supports.
  By Theorem \ref{thm,trop,main},
  $A^{p,*}$ is a $\Psi_*$-injective, 
  and 
  we have 
  $$R\Psi_* (A^{p,*}) 
  \cong \Psi_* (A^{p,*}) 
  \cong  A_{K_T}^{p,*} .$$
  The last isomorphism is compatible with the ``monodromy weight'' filtrations. 
  Hence the spectral sequence in Corollary \ref{spectral sequence for ss reduction}
  coincides with 
  the spectral sequence 
  $$E_1^{p,q}= H^{p+q} (X_c, \Gr^{W(M)}_{-p} A_{K_T}^{r,*}  ) 
  \Rightarrow  H^{p+q} (X_c, A_{K_T}^{r,*}) $$
  induced by the ``monodromy weight'' filtration $(W(M)_s A_{K_T}^{r,*} )_s $.
  
  We shall briefly recall monodromy weight filtrations in complex algebraic geometry, see \cite[Section 11.2]{PS08} for details.
  In the following, we assume that $K=\mathbb{C}$.
  Let $\Omega_{X}^* (\log X_c) $ be
  the log holomorphic differential forms  on $X(\CC)$.
  We have a natural increasing filtration $W_* \Omega_{X}^* (\log X_c)$, the weight filtration.
  Then the nearby cycle complex $\psi_\pi \underline{\CC}_{X(\CC)}$ at $c \in C$
  is quasi-isomorphic to the single complex $s(A^{*,*})$ associated to 
  a double complex
  $$A_{\Omega}^{p,q} := (\Omega_{X}^{p+q+1} (\log X_c )/ W_q \Omega_{X}^{p+q+1} (\log X_c )) |_{X_c(\CC)}$$
  ($* \geq 0$)
  with
  \begin{align*}
  d' \colon & A_{K_T}^{p,q } \ni \alpha \mapsto  \frac{ - d (\log t_c)}{2 \pi i} \wedge \alpha \in A_{K_T}^{p,q+1},  \\
  d'' \colon & A_{K_T}^{p,q } \ni \alpha \mapsto d \alpha \in A_{K_T}^{p+1,q}.
  \end{align*}
  We put 
  $ W(M)_r A_{\Omega}^{p,q} $ the image of $W_{r+2q+1} \Omega^{p+q+1}_{X} (\log X_c ) |_{X_c(\CC)}  $ to $A_{\Omega}^{p,q}$, the monodromy weight filtration. 
  
  Recall that 
  by \cite[Proposition 5.5]{M21},
  there exists a natural morphism 
  $$\Phi^*\K_T^p \ni f_1 \wedge \dots \wedge f_p 
  \mapsto  \frac{ - d (\log f_1)  }{2 \pi i  }  \wedge \dots \wedge \frac{ - d  (\log f_p) }{2 \pi i  }  
  \in   \Omega_X^p ,$$ 
  where we put $\Phi \colon X( \CC) \to X$ the natural map.
  It induces a morphism 
  \begin{align}
  H_{\Trop}^{p,q}(X) \to H_{\sing}^{p+q} (X(\CC),\CC) 
  \label{tro coh to sing coh CC}
   \end{align}
  which is compatible with 
  the cycle class map $\CH^p (X) \to  H_{\sing}^{2p} (X(\CC),\CC)$ and 
  the tropical cycle class map 
  $\CH^p(X) \otimes \Q \to H_{\Trop}^{p,p}(X)$ 
  (\cite[Remark 5.6]{M21}).
  This induces  a natural morphism 
  $$\Phi^*\K_{T,X}^p(\log X_c) \to \Omega_X^p (\log X_c),$$
  and hence induces a morphism
  $$\Phi^* A_{K_T}^{p,*}[-p] \to s(A_{\Omega}^{*,*}),$$
  which is compatible with the monodromy weight filtrations.
  Hence we get a morphism 
  from
  $$E_1^{p,q}= H^{p+q} (X_c, \Gr^{W(M)}_{-p} A_{K_T}^{r,*}[-r]  )  
  \Rightarrow  H^{p+q-r} (X_c, A_{K_T}^{r,*}) $$
  to the monodromy weight spectral sequence
  $$E_1^{p,q}= H^{p+q} (X_c (\CC), \Gr^{W(M)}_{-p} s(A_{\Omega}^{*,*})  ) \Rightarrow H^{p+q}(X_{\infty};\mathbb{C}), $$
  where $H^{p+q}(X_{\infty};\mathbb{C})$ is the limit mixed Hodge structure of 
  fibers of $\pi$ at $c \in C$.
  The morphism of $E^1$-terms are given by natural morphisms
  \begin{align*}
   \Gr^{W(M)}_{-p} A_{K_T}^{r,*} [-r]
   \cong & 	
    \bigoplus_{ \max \{ 0, p\} \leq u  }
    \bigoplus_{\substack{ J \subset I \\ \# J = -p + 2u+1 } }
     j_{J,*} \K_{T,X_{c,J}}^{r+p -u} [-r-u] \\
   \to & 	
    \bigoplus_{ \max \{ 0, p\} \leq u  }
    \bigoplus_{\substack{ J \subset I \\ \# J = -p + 2u+1 } }
     j_{J,*} \Phi_* \Omega_{X_{c,J}}^{*} [p-2u] \\
  \cong &
   \Phi_* \Gr^{W(M)}_{-p} s(A_{\Omega}^{*,*}). 
  \end{align*}
  Consequently, we have the following.
  \begin{prp}
  	There is a natural morphism from 
  	the ``monodromy weght'' spectral sequence
  	$$E_1^{p,q}= \bigoplus_r \bigoplus_{ \max \{0,p\} \leq u  }
    \bigoplus_{\substack{ J \subset I \\ \# J = -p + 2u+1 } }
  H^{r+p-u, p+q -u-r}_{\Trop} (X_{c,J})  \Rightarrow  \bigoplus_r H^{r,p+q-r}_{\Trop } (X_{\hat{K}_c}^{\Ber})$$
  to the monodromy weght spectral sequence
  $$E_1^{p,q}=\bigoplus_{ \max \{0,p\} \leq u  }
    \bigoplus_{\substack{ J \subset I \\ \# J = -p + 2u+1 } }
  H^{2p+q -2u}_{\sing} (X_{c,J}(\CC);\Q) (p -u)
    \Rightarrow   H^{p+q}_{\sing } (X_{\infty};\Q),$$
    where $(p-u)$ is the Tate twist.
  Moreover, the morphisms of $E_1$-terms 
  are the direct sums of morphisms
  $$\bigoplus_r H^{r+p-u, p+q -u-r}_{\Trop} (X_{c,J})  \to H^{2p+q -2u}_{\sing} (X_{c,J}(\CC);\Q) $$
  given by natural morphisms
  $\Phi^* \K_{T,X_{c,J}}^{*}[-*] \to  \Omega_{X_{c,J}}^*$.
  \end{prp}
  \begin{rem}
  	In particular, when for $J \subset I$,
  	we have 
  	\begin{align*}
  	H^{p, q }_{\Trop} (X_{c,J} ) & =0  \quad (p \neq q) ,\\
  	H_{\sing}^{2p+1}(X_{c,J}(\mathbb{C} ); \Q  )& =0 \quad  (p \geq 0),  \\
  	 H_{\sing}^{2p}(X_{c,J} (\mathbb{C} ); \Q )  & \cong \CH^p (X_{c,J}) \otimes \Q  \quad 
  	(p \geq 0),
  	\end{align*}
  	we have an isomorphism 
  	$$ H^{p, q }_{\Trop} (X_{\hat{K}_c}) \cong \Gr^W_{2p} H_{\sing}^{p+q}(X_{\infty} ; \Q ).$$
    (These assumptions hold when e.g., $X_{c,J}$ is a smooth projective toric variety, or the wonderful compactification of the complement of a hyperplane arrangement (\cite{CP95}) (see Subsection \ref{subsec example}).)
  	Moreover, since the right-hand side is of $(p,p)$-type, 
    by \cite[Corollary 11.25]{PS08}, 
    we also have 
  	$$ \dim H^{p, q }_{\Trop} (X_{\hat{K}_c}) \cong h^{p,q} (\pi^{-1}(c')),$$
  	for a closed point $c' \in C$ in general position. 
  
  Similar results were previously proved by 
  by Gross-Siebert (\cite{GS10}) for  certain (possibly non-semi-stable) toric degenerations of Calabi-Yau varieties using their cohomology of 
  certain tropical affine manifolds with singularities,
  and 
  by 
  Itenberg-Katzarkov-Mikhalkin-Zharkov (\cite{IKMZ19}) 
  for maximally degenerate smooth projective varieties 
  having smooth (i.e., locally matroidal) tropicalizations
  using tropical cohomology of the tropicalizations.
  \end{rem}
    
\subsection*{Acknowledgements}
 The author would like to express his sincere thanks to his adviser Tetsushi Ito for  his invaluable advice and persistent support.
 The author also thanks Yuji Odaka, Masanori Kobayashi, Hiroshi Iritani, Kazuhiko Yamaki, Sam Payne, Walter Gubler, Klaus K\"{u}nnemann for their comments and encouragement.
 The author also thanks  Kazuhiro Ito for answering the author's questions on references and his helpful comments on this paper.
 The author thanks Philipp Jell for letting the author know his paper.
 This work was supported by JSPS KAKENHI Grant Number 18J21577.
 Parts of this paper was written when the author is at Academia Sinica. 
 He would like to express his sincere gratitude to his mentor Yuan-Pin Lee for his thoughtful encouragement.

\end{document}